\documentclass[dvips]{imsart}

\RequirePackage[OT1]{fontenc}
\usepackage{bbm}
\usepackage[numbers]{natbib}
\usepackage{amsthm,amsmath}
\usepackage{graphics,ifthen}
\usepackage{latexsym}
\usepackage{mathrsfs}
\usepackage{amssymb}
\startlocaldefs

\newtheoremstyle{example}{\topsep}{\topsep}%
     {}%         Body font
     {}%         Indent amount (empty = no indent, \parindent = para indent)
     {\rmfamily}% Thm head font
     {}%        Punctuation after thm head
     {\newline}%     Space after thm head (\newline = linebreak)
     {\thmname{#1}\thmnumber{ #2}\thmnote{ #3}}%         Thm head spec

   \theoremstyle{example}

\newcommand{\indic}{\mathbb{I}}

\numberwithin{equation}{section}
\theoremstyle{plain}
\newtheorem{thm}{Theorem}[section]

\newtheorem{prop}{Proposition}[section]
\newtheorem{lem}{Lemma}[section]
\newtheorem{rem}{Remark}[section]

\newtheorem{cor}{Corollary}[section]

\newcommand{\Lower}[2]{\smash{\lower #1 \hbox{#2}}}
\newcommand{\ben}{\begin{enumerate}}
\newcommand{\een}{\end{enumerate}}
\newcommand{\bi}{\begin{itemize}}
\newcommand{\ei}{\end{itemize}}

\endlocaldefs

\begin{document}

\begin{frontmatter}
\title{Stick-breaking PG$(\alpha,\zeta)$-Generalized Gamma Processes\protect} \runtitle{$\mathrm{PG}(\alpha,\zeta)$-processes}
%\thankstext{T1}{Footnote to the title with the `thankstext' command.}

\begin{aug}
\author{\fnms{Lancelot F.} \snm{James}\thanksref{t1}\ead[label=e1]{lancelot@ust.hk}}

\thankstext{t1}{Supported in
part by the grant RGC-HKUST 601712 of the HKSAR.}

\runauthor{Lancelot F. James}

\affiliation{Hong Kong University of Science and Technology}

\address[a]{Lancelot F. James\\ The Hong Kong University of Science and
Technology, \\Department of Information Systems, Business Statistics and Operations Management,\\
Clear Water Bay, Kowloon, Hong Kong.\\ \printead{e1}.}

\contributor{James, Lancelot F.}{Hong Kong University of Science
and Technology}

\end{aug}

\begin{abstract}
Results are developed for a large family that is derived from Proposition 21 of  Pitman and Yor's~\cite{PY97}
paper on the remarkable two parameter Poisson Dirichlet family, with law denoted as $\mathrm{PD}(\alpha,\theta)$ for $0\leq \alpha<1,$ $\theta>-\alpha.$ Proposition 21 describes representations of  $\mathrm{PD}(\alpha,\theta)$ mass partitions, $(P_{i}),$ for the less general range of $\theta \ge 0,$ in terms of a generalized gamma subordinator whose jumps are restricted to a (random) gamma distributed time.   There is a more general class
dubbed $\mathrm{PG}(\alpha,\zeta),$ formed by replacing the gamma time by an arbitrary non-negative random variable $\zeta,$ which can be seen as a change of measure. When $\zeta$ is fixed this is the important generalized gamma class originally investigated by McCloskey~\cite{McCloskey}.
Here a larger class, $\mathrm{EPG}(\alpha,\zeta),$ is constructed which contains  $\mathrm{PD}(\alpha,\theta)$ for the entire range, among many others, Hence encompassing many of the random processes of this type used in applications as well as new ones.

 Our results center around the development of various properties of the $\mathrm{EPG}(\alpha,\zeta),$ which also gives additional insights into known
sub-classes. As highlights, an explicit and tractable stick-breaking representation, derived from size biased sampling, is obtained for the entire range of parameters
$(\alpha,\zeta).$
This result represents the first explicit case, outside of $\mathrm{PD}(\alpha,\theta)$ that holds for Poisson Kingman models generated by a stable subordinator for all $\alpha\in [0,1).$
Furthermore Markov chains are derived which establish links between Markov chains derived from stick-breaking(insertion/deletion), as described in Perman, Pitman and Yor and
Markov chains derived from successive usage of dual coagulation fragmentation operators described in Bertoin and Goldschmidt and Dong, Goldschmidt and Martin. Which have connections to certain types of fragmentation trees and coalescents. Our results are also suggestive of  new models and tools, for applications in Bayesian Nonparametrics/Machine Learning, where
$\mathrm{PD}(\alpha,\theta)$  bridges are often referred to as  Pitman-Yor
processes..
\end{abstract}



\end{frontmatter}
%\tableofcontents
\section{Introduction}
The paper by~Pitman and Yor~\cite{PY97} consists of an extensive survey and development of additional results for the remarkable two parameter Poisson Dirichlet distribution denoted as $\mathrm{PD}(\alpha,\theta)$, for parameters $0\leq \alpha<1,$ and $\theta>-\alpha.$  This law is specified
for an exchangeable sequence of ranked probabilities $(P_{i}),$ taking values in the space
$\mathcal{P}_{\infty}=\{\mathbf{s}=(s_{1},s_{2},\ldots):s_{1}\ge
s_{2}\ge\cdots\ge 0 {\mbox { and }} \sum_{i=1}^{\infty}s_{i}=1\}$ of partitions of mass, or \emph{mass partitions}, summing to $1.$  In many applications these $(P_{i})$ are coupled with a collection
$(U_{i}),$ of iid $\mathrm{Uniform}[0,1]$ variables independent of the $(P_{i}),$ leading to the construction of  $\mathrm{PD}(\alpha,\theta)$-bridges,
\begin{equation}
P_{\alpha,\theta}(y)=\sum_{k=1}^{\infty}P_{k}\indic_{\{U_{k}\le y\}} {\mbox { for }} y\in [0,1],
\label{PitmanYorProcess}
\end{equation}
which are random cumulative distribution functions. Perhaps the most remarkable and in many respects most important result described is its stick-breaking representation derived from size-biased sampling the $(P_{i}).$   Now, as in Pitman\cite[p. 110]{Pit06}, for a  sequence of probabilities $(p_{i}),$ such that $\sum_{i=1}^{\infty}p_{i}\leq 1,$ with possible \emph{dust} component $p_{0}=1-\sum_{k=1}^{\infty}p_{i},$ let $\mathrm{Rank}((p_{i}),p_{o})\in \mathcal{P}_{\infty}$ be a decreasing re-arrangement of terms in the sequence $((p_{i}),p_{o}).$  Furthermore let $\beta_{a,b}$ denote a $\mathrm{Beta}(a,b)$ distributed random variable.

\begin{enumerate}
\item[(PD)]
 Then from \cite{PY97}, Let $(P_{i})\sim \mathrm{PD}(\alpha,\theta)$ and let $(W_{k})$ denote a sequence of independent beta distributed random variables such that
$W_{k}\overset{d}=\beta_{\theta+k\alpha,1-\alpha}$  then the sequence $(\tilde{P}_{k}),$ obtained by size-biased sampling from $(P_{i}),$
can be represented as $\tilde{P}_{k}=(1-{W}_{k})\prod_{l=1}^{k-1}{W}_{l},$ where for each $k,$
$1-{W_{k}}\overset{d}=\beta_{1-\alpha,\theta+k\alpha}$ is the first size biased pick from a $\mathrm{PD}(\alpha,\theta+(k-1)\alpha)$  mass partititon. It follows that for any version of the independent $(W_{k}),$
\begin{equation}
(P_{i}):=\mathrm{Rank}(\tilde{P}_{k})\overset{d}=\mathrm{Rank}(1-W_{1},((1-W_{k})\prod_{l=1}^{k-1}W_{l})_{\{k\ge 2\}})
\label{firststick}
\end{equation}
and as a consequence a $\mathrm{PD}(\alpha,\theta)$-bridge can be represented as
\begin{equation}
P_{\alpha,\theta}(y)=\sum_{k=1}^{\infty}P_{k}\indic_{\{U_{k}\le y\}}\overset{d}=\sum_{k=1}^{\infty}(1-W_{k})\prod_{l=1}^{k-1}W_{l}\indic_{\{U_{k}\le y\}}.
\label{PitmanYorProcess2}
\end{equation}
\end{enumerate}
When $\alpha=0,$ $(W_{k})$ are iid $\mathrm{Beta}(\theta,1)$ random variables and in this case~(\ref{firststick}) corresponds to the Griffiths-Engen-McCloskey $\mathrm{GEM}(\theta)$ model for the Poisson Dirichlet model with mutation rate $\theta,$ here denoted as $\mathrm{PD}(0,\theta).$
See \cite{ABT,Ferg73,Griffiths} for more on this case.
 ~(\ref{firststick}) is seen as a two-parameter generalization sometimes denoted as $\mathrm{GEM}(\alpha,\theta).$
Additionally
$P_{0,\theta}$~(\ref{PitmanYorProcess2}) coincides with the stick-breaking representation of the Dirichlet process by Sethuraman~\cite{Sethuraman}.
Ishwaran and James~\cite{IJ2001}, named the process in~(\ref{PitmanYorProcess2}), $P_{\alpha,\theta},$ a \emph{Pitman-Yor Process}. Their work, see also~\cite{IJ2003}, played a key role in  popularizing its usage, and name, in Bayesian Nonparametric and Machine Learning applications.
See for example, \cite{RyanAdams,Goldwater,Teh, Wood1, Wood2}.
In general, and a point we shall revisit in terms of references, there are numerous uses of these  size-biased representations in the broad literature, many of which
are reflected in the area generally known as \emph{Combinatorial Stochastic Processes} and the investigation of coagulation/fragmentation phenomena as
described in the monographs of Bertoin~\cite{BerFrag} and Pitman~\cite{Pit06}. Moreover, while the independence properties described above are a key to its practical implementation and wider usage, what is also important is the interpretations gained from the size biased framework in~\cite{PPY92,PY97,Pit06}, see also \cite{BPY,PY92,PY97length} for more on this class and its relations to excursions and occupations times. This is not obvious from the independent representations above for $(W_{k}),$ and we believe this has not been fully exploited, in  applications where one's expertise might be quite removed from excursion theory. Simply put, there exists many unexplored Markov chains for novel modelling purposes. We next give an extended preview of one our focus points in this paper.
\subsection{$\mathrm{PG}(\alpha,\zeta)$ class and Stick-breaking Result}
  Our interest centers around another result, Pitman and Yor~\cite[Proposition 21]{PY97} which describes a subordinator representation of the
$\mathrm{PD}(\alpha,\theta)$ family for the restricted range of $\theta\ge 0$ as follows.  First throughout let $\gamma_{a}$ denote a gamma variable with shape parameter $a$ and scale $1,$ with law $\mathrm{Gamma}(a).$ Let $S_{\alpha}$ denote a positive stable random variable of index $0<\alpha<1$ with density $ f_{\alpha}$ and Laplace transform $\mathbb{E}[{\mbox e}^{-\omega S_{\alpha}}]={\mbox e}^{-\omega^{\alpha}}.$ Let $(\tau_{\alpha}(s): s\ge0)$
denote a generalized gamma subordinator such that at time $1,$ $\tau_{\alpha}(1)$ has the density ${\mbox e}^{-(t-1)}f_{\alpha}(t).$ Now from Proposition 21,
let $J_{1}>J_{2}>J_{3}\ldots,$ denote the ranked jumps of $\tau_{\alpha}$ over the random time interval $[0,\gamma_{\theta/\alpha},].$ It follows that
$\tau_{\alpha}(\gamma_{\theta/\alpha})=\sum_{k=1}^{\infty}J_{i}\overset{d}=\gamma_{\theta}$ and is independent of the normalized jumps
$(J_{k}/\tau_{\alpha}(\gamma_{\theta/\alpha}))\sim\mathrm{PD}(\alpha,\theta)$ for $\theta\ge 0.$
Hence this result contains $\mathrm{PD}(\alpha,0)$
as a limiting case $\theta\rightarrow 0,$
and by taking limits as $\alpha\rightarrow 0$ the important Poisson-Dirichlet $\mathrm{PD}(0,\theta)$ case, but does not contain the range $-\alpha<\theta<0.$ Nonetheless the range $\theta \ge 0$ is a significant range for the $\mathrm{PD}(\alpha,\theta)$ and Proposition 21 has many implications.

A more careful look reveals a larger class, in the following sense. First replace   $\gamma_{\theta/\alpha},$ with a constant $v.$
It follows that $\tau_{\alpha}(v)/v^{1/\alpha}$  has density ${\mbox e}^{-(v^{1/\alpha}t-v)}f_{\alpha}(t)$ and hence as described in Pitman~\cite{Pit02}, the corresponding law on $\mathcal{P}_{\infty}$
is the generalized gamma class first investigated by McCloskey~\cite{McCloskey}. This class also plays an important role in the more recent literature where its corresponding bridge is referred to as a Normalized Generalized Gamma process. We will use the notation $\mathrm{NGG}(v)$ to denote the law of the corresponding~$(P_{i}).$
Now taking $\zeta$ to be an arbitrary non-negative random variable  and restricting the interval now to $[0,\zeta]$ yields a super-class we call $\mathrm{PG}(\alpha,\zeta),$ for Poisson Gamma, which gains additional modelling flexibility through different choices of $\zeta.$ Note this class is described in the proof of Proposition 21 in \cite{PY97}.

We now state one of our major results which will be proved formally in Section 8. This is the size-biased stick-breaking representation for the entire $\mathrm{PG}(\alpha,\zeta)$ class. We point out that this result is extremely difficult to deduce by  attempting a direct assault via the integral formulae that are provided by \cite{PPY92}. This is due in part to the general intractability of the stable density $f_{\alpha},$ but as we shall show we can express the corresponding integrals in terms of an explicit integrand. Such representations might be useful for computer simulation using data augmentation methods but do not seem to help identify explicitly the random variables we now describe.

In order to emphasize the relation of our result to those of Pitman and Yor~\cite{PY97} we use the same notation they use for variables having exactly the same meaning in their context. Let $X_{1}>X_{2}>\ldots,$ denote the ranked points of a homogeneous Poisson process on $(0,\infty).$
Furthermore, $X_{k}=\Delta^{-\alpha}_{k}$ where $(\Delta_{k})$ are the ranked points of an $\alpha$-stable subordinator. For
$(\mathrm{e}_{l-1})$ a collection of iid exponential$(1)$ variables there is the relation $X_{1}=\mathrm{e}_{0}\overset{d}=\gamma_{1},$ and $X_{k}=\sum_{l=1}^{k-1}\mathrm{e}_{l}+X_{1}.$ In this case  there are independent variables $(R_{k})$ satisfying
$
R_{k}:={(X_{k}/X_{k+1})}^{1/\alpha}=\Delta_{k+1}/\Delta_{k}\overset{d}=\beta_{ka,1}.
$
See Pitman and Yor~\cite[p.870 eq, (63)]{PY97}. Now first condition  on the exponential variable $X_{1},$ then applying a change of measure we can replace $X_{1}$ with the general
$\zeta.$ So there are variables $\zeta_{k-1}:=\sum_{l=1}^{k-1}\mathrm{e}_{l}+\zeta,$ with $\zeta_{0}=\zeta.$

\begin{enumerate}
\item[(PG)]Suppose that $(P_{i})\sim \mathrm{PG}(\alpha,\zeta),$ then there is a stick-breaking representation obtained from size biased sampling that is of the form~(\ref{firststick}) where now $(\tilde{P}_{k})$ are defined by
$$
1-W_{k}=\beta^{(k)}_{1-\alpha,\alpha}[1-R_{k}]
$$
where $(\beta^{(k)}_{1-\alpha,\alpha})$ are iid $\mathrm{Beta}(1-\alpha,\alpha)$ random variables independent of
$R_{k}={(\zeta_{k-1}/\zeta_{k})}^{1/\alpha},$ for all $k.$
\end{enumerate}

Notice that although the sticks $(W_{k})$ are generally dependent, they have  remarkably simple forms that not only has implications in terms of practical modelling capabilities but also connects directly with results in \cite{PY97}. Here again we only present the result in its basic form that does not explicitly connect it up with other entities that appear in the excursion setting of~\cite{PPY92}. Those will be revealed later.
We close out this section with a few comments that also apply to our more general forthcoming result.
Note the largest class generated by a stable subordinator is denoted as $\mathrm{PK}_{\alpha}(h\cdot f_{\alpha})$ formed by conditioning on the total mass of a stable subordinator and mixing with respect to the density  $h(t)f_{\alpha}(t).$ The $\mathrm{PD}(\alpha,\theta)$ case arises by choosing a density proportional to
$t^{-\theta}f_{\alpha}(t).$ A significant feature of the $\mathrm{PK}_{\alpha}(h\cdot f_{\alpha})$ class and their limiting cases is that they generate the only infinite exchangeable partitions of $\mathbb{N},$ equivalently $[n]=\{1,2,\ldots,n\},$ such that their \emph{exchangeable partition probability functions}, $\mathrm{EPPF},$ is of product or Gibbs form, consistent in $n.$ This has many implications, a simple one is the generalized Chinese restaurant process used to generate such partitions has a relatively simple seating rule.  These results are described in~Pitman\cite{Pit02,Pit06} and Gnedin and Pitman\cite{GnedinPitmanI}. See also~\cite{HJL} for more on these calculations.

However, while it is possible to get a handle on the EPPF for many choices of $h,$ stick-breaking results are another matter entirely. Our results represent the first explicit unconditional stick-breaking representations that hold for such a class for all $0<\alpha<1$ outside of the $\mathrm{PD}(\alpha,\theta).$ In the case of $\alpha=1/2,$ where the $f_{1/2}(t)$ density is that of a $C/\gamma_{1/2}$ variable, for some positive C, and hence is completely explicit, Aldous and Pitman~\cite{AldousPit}, achieve the most general stick-breaking result possible by obtaining an explicit description in the case of the law  $\mathrm{PD}(1/2|t),$ that is the distribution of $(P_{i})|T=t.$  This follows as a fairly straightforward application of \cite[Theorem 2.1]{PPY92}.
Also in the $\alpha=1/2$ case, a  conditional result is given~in~\cite{Fav1}, for a special case which is the Inverse Gaussian distribution which corresponds to mixing $\mathrm{PD}(1/2|t),$  over the density proportional to $e^{-vt}f_{1/2}(t).$
Lastly, among the general Poisson-Kingman models discussed in~\cite{Pit02}, which is the class of laws represented by the framework of~\cite{PPY92}, the only case where one will obtain independent sticks $(W_{k})$ is the $\mathrm{PD}(\alpha,\theta)$ case.
\subsection{Further Discussion and General Outline}
The relationship between $\mathrm{PG}(\alpha,\zeta)$ models and $\mathrm{PD}(\alpha,\theta)$ models can be represented as follows
$$
\mathrm{PD}(\alpha,\theta)\subseteq \mathrm{PG}(\alpha,\zeta) {\mbox { for }}\theta\ge 0; \mathrm{PD}(\alpha,\theta)\cap \mathrm{PG}(\alpha,\zeta)=\emptyset  {\mbox { for }}-\alpha<\theta< 0.
$$
Not satisfied with this, a larger class is constructed called $\mathrm{EPG}(\alpha,\zeta),$ for extended Poisson Gamma, which naturally contains $\mathrm{PG}(\alpha,\zeta),$ and otherwise satisfies
$$
\mathrm{PD}(\alpha,\theta)\subseteq \mathrm{EPG}(\alpha,\zeta) {\mbox { for }}\theta>-\alpha.
$$
The bulk of this paper concentrates on identifying properties of this class, its precise relationship to $\mathrm{PG}(\alpha,\zeta),$ and its relationship to $\mathrm{PD}(\alpha,\theta).$ This leads to connections between various Markov chains in the literature among other results. Naturally we also obtain an explicit stick-breaking representation for the $\mathrm{EPG}(\alpha,\zeta)$ class, furthermore we provide an analysis of a related class  of variables where special cases have appeared in some guise in the literature. An outline is presented next. Section 2 presents some more preliminaries related to the class of laws generated by a stable subordinator (with some repetition). Section 3 presents the $\mathrm{PG}(\alpha,\zeta),$  in more detail.
Section 3.3. discusses key identities. Section 4 describes the explicit construction of $\mathrm{EPG}(\alpha,\zeta),$ models and begins to discuss various connections with the coagulation/fragmentation framework in BDGM\cite{BertoinGoldschmidt2004,Dong2006}.
Section 4.3 is of practical interest but also plays a key role in helping us understand the precise nature of  $\mathrm{EPG}(\alpha,\zeta)$ laws. Section 4.4 describes how the sampling schemes in Section 4.3 connects to BDGM \cite{BertoinGoldschmidt2004,Dong2006} and raises points about similarities to coalescent/fragmentation processes on partitions of $[n]=\{1,2,\ldots,n\}$ appearing in the literature. We close Section 4 with an open problem/question involving Aldous's beta-splitting model, used in this case as rule for coalescence, and the beta-coalescent. Section 5 is a key section that identifies in more detail what $\mathrm{EPG}(\alpha,\zeta)$ models are. This also plays a major role in the developments of the subsequent sections. Section 6 describes a Markov chain that is complementary to the stick-breaking induced chains in Pitman, Perman and Yor~\cite{PPY92}, and describes how such chains are directly related to BDGM. From this we identify explicitly a class of variables in the $\mathrm{EPG}(\alpha,\zeta)$ setting, where special cases have appeared in the literature on fragmentation trees as well as BDGM. Section 6.4 perhaps suggests a different way of constructing flows of bridges hence certain type of coalescent processes. Moreover it presents new ways to approximate Pitman-Yor processes. Sections 7 and 8 present the stick-breaking results which also connects up with the variables in Section 6. As mentioned previously, although we can express the \cite{PPY92} formulae in somewhat simpler forms, we only use this as guidelines to develop our results.
That section closes by showing how we obtained the simplest representations in the $\mathrm{PG}(\alpha,\zeta)$ case which was/is otherwise  not at all obvious. Section 9 revisits connections to the variables we described earlier appearing in \cite{PY97}. Note some of the ideas presented here have their origins in the unpublished manuscript James~\cite{James2010}

The $\mathrm{PD}(\alpha,\theta)$ distribution, and related models, have connections to Normal, gamma, stable and Poisson laws and hence it is not completely surprising that it arises in many guises in a wide variety of seemingly unrelated topics involving probability and statistics. Including
Bayesian statistics, population genetics, statistical physics, finance and machine learning. A sample of
additional references
follows,~\cite{RyanAdams,Cerquetti, Chatterjee, DurrettSchweinsberg, Feng, JLP,
Kerov, Ruelle, Sudderth,TalagrandSpin2,Vershik}.
In regards to terminology and style in this work our primary references will  be the monographs of~\cite{BerFrag,Pit06}  with references and citations representing a substantial stream of the burgeoning literature on the $\mathrm{PD}(\alpha,\theta)$ model. We will employ simple bridges as described in~\cite[Section 4]{BerFrag}. Otherwise we assume the reader has a basic knowledge of terms such as EPPF, Chinese restaurant process etc that we shall occasionally use.

\section{Pitman-Yor processes and Poisson-Kingman distributions determined by a stable subordinator}
For $0<\alpha<1,$ let  $T_{\alpha}:=(T_{\alpha}(s):s\ge 0)$ denote a $\alpha$-stable subordinator of index $\alpha.$  That is to say $T_{\alpha}(\cdot)$ is an independent increment process such that for each fixed $t,$ $T_{\alpha}(t)\overset{d}=t^{1/\alpha}S_{\alpha};$ where $S_{\alpha}$ is a positive stable random variable of index $\alpha, $with density denoted as $f_{\alpha},$ and  whose log Laplace transform is given by $-C\omega^{\alpha}$ for some constant $C>0$ and each $\omega>0.$
Hereafter, due to scaling properties, we can take $C=1.$ It is well-known that for each $u\in [0,1],$ one can set
$$
T_{\alpha}(u)=\sum_{k=1}^{\infty}\Delta_{k}\indic_{\{U_{k}\leq u\}}
$$
where $\Delta_{1}\ge\Delta_{2}\ge\cdots$ are the ranked jumps of $T_{\alpha}.$ As in Kingman~\cite{Kingman75}, one can define a random probability measure by
\begin{equation}
P_{\alpha,0}(u)\overset{d}=\frac{T_{\alpha}(u)}{T_{\alpha}(1)}=\sum_{k=1}^{\infty}P_{k}\indic_{\{U_{k}\leq u\}}
\label{rpmstable}
\end{equation}
In this case, setting $T_{\alpha}(1):=T,$ the ranked sequence of probabilities $(P_{i}=\Delta_{i}/T)$ is said to have a two-parameter Poisson-Dirichlet $(\alpha,0)$ law, denoted as $\mathrm{PD}(\alpha,0).$ The sequence, as discussed in~\cite{Pit06,PPY92, PY97}, arises in the study of various phenomena related to Brownian Motion, $\alpha=1/2,$ and more general Bessel processes. The random probability measure  $P_{\alpha,0}$ is sometimes referred to as a Pitman-Yor $(\alpha,0)$ process with, in this case, a Uniform base measure $\mathbb{U},$ defined by $\mathbb{E}[P_{\alpha,0}(y)]=y=\mathbb{U}(y)$, for $y\in [0,1].$ For this choice of base measure we refer to $P_{\alpha,0}$ as a $\mathrm{PD}(\alpha,0)$-bridge.

The $\mathrm{PD}(\alpha,0)$-bridge, represents one book-end of the Pitman-Yor $(\alpha,\theta)$ family of processes. The other is the  Dirichlet Process with total mass parameter $(\theta),$ popularized by Ferguson\cite{Ferg73}. As is well known, this may be constructed by normalizing a gamma subordinator in a manner similar to the $P_{\alpha,0}$ construction. The corresponding sequence of probabilities $(P_{i})$ is  said to follow a Poisson-Dirichlet law of index $\theta,$ denoted as $\mathrm{PD}(0,\theta).$
The general $\mathrm{PD}(\alpha,\theta)$ family can be obtained from the following construction. Let $(P_{i})$ be the ranked sequence of probabilities following a PD$(\alpha,0)$ law as described above. Note that $T_{\alpha}(1):=T,$ is a measurable function of $(P_{i})$ as seen by the relation
\begin{equation}
T=Z^{-1/\alpha}:=\lim_{i\rightarrow \infty}{(i\Gamma(1-\alpha)P_{i})}^{-1/\alpha}
\label{alphadiversity}
\end{equation}
almost surely
where $Z=T^{-\alpha}$ is the $\alpha$-diversity of a $\mathrm{PD}(\alpha,0)$ partition of the integers $[n]=\{1,\ldots,n\}.$
This point makes conditioning on $T$ meaningful. Then conditioning on $(P_{i})|T=t$ yields a conditional law on the sequence $(P_{i}),$ denoted as $\mathrm{PD}(\alpha|T=t):=\mathrm{PD}(\alpha|t).$

Now define the variables $S_{\alpha,\theta}$ for each $\theta>-\alpha,$ as having a density, denoted by $f_{\alpha,\theta},$ formed by polynomially tilting a stable density as follows
\begin{equation}
f_{\alpha,\theta}(t)=c_{\alpha,\theta}t^{-\theta}f_{\alpha}(t)
\label{PDdiversity}
\end{equation}
where
$
c_{\alpha,\theta}:=\Gamma(\theta+1)/\Gamma(\theta/\alpha+1),
$
and satisfies for $\delta+\theta>-\alpha$
\begin{equation}
\mathbb{E}[S^{-\delta}_{\alpha,\theta}]=\frac{\Gamma(\theta+1)}{\Gamma(\theta/\alpha+1)}
\mathbb{E}[S^{-(\delta+\theta)}_{\alpha}]=\frac{\Gamma(\frac{(\theta+\delta)}{\alpha}+1)}{\Gamma({\theta+\delta}+1)}\frac{\Gamma(\theta+1)}{\Gamma(\theta/\alpha+1)}.
\label{moment}
\end{equation}

Then the $\mathrm{PD}(\alpha,\theta)$ law on $(P_{i})$ is obtained by mixing $PD(\alpha|t)$ with respect to $f_{\alpha,\theta}(t)$ as follows
$$
\mathrm{PD}(\alpha,\theta):=\int_{0}^{\infty}\mathrm{PD}(\alpha|t)f_{\alpha,\theta}(t).
$$
A $\mathrm{PD}(\alpha,\theta)$ bridge or Pitman-Yor-$(\alpha,\theta)$ process, represented as in (\ref{PitmanYorProcess2}), is obtained by applying this change of measure to the bridge in~(\ref{rpmstable}). Since the $(U_{i})$ are independent of the $(P_{i}),$ this is equivalent to choosing $(P_{i})\sim \mathrm{PD}(\alpha,\theta).$ As described in Pitman~\cite{Pit02,Pit06}, this idea may be generalized to quite general mixing densities which we may write as $h(t)f_{\alpha}(t),$ for $h$ satisfying  $\mathbb{E}[h(S_{\alpha})]=1.$ Hence $(P_{i})$ and related quantities are said to have distributions determined by a Poisson-Kingman law generated by a $\alpha$-stable subordinator with mixing distribution $h\cdot f_{\alpha},$ say $\mathrm{PK}_{\alpha}(h\cdot f_{\alpha}),$ defined by
\begin{equation}
\mathrm{PK}_{\alpha}(h\cdot f_{\alpha}):=\int_{0}^{\infty}\mathrm{PD}(\alpha|t)h(t)f_{\alpha}(t).
\label{PKh}
\end{equation}
\begin{rem} From Pitman~\cite{Pit02,Pit06} it follows that $(P_{i})\sim\mathrm{PK}_{\alpha}(h\cdot f_{\alpha})$ satisfies (\ref{alphadiversity}), where $T$ has density $h(t)f_{\alpha}(t)$ and $Z=T^{-\alpha}$ is the corresponding $\alpha$-diversity. In the $\mathrm{PD}(\alpha,\theta)$ setting $T=S_{\alpha,\theta}$ and $Z=S^{-\alpha}_{\alpha,\theta}$ is the $\alpha$-diversity.
\end{rem}

Note one may replace the $(U_{i}),$ with arbitrary iid variables $(Y_{k})$ having non-atomic distribution $H$, which is equivalent to applying  a standard composition  operation, $P_{\alpha,\theta}\circ H$
i.e. $P_{\alpha,\theta}(H(s)).$ with expectation $\mathbb{E}[P_{\alpha,\theta}(H(s))]=H(s).$ This also holds when $H$ has atoms, however in that case the bijection, i.e. \emph{Kingman's correspondence}, between the bridge law and that of the $(P_{i})$ breaks down, which is easily seen by taking $H$ to be point mass. In particular, the $\mathrm{PD}(\alpha,\theta)$ generalized Chinese restaurant process, (see \cite{Pit96}), generating partitions of $[n]=\{1,2,\ldots,n\}$ for these models is not the same
See \cite{Buntine, Teh, Wood1,Wood2} for applications utilizing~(\ref{PitmanYorProcess}) with a discrete $H,$ which is appropriate for Natural Language models. See the forthcoming Section~\ref{samplingcoag} for a general approach to sampling.

\subsection{Exchangeable bridges and partitions}
Besides the $\mathrm{PD}(\alpha,\theta)$ class, we will work extensively with simple bridges, As such we present some relevant details that is essentially lifted from Bertoin's~\cite{BerFrag} book. Following Bertoin~\cite[Definition 2.1, p.67]{BerFrag},(see also
Pitman\cite[Section 5]{Pit06}), an infinite numerical sequence
$\mathbf{s}=(s_{1},s_{2},\ldots)$ is said to be a
\textbf{mass-partition} if $\mathbf{s}$ is an element of the
space,
$$
\mathcal{P}_{\mathrm{m}}=\{\mathbf{s}=(s_{1},s_{2},\ldots):s_{1}\ge
s_{2}\ge\cdots\ge 0 {\mbox { and }} \sum_{i=1}^{\infty}s_{i}\leq
1\}.
$$
The quantity
$
s_{0}:=1-\sum_{i=1}^{\infty}s_{i},
$
which may be $0,$ is referred to as the \emph{total mass of dust}.
From Bertoin~(\cite{BerFrag}, Definition 4.6, p. 191), a random
caglad process $b_{\mathbf{s}}$ on $[0,1]$ is said to be an
\textbf{s-bridge} if it is distributed as
$$
b_{\mathbf{s}}(y)=s_{0}\mathbb{U}(y)+\sum_{i=1}^{\infty}s_{i}\indic_{(U_{i}\leq
y)}, y\in[0,1],
$$
for $(U_{i})$  a sequence of iid $\mathrm{Uniform}[0,1]$ random variables.
If $\mathbf{s}\sim\mathbb{P}$, i.e. if
$\mathbf{s}$ is randomized according to some law $\mathbb{P},$
then $b_{\mathbf{s}}$ is said to be a $\mathbb{P}$-bridge. It
follows that $\mathcal{P}_{\infty}$ is a subspace of
$\mathcal{P}_{\mathrm{m}}$ such that $\sum_{i=1}^{\infty}s_{i}=1.$
Furthermore, for all $\mathbf{s}\in P_{\mathbf{m}},$
$\mathrm{Rank}(s_{0}, s)\in \mathcal{P}_{\infty}.$
Additionally let
\begin{equation}
b^{-1}_{\mathbf{s}}(r)=\inf\{v\in [0,1]: b_{\mathbf{s}}(v)>r\},
r\in[0,1]
\label{inversebridge}
\end{equation}
denote the right continuous inverse of the bridge. Equivalently
this is a random quantile function.   An
exchangeable partition of $[n]=\{1,2,\ldots,n\}$ generated from an exchangeable
bridge, say $b_{\mathbf{s}},$ can be obtained by the equivalence relations
$$
i\sim j{\mbox { iff }}b^{-1}_{\mathbf{s}}(U'_{i})=b^{-1}_{\mathbf{s}}(U'_{j})
$$
based on $n$ iid Uniform$[0,1]$ variables
$(U'_{1},\ldots,U'_{n}).$ An infinite partition, $\Pi$ of $\mathbb{N}$ is formed by
considering a countably infinite set of uniforms.
The distribution of such infinite exchangeable partitions is referred to as an exchangeable partition probability function (EPPF). We see that the random probability measure in~(\ref{PitmanYorProcess2})
is a $\mathbb{P}$-bridge , with $(s_{i})\overset{d}=(P_{i})\in \mathcal{P}_{\infty}$ distributed according to the law $\mathbb{P}=\mathrm{PD}(\alpha,\theta)$ with $s_{0}=0.$ An important property,  which we shall exploit, is that the law of the $\mathbb{P}$-bridge is in bijection to the law of the sequence $(P_{i})\sim \mathbb{P},$ and also to the corresponding EPPF , specifying the law of the exchangeable partition $\Pi$ with ranked frequencies $(P_{i}),$ which we shall refer to as a $\mathbb{P}$-EPPF.
\subsubsection{Simple Bridges}
In this manuscript we will also utilize properties of simple bridges. In particular if $\mathbf{s}=(u,0,\ldots)$ is a simple mass-partition then
$b_{u}(y)=(1-u)\mathbb{U}(y)+u\indic_{(U_{1}\leq y)}$ is referred to as a
\textbf{simple bridge}. If $u=s_{1}$ is a random variable then one has a randomized simple bridge given by,
\begin{equation}
b_{s_{1}}(y)=s_{0}\mathbb{U}(y)+s_{1}\indic_{(U_{1}\leq y)}
\label{simplebridge}
\end{equation}
The inverse of a simple bridge, defined as a special case of~(\ref{inversebridge}), is denoted as
$b^{-1}_{s_{1}}.$ From
Bertoin(\cite{BerFrag}, eq. (4.14), p. 194) one sees that for
$(U'_{k})_{k\ge 1}$ iid Uniform$[0,1]$ random variables
independent of $U_{1},$
$$b^{-1}_{s_{1}}(U'_{k})=U_{1}, {\mbox { iff }}
U'_{k}\in(s_{0}U_{1}, (1-s_{0})+s_{0}U_{1}),$$ having  length
$s_{1}=1-s_{0}$ and otherwise
$b^{-1}_{s_{1}}(U'_{k})\overset{d}=U^{0}_{k}$ has an independent
Uniform$[0,1]$ distribution, that is for $U'_{k}\in
[0,s_{0}U_{1}]\cup[s_{0}U_{1}+1-s_{0},1].$ More precisely, define
for each $k,$
\begin{equation}
I_{k}\overset{d}=\indic_{(b^{-1}_{s_{1}}(U'_{k})=U_{1})}\overset{d}=\indic_{(U'_{k}\leq
s_{1})}\sim \mathrm{Bernoulli}(s_{1}) \label{indicator}
\end{equation}
then
$
\mathbb{P}(b^{-1}_{s_{1}}(U'_{k})\leq y|I_{k}=0)=y, y\in [0,1].
$
\section{The $\mathrm{PG}(\alpha,\zeta)$ Generalized Gamma family of processes}
We now show how to construct the $\mathrm{PG}(\alpha,\zeta)$ Generalized Gamma family of processes. The family constitutes processes that are governed by a large sub-class of $\mathrm{PK}_{\alpha}(h\cdot f_{\alpha})$ distributions that are, for each $\alpha,$ indexed further by the choice of any non-negative random variable $\zeta.$ We first present a description which directly relates to the $\mathrm{PK}_{\alpha}(h\cdot f_{\alpha})$ construction discussed in the previous section. First condition $(P_{i})\sim \mathrm{PD}(\alpha,0)$ on $T=s,$ yielding $(P_{i})\sim \mathrm{PD}(\alpha|s)$ and then mix with respect to the density of  random variables
of the form
$$
T=\frac{\tau_{\alpha}(\zeta)}{\zeta^{1/\alpha}}.
$$
$\zeta$ is a non-negative random variable taken independent of  $(\tau_{\alpha}(s),s>0),$ which is a generalized gamma subordinator whose L\'evy exponent, i.e. its -log Laplace transform of $\tau_{\alpha}(1),$ is given by
\begin{equation}
\psi_{\alpha}(\omega)=(1+\omega)^{\alpha}-1
\label{GGexp}
\end{equation}
for $\omega>0.$ Note, unless specified otherwise, for non-negative quantities $a,b$ we will consider $\tau_{\alpha}(a)+\tau_{\alpha}(b)=\tau_{\alpha}(a+b),$
rather than equivalent in distribution, and otherwise $\tau_{\alpha}(a)$ and $\tau_{\alpha}(b)$ are independent. Meaning, for instance, that we are looking at the decomposition of the interval $(0,a+b)$ where $b$ represents the length of the interval $(a,a+b).$
The conditional density of $T|\zeta$ is given by
$$
f_{\alpha}(t|\zeta):=f_{\alpha}(s){\mbox e}^{-(s\zeta^{1/\alpha}-\zeta)},
$$
which is the density of a generalized gamma random variable, $\tau_{\alpha}(\zeta)/\zeta^{1/\alpha}.$ The random variable is infinitely divisible and is specified for fixed $\zeta$ by its -log Laplace transform,
\begin{equation}
\psi_{\alpha,\zeta}(\omega):=\zeta\psi_{\alpha}(\omega/\zeta^{1/\alpha})=(\zeta^{1/\alpha}+\omega)^{\alpha}-\zeta
\label{GGexp2}
\end{equation}
As such, conditional on $\zeta,$ we can associate this variable with a generalized gamma subordinator of the form
$(\tau_{\alpha,\zeta}(s),0\leq s\leq 1),$ such that for each $s,$ $\tau_{\alpha,\zeta}(s)$ has L\'evy exponent
$s\psi_{\alpha,\zeta}(\omega),$ and $\tau_{\alpha,\zeta}(1)\overset{d}=T.$
It follows that the (unconditional) density of $T$ can be expressed as,
$$
f_{\alpha}(s)\mathbb{E}[{\mbox e}^{-(s\zeta^{1/\alpha}-\zeta)}].
$$
Hence the $(P_{i})$ constructed in this manner follow a $\mathrm{PK}_{\alpha}(h\cdot f_{\alpha})$ distribution with
\begin{equation}
h(s)=\mathbb{E}[{\mbox e}^{-(s\zeta^{1/\alpha}-\zeta)}].
\label{genh}
\end{equation}
Clearly, in addition to $\alpha,$ this family is indexed by $\zeta,$ and we write $(P_{i})~\sim \mathrm{PG}(\alpha,\zeta)$, to denote this class of Poisson Kingman laws for the ranked sequence $(P_{i}).$ It follows that a $\mathrm{PG}(\alpha,\zeta)-$bridge, say $\{Q_{\alpha,\zeta}(y):0\leq y\leq 1\},$ can be represented as,
\begin{equation}
Q_{\alpha,\zeta}(y)\overset{d}=\sum_{k=1}^{\infty}P_{k}\indic_{\{U_{k}\leq y\}}, {\mbox { for }}(P_{i})\sim \mathrm{PG}(\alpha,\zeta).
\label{PGbridge1}
\end{equation}
As mentioned in the introduction, this construction of $\mathrm{PG}(\alpha,\zeta)$ laws coincides with random processes discussed in Pitman and Yor \cite[p. 877-878]{PY97}. This construction is used to prove  Pitman and Yor \cite[Proposition 21, p. 869]{PY97}. Proposition 21 of that work shows that if $\zeta\overset{d}=\gamma_{\theta/\alpha}$ where $\gamma_{\theta/\alpha}$ denotes a random variable with a gamma distribution with shape parameter $(\theta/\alpha),$ and scale $1,$  then $$\mathrm{PG}(\alpha,\gamma_{\theta/\alpha})=\mathrm{PD}(\alpha,\theta){\mbox { and hence }}Q_{\alpha,\gamma_{\theta/\alpha}}(y)\overset{d}=P_{\alpha,\theta}(y),
$$ for $\theta>0.$ Note this does not include the case of $\mathrm{PD}(\alpha,\theta)$ for $-\alpha<\theta<0.$ It is evident from results for $\mathrm{PD}(\alpha,\theta),$ that the  stable and Dirichlet cases arise as,
$$\mathrm{PG}(\alpha,0)=\lim_{\theta\rightarrow 0}\mathrm{PG}(\alpha,\gamma_{\theta/\alpha})=
\mathrm{PD}(\alpha,0){\mbox { and }}\lim_{\alpha\rightarrow 0}\mathrm{PG}(\alpha,\gamma_{\theta/\alpha})=\mathrm{PD}(0,\theta).$$ Furthermore, the case $\alpha=1$ is well defined, as $\mathrm{PG}(1,\zeta)$ for any $\zeta$ corresponds to the case where $(P_{i})=(0,0,\ldots).$ Equivalently, it is simple to show that,
$$
Q_{1,\zeta}(y):=\lim_{\alpha\rightarrow 1}Q_{\alpha,\zeta}(y)\overset{d}=\mathbb{U}(y), {\mbox { for }}y\in[0,1].
$$
That is when $\alpha\rightarrow 1,$ the random probability measure converges converges to a simple bridge with dust $s_{0}=1.$
 When $\zeta=v$ is a positive constant, $\mathrm{PG}(\alpha,v)$ corresponds to the case of the Poisson-Kingman model determined by the generalized gamma subordinator as described in Pitman \cite[Section 5.2]{Pit02} and attributed to McCloskey~\cite{McCloskey}. In this case the bridge has been studied from a Bayesian perspective in \cite{James2002,LMP1,LMP2, JLP2}, where it is often referred to as a Normalized generalised gamma process. We write $\mathrm{PG}(\alpha,v)=\mathrm{NGG}(
 \alpha,v)$ to denote the law of the normalized jumps. Due to the generality of $\zeta$, the class of $\mathrm{PG}(\alpha,\zeta)$  laws is significantly larger than the special cases mentioned.  See Donati-Martin and Yor~\cite{Donati} for an interpretation of $(\tau_{\alpha,\zeta}(s),0\leq s\leq 1),$
as the inverse-local time of  a \emph{downwards Bessel process} with random drift $\zeta^{1/\alpha}.$
\subsection{Representation as a randomly time-changed normalized generalized gamma process}
As noted above, for fixed $\zeta:=v,$ the random probability measures correspond to the class of normalized generalized gamma processes studied for instance in  \cite{James2002,LMP1,LMP2, JLP2} and more recently in~\cite{Changyou}.
It is then not difficult to see that one can write a $\mathrm{PG}(\alpha,\zeta)-$bridge as
\begin{equation}
Q_{\alpha,\zeta}(y)\overset{d}=\frac{\tau_{\alpha}(\zeta y)}{\tau_{\alpha}(\zeta)}\overset{d}=\sum_{i=1}^{\infty}P_{i}\indic_{(U_{i}\leq
y)},  y\in[0,1]
\label{Laplacerpm}
\end{equation}
where, $(P_{i})\sim \mathrm{PG}(\alpha,\zeta).$ Note that by scaling
$$
Q_{\alpha,\zeta}(y)\overset{d}=\frac{\tau_{\alpha}(\zeta y)}{\tau_{\alpha}(\zeta)}\overset{d}=\frac{\tau_{\alpha,\zeta}
(y)}{\tau_{\alpha,\zeta}(1)}.
$$
\subsubsection{Representing a $\mathrm{PitmanYor}(\alpha,\theta)$ process for $\theta\ge 0$ as a normalized process}\label{construction}
As we mentioned previously and demonstrated in terms of $(P_{i})\sim PD(\alpha,\theta),$ the fact that Pitman-Yor processes with parameters $0\leq \alpha<1$, and $\theta\ge 0$ are $\mathrm{PG}(\alpha,\zeta)$  processes follows as a direct consequence of Pitman and Yor \cite[Proposition 21, p. 869]{PY97}, see also \cite{BerFrag,Dong2006,Feng}. The $(\alpha,0)$ model and $(0,\theta)$ Dirichlet process arise as limiting cases.
Setting $\zeta\overset{d}=\gamma_{\theta/\alpha}$ it follows that a $\mathrm{PitmanYor}(\alpha,\theta)$ process can be represented as
\begin{equation}
P_{\alpha,\theta}(y)\overset{d}=\frac{\tau_{\alpha}(\gamma_{\theta/\alpha} y)}{\tau_{\alpha}(\gamma_{\theta/\alpha})}=
Q_{\alpha,\gamma_{\frac{\theta}{\alpha}}}(y).
\label{PYb}
\end{equation}
We note also the following important identities that can be found in Pitman and Yor \cite[equations (98,99,100)]{PY97}, for $\theta\ge 0$
\begin{equation}
\label{PYlamperti}
\frac{\tau_{\alpha}(\gamma_{\theta/\alpha})}{\gamma^{1/\alpha}_{\theta/\alpha}}{\mbox { is independent of  }}
\tau_{\alpha}(\gamma_{\theta/\alpha})\overset{d}=\gamma_{\theta}.
\end{equation}
This implies, generalizing the case of the Dirichlet process, that $P_{\alpha,\theta}$ constructed as (\ref{PYb}) is independent of its total mass which has a Gamma$(\theta)$ distribution.
\subsection{Mixing distribution with respect to $\zeta$}
As we discussed above, $\zeta$ may be taken to be an arbitrary non-negative random variable. However, similar to Poisson-Kingman mixing distributions for the stable case, $h(s)f_{\alpha}(s),$ where the canonical distribution is that of $S_{\alpha}$, we will express the distribution of $\zeta,$ so that its canonical form corresponds to a exponential$(1)$ distribution. This is done as follows, let $g(s)$ denote a non-negative function such that
$$
\mathbb{E}[g(\gamma_{1})]=1.
$$
Then a general mixing distribution for $\zeta$ may be expressed as
\begin{equation}
f_{\zeta}(y)=g(y){\mbox e}^{-y}
\label{zetamix}
\end{equation}
Setting $g(y)=\delta_{v}(y){\mbox e}^{y}$ corresponds to $\zeta=v.$
Using~(\ref{zetamix}), $h(s)$ in~(\ref{genh}) may be expressed as
$$
h(s)=\Gamma(1+\alpha)s^{-\alpha}\mathbb{E}[g({(s/\gamma_{\alpha})}^{\alpha})]=\alpha\int
_{0}^{\infty}g(y^{\alpha})y^{\alpha-1}{\mbox e}^{-ys}dy.
$$
The latter expression is the Laplace transform of a general function $\tilde{g}_{\alpha}(y):=\alpha y^{\alpha-1}g(y^{\alpha})$ evaluated at $s.$ Writing this Laplace transform as $\mathbb{L}_{\tilde{g}_{\alpha}}(s),$ gives
$$
h(s)f_{\alpha}(s)=f_{\alpha,\alpha}(s)\mathbb{E}[g({(s/\gamma_{\alpha})}^{\alpha})]=f_{\alpha}(s)\mathbb{L}_{\tilde{g}_{\alpha}}(s).
$$
Hence
\begin{equation}
\mathrm{PG}(\alpha,\zeta)=\int_{0}^{\infty}\mathrm{NGG}(\alpha,v)f_{\zeta}(v)dv=\mathrm{PK}_{\alpha}(\mathbb{L}_{\tilde{g}_{\alpha}}\cdot f_{\alpha}).
\label{Laplacetransformclass}
\end{equation}
Setting $g(v)=1,$ corresponding to the  canonical case where $\zeta$ has an exponential$(1)$ distribution,
gives
$$
\mathrm{PG}(\alpha,\gamma_{1})=\mathrm{PD}(\alpha,\alpha){\mbox { and }}Q_{\alpha,\gamma_{1}}\sim \mathrm{PitmanYor}(\alpha,\alpha)
$$

\begin{rem}
Note that there is the identity, see \cite{JamesLamperti} and related references therein,
$$
\frac{S_{\alpha,\alpha}}{\gamma_{\alpha}}\overset{d}=\frac{S_{\alpha}}{\gamma_{1}}\overset{d}=\gamma^{-1/\alpha}_{1},
$$
leading to $\int_{0}^{\infty}f_{\alpha,\alpha}(s)\mathbb{E}[g({(s/\gamma_{\alpha})}^{\alpha})]ds=\mathbb{E}[g(S^{\alpha}_{\alpha,\alpha}/\gamma^{\alpha}_{\alpha})]=1.$
\end{rem}

\begin{rem}Several representations for the law $\mathrm{PG}(\alpha,\zeta)$ have been given above. For practical matters, such as statistical simulation of relevant quantities produced by $\mathrm{PG}(\alpha,\zeta),$ and in fact for the derivation of the results herein, we believe the easiest one to use is the form of $h(s)$ given in terms of $\zeta,$ equation~(\ref{genh}). However, the description of $h(s)$ in terms of the Laplace transform $\mathbb{L}_{\tilde{g}_{\alpha}},$ (\ref{Laplacetransformclass}), perhaps more clearly captures the characteristics of the type of sub-models of $\mathrm{PK}_{\alpha}(h\cdot f_{\alpha})$ that are within $\mathrm{PG}(\alpha,\zeta).$
\end{rem}

\subsection{Identities}
This section develops key identities that we shall use. Note, similar to \cite[Section 8]{Kingman75}, this represents part of a calculus that allows one to replace integrals with respect to the implicit form of $f_{\alpha}$ with explicit quantities that hold for any $\alpha \in(0,1).$ In a sense this is a continuation of related work in James~\cite{JamesLamperti} and has connections to \cite{PPY92,Pit06,PY92} and references discussed therein.

Before we present new results we highlight two, what we would describe as, fundamental identities. In both cases the identities hold for the range $\theta>-\alpha,$ and the variables appearing on the right hand side of the equation are independent.
First from Perman, Pitman and Yor~\cite[Lemma 3.11]{PPY92}, see also Pitman and Yor~\cite[Theorem 1.3]{PY92} and Pitman~\cite[p. 82-83]{Pit06}, there is,
$$
S_{\alpha,\theta}\overset{d}=S_{\alpha,\theta+\alpha}\times \beta^{-1}_{\theta+\alpha,1-\alpha}.
$$
and importantly a \emph{version} of these variables where $(S^{-\alpha}_{\alpha,\theta}, S^{-\alpha}_{\alpha,\theta+\alpha})$ can be interpreted as the $\alpha$-diversities of $
(\mathrm{PD}(\alpha,\theta),\mathrm{PD}(\alpha,\theta+\alpha))$
related by the \emph{equality}
\begin{equation}
S_{\alpha,\theta}=S_{\alpha,\theta+\alpha}\times \beta^{-1}_{\theta+\alpha,1-\alpha},
\label{PPYfundid}
\end{equation}
where $1-\beta_{\theta+\alpha,1-\alpha}$ can be interpreted as the first size biased pick from $\mathrm{PD}(\alpha,\theta).$
The identity actually encodes the stick-breaking results in the $\mathrm{PD}(\alpha,\theta)$ case. For instance setting $\theta=0,$ one can easily recover the Markov Chain results of~\cite[Theorem 2.1, Corollary 2.3]{PPY92} in the $\alpha$-stable setting,  hence a result that gives expressions, albeit not necessarily tractable, for all  $\mathrm{PK}_{\alpha}(h\cdot f_{\alpha}).$

That is a Markov Chain $(S_{\alpha,0},S_{\alpha,\alpha},S_{\alpha,2\alpha},\ldots),$ such that the conditional density of
$S_{\alpha,\alpha}|S_{\alpha,0}=t$ is the same as the conditional density of $S_{\alpha,k\alpha}|S_{\alpha,(k-1)\alpha}=t$ for k$=1,2,\ldots$ etc.
Furthermore, conditioning on $S_{\alpha,0}=t$ and then mixing with respect to $h(t)f_{\alpha}(t)$ leads to the a description of the chain in the
 $\mathrm{PK}_{\alpha}(h\cdot f_{\alpha})$ case.

We will demonstrate this, and other relevant results, for another related Markov Chain that will be based on the first part of the following distributional identity  that follows as a specialization of James~\cite[p.8, eq. (2.11)]{JamesLamperti},
\begin{equation}
S_{\alpha,\theta}\overset{d}=S_{\alpha,\theta+1}\times \beta^{-1/\alpha}_{\left(\frac{\theta+\alpha}{\alpha},\frac{1-\alpha}{\alpha}\right)}\overset{d}=S_{\alpha,\theta+\alpha}\times \beta^{-1}_{\theta+\alpha,1-\alpha}.
\label{Jamesfundid}
\end{equation}
The correspondence  between (\ref{PPYfundid}) and (\ref{Jamesfundid}) is also of interest.
Similar to (\ref{PPYfundid}, we will establish an appropriate version of (\ref{Jamesfundid}) in Proposition~\ref{MarkovAlpha}.

There is also the related identity from James~\cite{JamesLamperti}, and also Bertoin and Yor~(\cite{BerYor}, Lemma 6), for $\theta>-\alpha,$
\begin{equation}
\gamma^{1/\alpha}_{\frac{\theta+\alpha}{\alpha}}\overset{d}=\frac{\gamma_{\theta+\alpha}}{S_{\alpha,\theta+\alpha}}\overset{d}=\frac{\gamma_{1+\theta}}{S_{\alpha,\theta}}.
\label{gammaid}
\end{equation}
Our first result is now given.
\begin{prop}\label{keyidgen} Let $\gamma_{1}$ and $\tau_{\alpha}(\zeta)$ be independent,  and use the fact that $\gamma_{1}\overset{d}=\tau_{\alpha}(\gamma_{1/\alpha}),$ and $\gamma_{1/\alpha}=\gamma'_{1}+\gamma_{(1-\alpha)/\alpha}.$ Then for any $\zeta\ge 0,$ fixed or random,
\begin{equation}
\frac{\zeta^{1/\alpha}\gamma_{1}}{\tau_{\alpha}(\zeta)}\overset{d}=
{(\gamma'_{1}+\zeta)}^{1/\alpha}-\zeta^{1/\alpha}.
\label{keydd}
\end{equation}

\begin{equation}
\frac{\gamma_{1}}{\tau_{\alpha}(\zeta)}\overset{d}=
\frac{\tau_{\alpha}\left(\gamma_{\frac{(1-\alpha)}{\alpha}}+\gamma'_{1}\right)}{\tau_{\alpha}(\zeta)}\overset{d}=
{\left(\frac{\gamma'_{1}+\zeta}{\zeta}\right)}^{1/\alpha}-1.
\label{keydd2}
\end{equation}
with conditional density, given $\zeta>0,$
\begin{equation}
\alpha \zeta{(y+1)}^{\alpha-1}{\mbox e}^{-\zeta((y+1)^{\alpha}-1)}.
\label{dengenid}
\end{equation}
\begin{proof}Conditioning on $\zeta,$ the survival distribution of $\gamma_{1}/\tau_{\alpha}(\zeta)$ is given by
$$
\mathrm{P}\left(\frac{\gamma_{1}}{\tau_{\alpha}(\zeta)}>y\right)={\mbox e}^{-\zeta\psi_{\alpha}(y)}=\mathrm{P}(\gamma'_{1}>\zeta((y+1)^{\alpha}-1)),
$$
which, by first conditioning on $\tau_{\alpha}(\zeta),$ and then taking its expected value, is the Laplace transform of $\tau_{\alpha}(\zeta),$ specified by~(\ref{GGexp}). It is clear that the last probability is the survival distribution of the random variable appearing on the right-hand side of (\ref{keydd2}). Since this holds for any positive $\zeta,$~(\ref{keydd}) follows by multiplying by $\zeta^{1/\alpha}$
\end{proof}
\end{prop}
Note setting $\zeta =0,$ in~(\ref{keydd}) gives the well known relation
\begin{equation}
\frac{\gamma_{1}}{S_{\alpha}}\overset{d}=\gamma^{1/\alpha}_{1}
\label{expstable}
\end{equation}
See \cite{Chaumont,DevroyeLinnik, JamesLamperti}. We now discuss a generalization to the case of variables $\gamma_{1}/S_{\alpha,\theta}.$ First by basic calculations involving the stable density, an expression for the survival distribution $\mathbb{S}_{\alpha,\theta}(y):=\mathrm{P}(\gamma_{1}/S_{\alpha,\theta}>y),$ is
\begin{equation}
\mathbb{S}_{\alpha,\theta}(y)=c_{\alpha,\theta}\int_{0}^{\infty}{\mbox e}^{-ys}s^{-\theta}f_{\alpha}(s)ds
\label{basicsurvival}
\end{equation}
The expression in~(\ref{basicsurvival}), in a practical sense, is not very appealing as it involves the generally unknown form of the stable density. However, this expression is important as it arises in formulae connected to stick-breaking and related quantities. In particular the cases $\theta=n\alpha,$ and $\theta=n.$
Now recall from~(\ref{PYlamperti}), that for $\theta>0$
$$
\frac{\tau_{\alpha}(\gamma_{\theta/\alpha})}{\gamma^{1/\alpha}_{\theta/\alpha}}\overset{d}=S_{\alpha,\theta}, {\mbox { and } } \tau_{\alpha}(\gamma_{\theta/\alpha})\overset{d}=\gamma_{\theta},
$$
specializing~ Proposition~\ref{keyidgen}, yields the following useful generalization of ~(\ref{expstable}), and importantly more tractable forms of~(\ref{basicsurvival}).
\begin{prop}For $\theta\ge 0,$ there is the identity
$$
E_{\alpha,\theta}:=\frac{\gamma_{1}}{S_{\alpha,\theta}}\overset{d}=
{\left(\gamma'_{1}+\gamma_{\frac{\theta}{\alpha}}\right)}^{1/\alpha}-\gamma_{\frac{\theta}{\alpha}}^{1/\alpha}.
$$
For $\theta>-\alpha,$ apply~(\ref{PPYfundid},\ref{Jamesfundid}), to obtain
\begin{equation}
E_{\alpha,\theta}\overset{d}=\beta_{\theta+\alpha,1-\alpha}E_{\alpha,\theta+\alpha}\overset{d}
=
\beta^{1/\alpha}_{\left(\frac{\theta+\alpha}{\alpha},\frac{1-\alpha}{\alpha}\right)}
E_{\alpha,1+\theta}.
\label{biascase1}
\end{equation}
\end{prop}
\begin{rem} Setting $\theta=1,$ leads to the interesting identity
$$
\frac{\gamma_{1}}{S_{\alpha,1}}\overset{d}=\gamma^{1/\alpha}_{\frac{1}{\alpha}}\overset{d}=
{\left(\gamma'_{1}+\gamma_{\frac{1}{\alpha}}\right)}^{1/\alpha}-\gamma_{\frac{1}{\alpha}}^{1/\alpha},
$$
where the density of $\gamma^{1/\alpha}_{1/\alpha},$ evaluated at $y,$ is ${\mbox e}^{-y^{\alpha}}=\mathrm{E}[{\mbox e}^{-yS_{\alpha}}].$ Note the first in distribution equality follows from (\ref{gammaid}).
\end{rem}
We next describe the survival distribution which appears in the description of the joint distribution of stick-breaking and related components.
\begin{prop}It follows that for $\theta\ge 0,$
$$\mathbb{S}_{\alpha,\theta}(y):=\mathrm{P}(\gamma_{1}/S_{\alpha,\theta}>y)=
\mathrm{P}(\gamma'_{1}>{(y+\gamma^{1/\alpha}_{\theta/\alpha})}^{\alpha}-\gamma_{\theta/\alpha})
$$
Leading to various representations via change of variables.
\begin{enumerate}
\item[(i)]In particular, $\Gamma(\theta/\alpha)\mathbb{S}_{\alpha,\theta}(y)$ can be expressed as,
\begin{eqnarray}
\nonumber \int_{0}^{\infty}{\mbox e}^{-{(x^{1/\alpha}+y)}^{\alpha}}x^{\theta/\alpha-1}dx &=&\alpha\int_{0}^{\infty}{\mbox e}^{-{(r+y)}^{\alpha}}r^{\theta-1}dr  \\
   &=&\alpha\int_{y}^{\infty}{\mbox e}^{-{s}^{\alpha}}{(s-y)}^{\theta-1}ds.
   \label{Sid}
\end{eqnarray}
\item[(ii)]For $\theta>-\alpha,$ combining (\ref{biascase1}) and (\ref{Sid}), leads to
$$
\mathbb{S}_{\alpha,\theta}(y)=\frac{1}{\Gamma(\frac{\theta+\alpha}{\alpha})}\int_{y^{\alpha}}^{\infty}{\mbox e}^{-r}{(r^{1/\alpha}-y)}^{\theta}dr.
$$\qed
\end{enumerate}
\end{prop}
We close with expressions for the density. By definition, the density of $E_{\alpha,\theta}=\gamma_{1}/S_{\alpha,\theta},$ for $\theta\ge 0$ can be expressed in terms of $f_{\alpha}$ as
$$
f_{E_{\alpha,\theta}}(y)=\frac{\Gamma(\theta+1)}{\Gamma(\frac{\theta}{\alpha}+1)}\int_{0}^{\infty}{\mbox e}^{-sy}s^{1-\theta}f_{\alpha}(s)ds
$$
However, since $1-\mathbb{S}_{\alpha,\theta}(y)$ is the distribution function of
$E_{\alpha,\theta},$ there are other forms of $f_{E_{\alpha,\theta}}(y)$ by taking the derivative of $1-\mathbb{S}_{\alpha,\theta}(y).$
\begin{prop}\label{denE }The density of $E_{\alpha,\theta}=\gamma_{1}/S_{\alpha,\theta},$ for $\theta\ge 0$  can be expressed as follows
\begin{enumerate}
\item[(i)]$
f_{E_{\alpha,\theta}}(y)=\frac{y^{\theta+\alpha-1}\alpha^{2}}{\Gamma(\frac{\theta}{\alpha})}\int_{0}^{\infty}{\mbox e}^{-{(r+1)}^{\alpha}y^{\alpha}}r^{\theta-1}{(1+r)}^{\alpha-1}dr.
$
\item[(ii)]$
f_{E_{\alpha,\theta}}(y)=\frac{y^{\theta+\alpha-1}\alpha}{\Gamma(\frac{\theta}{\alpha})}\int_{1}^{\infty}
{\mbox e}^{-vy^{\alpha}}{(v^{1/\alpha}-1)}^{\theta-1}dv.
$
\item[(iii)]$
f_{E_{\alpha,1+\theta}}(y)=\alpha\mathbb{S}_{\alpha,\theta}(y).$
\end{enumerate}
\end{prop}

\section{Extended $\mathrm{PG}(\alpha,\zeta)$ models}
As discussed in the previous Section, while the class of Pitman-Yor processes, and hence $\mathrm{PD}(\alpha,\theta)$ laws, for the range $\theta\ge 0$ are within the class of $\mathrm{PG}(\alpha,\zeta)$ models, those in the range $-\alpha<\theta<0$ are not. However, as shown in James~\cite{James2010}, one can express a Pitman-Yor process $P_{\alpha,\theta}$ as a normalized process for the entire range $0\leq \alpha<1, \theta>-\alpha,$ as follows
\begin{eqnarray}
\nonumber  P_{\alpha,\theta}(y) &:=& P_{\alpha,1+\theta}(\lambda_{\alpha,\gamma_{\frac{(\theta+\alpha)}{\alpha}}}(y))  \\
   &=&\frac{\tau_{\alpha}(\gamma_{(\theta+\alpha)/\alpha}\mathbb{U}(y)+\gamma_{(1-\alpha)/\alpha}\indic_{\{U_{1}\leq y\}})}
   {\tau_{\alpha}(\gamma_{(\theta+\alpha)/\alpha}+\gamma_{(1-\alpha)/\alpha})}
\label{PYdef}\\
&=&(1-\beta_{1-\alpha,\theta+\alpha})P_{\alpha,\theta+\alpha}(y)+\beta_{1-\alpha,\theta+\alpha}\indic_{(U_{1}\leq y)}.
\label{deleterep}
\end{eqnarray}
where
\begin{equation}
\lambda_{\alpha,\gamma_{\frac{(\theta+\alpha)}{\alpha}}}(y)=
(1-\beta_{(\frac{1-\alpha}{\alpha},\frac{\theta+\alpha}{\alpha})})\mathbb{U}(y)+\beta_{(\frac{1-\alpha}{\alpha},
\frac{\theta+\alpha}{\alpha})}\indic_{(U_{1}\leq y)}
\label{PDsimple})
\end{equation}
is a simple bridge that is independent of the Pitman-Yor process $P_{\alpha,1+\theta}.$  In terms of an explicit construction on the same space we use the form where $P_{\alpha,1+\theta}$ is represented as a $\mathrm{PG}(\alpha,\gamma_{\frac{1+\theta}{\alpha}})$ bridge for  $\theta>-\alpha.$  That is,
$$
P_{\alpha,1+\theta}(y)=\frac{\tau_{\alpha}(\gamma_{(1+\theta)/\alpha} y)}{\tau_{\alpha}(\gamma_{(1+\theta)/\alpha})}{\mbox { and }}P_{\alpha,\theta+\alpha}(y)=\frac{\tau_{\alpha}(\gamma_{(\theta+\alpha)/\alpha} y)}{\tau_{\alpha}(\gamma_{(\theta+\alpha)/\alpha})}.
$$
Furthermore, the stick length, or otherwise the first size-biased pick from $\mathrm{PD}(\alpha,\theta)$ is
$$
\beta_{1-\alpha,\theta+\alpha}=\frac{\tau_{\alpha}(\gamma_{(1-\alpha)/\alpha})}{\tau_{\alpha}(\gamma_{(1+\theta)/\alpha})}=
P_{\alpha,1+\theta}(\beta_{(\frac{1-\alpha}{\alpha},
\frac{\theta+\alpha}{\alpha})}).
$$

Note there is also an interpretation in terms of T-partitions in the sense of Aldous and
Pitman~(\cite{AldousT}, Section 5.1). This follows, since in the general setting for $(L_{t}; 0\leq
t\leq 1),$  a local time process started at 0 governed by $\mathrm{PD}(\alpha,\theta),$ and letting
$\bar{L}_{t}=L_{t}/L_{1},$  there is the equivalence,
$$
P_{\alpha,\theta}(q)\overset{d}=\inf\{t:\bar{L}_{t}\ge q\}, 0\leq
q\leq 1.
$$
This also interprets as the time spent positive up till time $1$ of a generalized Bessel bridge skewed by $q.$ See Pitman and Yor~\cite{PY97} and James~\cite[Section 5]{JamesLamperti}. Further to this, as described in \cite[Section 5]{JamesLamperti} let $B=(B_{t}:t>0)$ denote a Bessel bridge of dimension $2-2\alpha,$ having normalized excursion lengths following a $\mathrm{PD}(\alpha,0)$ law and symmetrize $B$ such that $P(B_{t}>0)=q.$ Let
$A^{+}_{G_{1}}=\int_{0}^{G_{1}}\indic_{(B_{s}>0)}ds$ denote the
time spent positive of $B$ up till time $G_{1},$ which
is the time of the last zero of $B$ before time $1.$ Now from Pitman and Yor~(\cite{PY92},
Theorem 1.3.1), and Perman, Pitman and Yor~(\cite{PPY92}, Theorem
3.8, Lemma 3.11) the following holds under a change of measure from $\mathrm{PD}(\alpha,0)$ to $\mathrm{PD}(\alpha,\theta)$ ,
$$
(A^{+}_{G_{1}},G_{1})\overset{d}=(G_{1}A^{(br)}_{1},G_{1})\overset{d}=(\beta_{\theta+\alpha,1-\alpha}P_{\alpha,\theta+\alpha}(q),\beta_{\theta+\alpha,1-\alpha})
$$
which reflects a decomposition of $A^{+}_{1}:=\int_{0}^{1}\indic_{(B_{s}>0)}ds,$ corresponding to equation~(\ref{deleterep}),
$$
A^{+}_{1}\overset{d}=A^{+}_{G_{1}}+(1-G_{1})\xi_{q}\overset{d}=G_{1}A^{(br)}_{1}+(1-G_{1})\xi_{q}
$$
for $\xi_{q}$ a $\mathrm{Bernoulli}(q)$ variable. Note further from \cite[Section 5]{JamesLamperti} or otherwise from equation, written in terms of Bernoulli variables,
\begin{eqnarray*}
\nonumber P_{\alpha,\theta}(q)& = & P_{\alpha,1+\theta}
\left(\beta_{(\frac{\theta+\alpha}{\alpha},\frac{1-\alpha}{\alpha})}q+
(1-\beta_{(\frac{\theta+\alpha}{\alpha},\frac{1-\alpha}{\alpha})})
\xi_{q}\right)\\
                   & = & \beta_{\theta+\alpha,1-\alpha}P_{\alpha,\theta+\alpha}(q)+(1-\beta_{\theta+\alpha,1-\alpha})\xi_{q},
\end{eqnarray*}
there is the relation
\begin{equation*}
 P_{\alpha,1+\theta}
\left(\beta_{(\frac{\theta+\alpha}{\alpha},\frac{1-\alpha}{\alpha})}q\right)=\beta_{\theta+\alpha,1-\alpha}P_{\alpha,\theta+\alpha}(q).
\label{coagratecorrespondence}
\end{equation*}

Equation~(\ref{deleterep}) is also a special case of the following equation holding for $0\leq \delta\leq 1,$
\begin{equation}
P_{\alpha,\theta}(y)=\beta_{\theta+\alpha\delta,1-\alpha\delta}P_{\alpha,\theta+\alpha\delta}(y)+
(1-\beta_{\theta+\alpha\delta,1-\alpha\delta})P_{\alpha,-\alpha\delta}(y).
\label{Pitfrag}
\end{equation}
Using a recursive argument it is easy to check that~(\ref{Pitfrag}) encodes Pitman's\cite{Pit99} fragmentation operator through the following representation,
\begin{equation}
P_{\alpha,\theta}(y)=\sum_{k=1}^{\infty}(1-W_{k,\delta})\prod_{l=1}^{k-1} W_{l,\delta}P^{(k)}_{\alpha,-\alpha\delta}(y)
\label{Pitfrag2}
\end{equation}
where $(W_{k,\delta})\sim GEM(\alpha\delta,\theta)$ independent of a sequence of iid $\mathrm{PD}(\alpha,-\alpha\delta)$ bridges $(P^{(k)}_{\alpha,-\alpha\delta}).$

These points motivate the definition of  an extended $\mathrm{PG}(\alpha,\zeta)$ class of models. First, throughout, set
\begin{equation}
\varepsilon_{\alpha}\overset{d}=\gamma_{\frac{1-\alpha}{\alpha}},
\label{gammanotation}
\end{equation}

and notice that $\tau_{\alpha}(\varepsilon_{\alpha})\overset{d}=\gamma_{1-\alpha}.$
This choice of distribution is natural as it appears in the distribution of all sized biased jumps generated by the $\mathrm{PK}_{\alpha}(h\cdot f_{\alpha})$ class. See also the discussion in Pitman and Yor~\cite[Section 3]{PY92}.
\subsection{$\mathrm{EPG}(\alpha,\zeta)$-bridge}
Now, as a generalization of~(\ref{PDsimple}), for each $\zeta\ge 0,$ set
$$
q_{\alpha,\zeta}=\frac{\zeta}{\varepsilon_{\alpha}+\zeta}
$$
and define a simple bridge,
\begin{equation}
\lambda_{\alpha,\zeta}(y):=q_{\alpha,\zeta}\mathbb{U}(y)+(1-q_{\alpha,\zeta})\indic
_{\{U_{1}\leq y\}}.
\label{basicbridge}
\end{equation}
Using the same $(\varepsilon_{\alpha},\zeta)$ variables, construct a  $\mathrm{PG}(\alpha,\varepsilon_{\alpha}+\zeta)-\mathrm{bridge} ,$
$Q_{\alpha,\varepsilon_{\alpha}+\zeta}.$ Then we say a random cumulative distribution function $F_{\alpha,\zeta}$ is an extended $\mathrm{PG}(\alpha,\zeta)$-bridge, with law denoted as $\mathrm{EPG}(\alpha,\zeta)$-bridge, if
\begin{equation}
F_{\alpha,\zeta}(y)\overset{d}=Q_{\alpha,\varepsilon_{\alpha}+\zeta}(\lambda_{\alpha,\zeta}(y)), y\in [0,1].
\label{EPGbridge}
\end{equation}
Using
$
\tau_{\alpha}((\varepsilon_{\alpha}+\zeta) \lambda_{\alpha,\zeta}(y))=\tau_{\alpha}(\zeta y+\varepsilon_{\alpha}\indic
_{\{U_{1}\leq y\}} )=\tau_{\alpha}(\zeta y)+\tau_{\alpha}(\varepsilon_{\alpha})\indic_{\{U_{1}\leq y\}}
$
shows that $F_{\alpha,\zeta}$ can be expressed as an exchangeable bridge without dust,
\begin{equation}
F_{\alpha,\zeta}(y)\overset{d}=\frac{\tau_{\alpha}(\zeta)}{\tau_{\alpha}(\varepsilon_{\alpha})+\tau_{\alpha}(\zeta)}Q_{\alpha,\zeta}(y)
+\frac{\tau_{\alpha}(\varepsilon_{\alpha})}{\tau_{\alpha}(\varepsilon_{\alpha})+\tau_{\alpha}(\zeta)}\indic_{\{U_{1}\leq y\}}.
\label{EPGbridge2}
\end{equation}
Note that $\tau_{\alpha}(\varepsilon_{\alpha})\overset{d}=\gamma_{1-\alpha},$ and is, importantly, independent of $Q_{\alpha,\zeta},$ and $\tau_{\alpha}(\zeta).$
Hereafter, set
\begin{equation}
\tilde{P}^{\dagger}_{\alpha,\zeta}=1-Q_{\alpha,\varepsilon_{\alpha}+\zeta}(q_{\alpha,\zeta})=
\frac{\tau_{\alpha}(\varepsilon_{\alpha})}{\tau_{\alpha}(\varepsilon_{\alpha})+\tau_{\alpha}(\zeta)}\overset{d}=\frac{\gamma_{1-\alpha}}
{\gamma_{1-\alpha}+\tau_{\alpha}(\zeta)}.
\label{LittleBigp}
\end{equation}
Note, we shall later verify that $\tilde{P}^{\dagger}_{
\alpha,\zeta}$ is the first size-biased pick from $\mathrm{EPG}(\alpha,\zeta).$
It is now easy to see from~(\ref{EPGbridge2}) that,
$$
F_{\alpha,\gamma_{\frac{\theta+\alpha}{\alpha}}}(y)\overset{d}=P_{\alpha,\theta}(y){\mbox { for }}\theta>-\alpha.
$$
Furthermore,
$
F_{\alpha,0}(y)\overset{d}=\indic_{\{U_{1}\leq y\}}.
$
\begin{rem}Note that we have used a more liberal in distribution definition of $F_{\alpha,\zeta}.$ In subsequent Sections we shall take a version of $F_{\alpha,\zeta}$ that is equivalent to the constructions on the right hand side of the above  equations (\ref{EPGbridge}) and (\ref{EPGbridge2}).
\end{rem}
\subsection{$\mathrm{EPG}(\alpha,\zeta)$ masses}\label{sampling}
If $(P_{i})\in\mathcal{P}_{\infty}$ are the ranked probability masses of $F_{\alpha,\zeta},$ then the law of this sequence is denoted as $\mathrm{EPG}(\alpha,\zeta),$ and from~(\ref{EPGbridge2}), satisfies
\begin{equation}
(P_{i})\overset{d}=\mathrm{Rank}((1-\tilde{P}_{1,\zeta})(Q_{k}),\tilde{P}^{\dagger}_{1,\zeta})\sim \mathrm{EPG}(\alpha,\zeta)
\label{EPGsequenceInsert}
\end{equation}
where $(Q_{k})\sim \mathrm{PG}(\alpha,\zeta)$ and $\tilde{P}^{\dagger}_{1,\zeta}={\tau_{\alpha}(\varepsilon_{\alpha})}/{(\tau_{\alpha}(\varepsilon_{\alpha})+\tau_{\alpha}(\zeta))}.$
Additionally, using the representation in $(\ref{EPGbridge}),$ it follows that as in~(\ref{indicator}) we can construct indicator variables
\begin{equation}
I_{k}=\indic_{\{\lambda^{-1}_{\alpha,\zeta}(U'_{k})=U_{1}\}}\sim \mathrm{Bernoulli} (1-q_{\alpha,\zeta}),
\label{indicatorzeta}
\end{equation}
where the $(U'_{k})$ are the atoms of the bridge $Q_{\alpha,\varepsilon_{\alpha}+\zeta},$ such that obviously from the construction in $(\ref{EPGbridge}),$
\begin{equation}
(P_{i})\overset{d}=\mathrm{Rank}((Q_{l}:I_{l}=0);\sum_{\{k:I_{k}=1\}}Q_{k} )\sim \mathrm{EPG}(\alpha,\zeta)
\label{EPGcoag}
\end{equation}
when $(Q_{k})\sim \mathrm{PG}(\alpha,\varepsilon_{\alpha}+\zeta).$

Using the same mixing distribution for $\zeta,$ (\ref{zetamix}), it follows that
$$
\mathrm{EPG}(\alpha,\gamma_{1})=\mathrm{PD}(\alpha,0)
$$
is the canonical distribution for the $\mathrm{EPG}(\alpha,\zeta)$ model. Also
$$
\mathrm{EPG}(\alpha,\gamma_{1+\theta/\alpha})=\mathrm{PD}(\alpha,\theta) {\mbox { for }}\theta>-\alpha.
$$
When $\zeta=0,$ $\mathrm{EPG}(\alpha,0)$ is the distribution that assigns point mass to the vector $(1,0,0,\ldots)\in \mathcal{P}_{\infty}.$ Note that by taking limits, $\theta\rightarrow -\alpha,$ we can define $\mathrm{PD}(\alpha,-\alpha):=\mathrm{EPG}(\alpha,0).$
\subsection{Sampling from an $\mathrm{EPG}(\alpha,\zeta)$-bridge}\label{samplingcoag}
Before continuing with the structural properties of $\mathrm{EPG}(\alpha,\zeta),$ we look at  the classical Bayesian framework of sampling from a discrete random probability measure. Notice that conditional on $(\varepsilon_{\alpha},\zeta, \lambda_{\alpha,\zeta}),$ (\ref{EPGbridge}) shows that $F_{\alpha,\zeta}$ is a $\mathrm{NGG}(\alpha,\varepsilon_{\alpha}+\zeta)$ random cumulative distribution function with base measure $\lambda_{\alpha,\zeta}.$
In addition, suppose that $X_{1},\ldots,X_{n}|F_{\alpha,\zeta}$ are iid $F_{\alpha,\zeta},$ then
\begin{equation}
X_{i}\overset{d}=F^{-1}_{\alpha,\zeta}(U'_{i})\overset{d}=\lambda^{-1}_{\alpha,\zeta}\circ Q^{-1}_{\alpha,\varepsilon_{\alpha}+\zeta}(U'_{i})
\label{quantileX}
\end{equation}
where $(U'_{i})$ are iid Uniform$[0,1]$ independent of $F_{\alpha,\zeta}.$ Now let $\hat{U}_{j}$ for $j=1,\dots, \hat{K}_{n}$ denote the unique iid Uniform$[0,1]$ values obtained by sampling from $F_{\alpha,\zeta}.$  It follows that ${\{B_{1},\ldots,B_{\hat{K}_{n}}\}},$ defined by
$$
B_{j}={\{i:F^{-1}_{\alpha,\zeta}(U'_{i})=\hat{U}_{j}\}},
$$
for $j=1,\dots, \hat{K}_{n}$ forms a random partition of $[n],$ whose law is an $\mathrm{EPG}(\alpha,\zeta)$-{EPPF}. Since it is known how to sample a random partition of $[n]$ from a $\mathrm{PG}(\alpha,\zeta)$-EPPF, equation~(\ref{quantileX}) shows how one can obtain ${\{B_{1},\ldots,B_{\hat{K}_{n}}\}},$ via a practical sampling scheme involving several stages.
\begin{enumerate}
\item[(i)] First draw $(\varepsilon_{\alpha},\zeta),$ or directly $(\varepsilon_{\alpha}+\zeta,\varepsilon_{\alpha}/(\varepsilon_{\alpha}+\zeta))$
\item[(ii)] Draw a random partition of $[n],$ ${\{A_{1},\ldots,A_{K_{n}}\}}$ from a $\mathrm{PG}(\alpha,\varepsilon_{\alpha}+\zeta)$-EPPF.
Practically, this is obtained by a Generalized Chinese restaurant process sampling procedure,which does not actually involve evaluating $Q^{-1}_{\alpha,\varepsilon_{\alpha}+\zeta}(U'_{i}).$ Furthermore, conditional on
$\varepsilon_{\alpha}+\zeta=t,$ one can use the EPPF of an $\mathrm{NGG}(\alpha,t)$ process.
\item[(iii)] Draw $(U^{*}_{1},\ldots,U^{*}_{K_{n}})$ iid Uniform$[0,1]$ variables.
\item[(iv)] Recall that $U_{1}$ is the atom of $\lambda_{\alpha,\zeta},$ and has a Uniform$[0,1]$ distribution. The partition
${\{B_{1},\ldots,B_{\hat{K}_{n}}\}},$ is obtained as follows. Blocks of ${\{A_{1},\ldots,A_{K_{n}}\}},$ are merged into a set $B'_{1}$ defined as
$$
B'_{1}={\{A_{i}:\lambda^{-1}_{\alpha,\zeta}(U^{*}_{i})=U_{1}\}},
$$
if $B'_{1}$ is not empty, set $B_{1}=B'_{1}, $the remaining ${K}_{n}-|B'_{1}|=\hat{K}_{n}-1$ blocks of ${\{A_{1},\ldots,A_{K_{n}}\}},$ are relabeled $B_{2},\ldots,B_{\hat{K}_{n}}.$ If $B'_{1}=\emptyset,$ $\hat{K}_{n}=K_{n}$
and one sets $B_{k}=A_{k}$ for $k=1,\ldots,K_{n}.$
\item[(v)] Note in the Pitman-Yor case, $\lambda_{\alpha,\gamma_{(\theta+\alpha)/\alpha}}$ and $P_{\alpha,1+\theta}$ are independent, so conditioning as in [(ii)] is not necessary.
\end{enumerate}
If an application requires the use of a discrete (non-random) base measure, $H,$  such as described in \cite{Buntine,Teh,Wood1,Wood2} for models arising in Machine Learning, say of the form
$$
H(y)=\sum_{k=1}^{N}q_{k}\indic_{\{w_{k}\leq y\}},
$$
for fixed points $(w_{k})$ and probability weights $(q_{k})$ summing to $1.$ First compose $H$ with $F_{\alpha,\zeta},$ i.e. $F_{\alpha,\zeta}\circ H,$ which results in $\mathbb{E}[F_{\alpha,\zeta}(H(y)]=\mathbb{U}\circ H(y)=H(y).$ Then a partition of $[n]$ generated from $F_{\alpha,\zeta} \circ H$ is obtained by adding the following additional step to the procedure described above.
\begin{enumerate}
\item[(D)]\textsc{Step for discrete base measures.} Generate sets $C_{1},\ldots,C_{N}$ by
$$
C_{k}=\{B_{j}:H^{-1}(\hat{U}_{j})=w_{k}\}.
$$
The non-empty sets form a partition of $[n]$
\end{enumerate}
\begin{rem}\label{Binomial}
Note that conditional on $K_{n}$ and $\varepsilon_{\alpha}/(\varepsilon_{\alpha}+\zeta),$
the chance that each block $A_{i}$ is in $B'_{1}$ follows an independent Bernoulli distribution with success probability $1-q_{\alpha,\zeta}=\varepsilon_{\alpha}/(\varepsilon_{\alpha}+\zeta).$ Hence it follows that  $|B'_{1}|,$ the size of $B'_{1},$ follows a Binomial distribution with parameters $K_{n}=k$ and $1-q_{\alpha,\zeta}$ as specified.
\end{rem}
\subsection{First comments on BDGM style coagulation/fragmentation}
The sampling scheme in Section~\ref{samplingcoag} can be seen as a coagulation operation on partitions of $[n]$ as described in~Bertoin\cite[Section 4.2.1]{BerFrag}, except here the coag operator, induced by the simple bridge $\lambda_{\alpha,\zeta},$ is in general not independent of the input which is a partition ${\{A_{1},\ldots,A_{K_{n}}\}},$ having the distribution of a $\mathrm{PG}(\alpha,\varepsilon_{\alpha}+\zeta)-\mathrm{EPPF}.$ The resulting output ${\{B_{1},\ldots,B_{\hat{K}_{n}}\}},$ is shown to have the distribution of a $\mathrm{EPG}(\alpha,\zeta)-\mathrm{EPPF}.$ This coagulation operation on $\mathbb{N}={\{1,2,3,\ldots\}}$ expressed here in the general $\mathrm{EPG}(\alpha,\zeta)$ setting is, by Kingman's correspondence, in bijection to the coagulation operation on $\mathcal{P}_{\infty}$ that was described in~
Bertoin and Goldschmidt~\cite{BertoinGoldschmidt2004} and Dong, Goldschmidt, and
Martin(DGM)~\cite{Dong2006} (denoted as BDGM), for the case of $\mathrm{PD}(\alpha,\theta)$ models. That is the coagulation of $\mathrm{PD}(\alpha,1+\theta)$ partition of $[n],$ or $\mathbb{N},$ by an independent bridge with $q_{\alpha,\gamma_{(\theta+\alpha)/\alpha}}\overset{d}=\beta_{(\theta+\alpha)/\alpha,(1-\alpha)/\alpha}$ leads to a partition derived from a $\mathrm{PD}(\alpha,\theta)$ distribution. Their dual $\mathrm{PD}(\alpha,1-\alpha)$ fragmentation operator can be encoded by the following bridge equation
$$
P_{\alpha,1+\theta}(y)=(1-\beta_{1-\alpha,\theta+\alpha})P_{\alpha,\theta+\alpha}(y)+\beta_{1-\alpha,\theta+\alpha}P_{\alpha,1-\alpha}(y)
$$
where all quantities on the right hand side are independent and
$$
P_{\alpha,1+\theta}(\beta_{\left(\frac{1-\alpha}{\alpha},\frac{\theta+\alpha}{\alpha}\right)})= \beta_{1-\alpha,\theta+\alpha},
$$
is the first size biased pick from a $\mathrm{PD}(\alpha,\theta)$ model.

We find it interesting to elaborate on the partition of $[n]$ based viewpoint of the BDGM coagulation operator as described in the previous Section.  Note that Remark~\ref{Binomial} shows that in step[(iv)] one is performing some sort of $p_{\alpha,\zeta}=(1-q_{\alpha,\zeta})$ merger in the language of
Berestycki~\cite[p. 69-70]{Berestycki}. That is $j$ of the ${\{A_{1},\ldots,A_{K_{n}}\}},$ blocks are said to coalesce if $|B'_{1}|=j\ge 2.$ The next result, which follows from elementary calculations,  describes some more details about the distribution of $|B'_{1}|$ in the $\mathrm{PD}(\alpha,\theta)$ setting.
Write
$$
\rho_{\alpha,\theta}(p)=\frac{\Gamma(\frac{1+\theta}{\alpha})}
{\Gamma(\frac{\theta+\alpha}{\alpha})\Gamma(\frac{1-\alpha}{\alpha})}p^{1/\alpha}{(1-p)}^{\theta/\alpha}
$$

\begin{prop}In the $\mathrm{PD}(\alpha,\theta)$ setting of Section~\ref{samplingcoag}, given blocks  ${\{A_{1},\ldots,A_{K_{n}}\}},$ formed by a $\mathrm{PD}(\alpha,1+\theta)-\mathrm{EPPF},$ in particular given $K_{n}=b\ge 2,$ $2 \leq j\leq b$ blocks are said to coalesce into a single block $B'_{1}$ if $|B'_{1}|=j,$ otherwise no blocks coalesce if  $|B'_{1}|=0$ or $1.$ The general distribution of $|B'_{1}|,$ given $K_{n}=b$ is a  mixed Binomial distribution with probability mass function
$$
p_{\alpha,\theta}(j|b)={b \choose j}\int_{0}^{1}p^{j-2}(1-p)^{b-j}\rho_{\alpha,\theta}(p)dp$$
which can be expressed as
$$
p_{\alpha,\theta}(j|b)={b \choose j}\frac{\Gamma(\frac{1+\theta}{\alpha})
\Gamma(\frac{\theta+\alpha}{\alpha}+b-j)\Gamma(\frac{1}{\alpha}+j-1)}
{\Gamma(\frac{\theta+\alpha}{\alpha})\Gamma(\frac{1-\alpha}{\alpha})
\Gamma(\frac{1+\theta}{\alpha}+b)}
$$

\begin{enumerate}
\item[(i)] In the Brownian cases, $\mathrm{PD}(1/2,\theta),$ $\theta>-1/2,$
$$
p_{1/2,\theta}(j|b)=\frac{(2\theta+1)(2\theta+b-j)!b!}{(2\theta+b+1)!(b-j)!}
$$
In particular $p_{1/2,0}(j|b)=1/(b+1)$ is the discrete uniform distribution on ${\{0,\ldots,b\}},$ and
$p_{1/2,1/2}(j|b)=2(b+1-j)/[(b+1)(b+2)].$
\item[(ii)] In the limiting Dirichlet case, $\mathrm{PD}(0,\theta),$
$$
p_{0,\theta}(j|b)={b \choose j}p^{j}_{\theta}{(1-p_{\theta})}^{b-j}
$$
is a proper Binomial distribution with success probability $p_{\theta}=1/(\theta+1),$ for $\theta>0.$
\end{enumerate}
\end{prop}
Note that when $\theta=0,$ corresponding to a $\mathrm{PD}(\alpha,0)$ model, there is not much simplification for general $\alpha.$ The dual $\mathrm{PD}(\alpha,1-\alpha)$ fragmentation operation of \cite{BertoinGoldschmidt2004,Dong2006} in this case is encoded in the equation
$$
P_{\alpha,1}(y)=(1-\beta_{1-\alpha,\alpha})P_{\alpha,\alpha}+\beta_{1-\alpha,\alpha}P_{\alpha,1-\alpha}(y)
$$
where the variables on the right hand side are independent and,
$$
P_{\alpha,1}(\beta_{\left(\frac{1-\alpha}{\alpha},1\right)})= \beta_{1-\alpha,\alpha},
$$
is the first size biased pick from  $\mathrm{PD}(\alpha,0).$ Their recursive procedure, in general employing a different rule $p_{\alpha,\theta+k-1}$ varying at each step $k,$ leads to the dual Markov Chain with states represented by the laws
$$
\mathrm{PD}(\alpha,0),\mathrm{PD}(\alpha,1),\ldots,\mathrm{PD}(\alpha,n),\ldots,
$$
\subsection{A connection to the Aldous $\beta$-splitting model?}
The case $\theta=1-2\alpha,$ corresponds to a coagulation that results in a $\mathrm{PD}(\alpha,1-2\alpha)$ model. Rather interestingly, setting
$$\beta=\frac{(1-\alpha)}{\alpha}-1,$$
one can see that the probability mass function of $|B'_{1}|,$ given $K_{n}=b$ and the event
 ${\{|B'_{1}|\neq 0{\mbox { or b}}\}},$ equates to the splitting kernel in the $\beta$-splitting model of Aldous~\cite{AldousClad} for the range $\beta>-1.$ That is using the notation in~\cite[p.1824]{Haas},
 $$
 p_{\alpha,1-2\alpha}(j|b)\propto \tilde{q}^{\mathrm{Aldous}-\beta}_{b}(j) {\mbox { for }}  1\leq j\leq b-1.
 $$
 The Yule model case of $\beta=0$ corresponds to a $\mathrm{PD}(1/2,0)$ model with $p_{1/2,0}(j|b)=1/(b+1).$
 The \emph{symmetric random trie} case $\beta\rightarrow\infty$ corresponds to a $\mathrm{PD}(0,1)$ model with $p_{0,1}(j|b)={b \choose j}{(1/2)}^{b}.$
 In particular, for $b\ge 2,$
 $$
 \tilde{q}^{\mathrm{Aldous}-0}_{b}(j)=\frac{1}{b-1}{\mbox { and }} \tilde{q}^{\mathrm{Aldous}-\infty}_{b}(j)={b \choose j}\frac{1}{2^{b}-2}
 $$

\begin{rem}
See \cite{FordD,Haas,McCullagh} for more on the $\beta$-splitting model and its connection to fragmentation trees.  In addition
See \cite[Proposition 14, and page 2004-2005]{PitmanWinkel}, \cite[Section 3.3,and Proposition 27]{PitmanWinkel2}, \cite[Proposition 18]{Haas} and \cite[p. 1738]{Dong2006}
which  describe relations to the $\alpha$-model of Ford~\cite{FordD} and the Brownian CRT of Aldous~\cite{AldousCRTI,AldousCRTIII}.
\end{rem}
\begin{rem}
What is also interesting is that in the above mentioned references the $\beta$-splitting rule  has been naturally linked to fragmentation of trees, and furthermore there are some natural connections made to the $\mathrm{PD}(\alpha,1-\alpha)$ fragmentation operator of \cite{BertoinGoldschmidt2004,Dong2006} rather than their dual coagulation operator.  Here one sees that this \textsc{$\beta$-splitting rule} is being used in a coagulation/coalescent context as a \textsc{merging rule} to coalesce some blocks ${\{A_{1},\ldots,A_{K_{n}}\}},$  of a partition of $[n],$ into a single block $B'_{1}.$
\end{rem}
Notice in the case of the Brownian CRT (embedded in Brownian excursion) \cite[p. 1738]{Dong2006}, which corresponds to the case where $\alpha=1/2,$ and $\theta=1/2$ and otherwise relates to the $\mathrm{PD}(1/2,1/2)$ fragmentation operation in~\cite{Dong2006}, with states having laws, (note that the chain induced by the dual coagulation operation is read right to left),
$$
\mathrm{PD}(1/2,1/2),\mathrm{PD}(1/2,3/2),\ldots,\mathrm{PD}(1/2,n-1/2),\ldots,
$$
is in fact obviously not a special case of the $\mathrm{PD}(\alpha,1-2\alpha)$ model we are discussing. Rather, it is a special case of a $\mathrm{PD}(\alpha,\alpha)$ or $\mathrm{PD}(\alpha,1-\alpha)$ model which are the same in the Brownian Bridge case. Interestingly, in the general $\mathrm{PD}(\alpha,1-2\alpha)$ setting the $\mathrm{PD}(\alpha,1-\alpha)$ fragmentation operator of~\cite{BertoinGoldschmidt2004,Dong2006} is encoded by the equation
$$
P_{\alpha,2-2\alpha}=\beta_{1-\alpha,1-\alpha}P_{\alpha,1-\alpha}(y)+(1-\beta_{1-\alpha,1-\alpha})P'_{\alpha,1-\alpha}(y)
$$
where, $P_{\alpha,1-\alpha}$ and $P'_{\alpha,1-\alpha}$ are independent $\mathrm{PD}(\alpha,1-\alpha)$ bridges, and if we consider $P'_{\alpha,1-\alpha}$ to contain the fragmenting mass partition, then
$$
P_{\alpha,2-2\alpha}(\beta_{\left(\frac{1-\alpha}{\alpha},\frac{1-\alpha}{\alpha}\right)})= \beta_{1-\alpha,1-\alpha}.
$$
is $1$ minus the first size-biased pick from  $\mathrm{PD}(\alpha,1-2\alpha),$ which of course marginally has the same symmetric $\mathrm{Beta}(1-\alpha,1-\alpha)$ distribution. For $0\leq \alpha<1$ the chain induced by the dual operations ~\cite{BertoinGoldschmidt2004,Dong2006} is described by
$$
\mathrm{PD}(\alpha,1-2\alpha),\mathrm{PD}(\alpha,2-2\alpha),\ldots,\mathrm{PD}(\alpha,n-2\alpha),\ldots,
$$
When $\alpha=1/2,$ which agrees with the Yule model $\beta=0,$ we see a fragmentation chain starting from the ranked excursions lengths, that is $\mathrm{PD}(1/2,0),$ induced by a process behaving like Brownian Motion,
$$
\mathrm{PD}(1/2,0),\mathrm{PD}(1/2,1),\ldots,\mathrm{PD}(1/2,n-1),\ldots,
$$
When $\alpha=0$, which is the $\mathrm{PD}(0,1)$ Dirichlet case and agrees with the \emph{symmetric random trie} case $\beta\rightarrow\infty,$ one has
$$
\mathrm{PD}(0,1),\mathrm{PD}(0,2),\ldots,\mathrm{PD}(0,n),\ldots,
$$
\begin{rem}\label{conject}Note that, in the sense of \cite{Pit99} and \cite[p.190]{BerFrag}, $p_{\alpha,\theta}$ for $2\leq j\leq b,$
is proportional to the coagulation rates of a $\Lambda$ coalescent where  the measure $\Lambda$ is a
$\mathrm{Beta}((1+\alpha)/\alpha,(\theta+\alpha)/\alpha)$ distribution.
\end{rem}
\subsection{A question/open problem}
Recall from \cite{Berestycki,BBS} that the Beta-coalescent is usually specified for $\Lambda$ corresponding to $\mathrm{Beta}(2-\delta,\delta)$, for
$0<\delta<2.$  However, for $0<\delta\leq 1,$ one can reparametrixe this in terms of $\alpha,$  to see that this is exactly a $\mathrm{Beta}(1+\alpha,1-\alpha)$ distribution. Furthermore, a variable $\beta_{1-\alpha,1+\alpha}$ is a size biased pick from a $\mathrm{PD}(\alpha,1)$ distribution. In regards to Remark~\ref{conject}, apply the natural operation
$$
P_{\alpha,1+\theta+2\alpha}(\beta_{((1+\alpha)/\alpha,(\theta+\alpha)/\alpha})=\beta_{1+\alpha,\theta+\alpha}
$$

Notice that, besides the $\mathrm{PD}(0,\theta)$ case with $\alpha=0,$ there is only one choice where $\theta+\alpha=1-\alpha,$ and hence can be interpreted in terms of a size-biased pick from a
$\mathrm{PD}(\alpha,\theta)$ distribution. It is of course $\theta=1-2\alpha,$ which agrees with Aldous's $\beta$-splitting case described above. Coincidental perhaps, but this is suggestive of a bijection betwween a coalescent using Aldous's $\beta$-splitting rule and the Beta-coalescent in the sense of \cite{Pit99}, or perhaps in another sense. The question is,; Can one formally describe such processes? To add weight to this, lets look at a set of nested $\mathrm{PD}$ equations, which we believe always tell some sort of story. There is
\begin{eqnarray*}
\nonumber P_{\alpha,1-2\alpha}(y)&=&  P_{\alpha,2-2\alpha}\left(\beta_{\left(\frac{1-\alpha}{\alpha},\frac{1-\alpha}{\alpha}\right)}\mathbb{U}(y)+(1-\beta_{\left(\frac{1-\alpha}{\alpha},\frac{1-\alpha}{\alpha}\right)})\indic_{\{U_{1}\leq y\}}\right)\\
                   & = & \beta_{1-\alpha,1-\alpha}P_{\alpha,1-\alpha}(y)+(1-\beta_{1-\alpha,1-\alpha})\indic_{\{U_{1}\leq y\}},
\end{eqnarray*}
 where
  \begin{eqnarray*}
\nonumber P_{\alpha,1-\alpha}(y)&=&  P_{\alpha,2-\alpha}\left(\beta_{\left(\frac{1}{\alpha},\frac{1-\alpha}{\alpha}\right)}\mathbb{U}(y)+(1-\beta_{\left(\frac{1}{\alpha},\frac{1-\alpha}{\alpha}\right)})\indic_{\{U_{2}\leq y\}}\right)\\
                   & = & \beta_{1,1-\alpha}P_{\alpha,1}(y)+(1-\beta_{1,1-\alpha})\indic_{\{U_{2}\leq y\}},
\end{eqnarray*}
and where
  \begin{eqnarray*}
\nonumber P_{\alpha,1}(y)&=&  P_{\alpha,2}\left(\beta_{\left(\frac{1+\alpha}{\alpha},\frac{1-\alpha}{\alpha}\right)}\mathbb{U}(y)+(1-\beta_{\left(\frac{1+\alpha}{\alpha},\frac{1-\alpha}{\alpha}\right)})\indic_{\{U_{3}\leq y\}}\right)\\
                   & = & \beta_{1+\alpha,1-\alpha}P_{\alpha,1+\alpha}(y)+(1-\beta_{1+\alpha,1-\alpha})\indic_{\{U_{3}\leq y\}}.
\end{eqnarray*}

Starting with the last equation above and reading backwards these equations can be formed by consecutive insertion of size biased picks as described in~\cite[Proposition 35, Section 6.1]{PY97}.
Note in the Dirichlet case, $\alpha=0,$ all 3 equations look the same, due to invariance property of the Dirichlet process under size biased sampling, except for the atom labels. Hence we see for $\ell=1,2,\ldots $  there is the relation, for $\tilde{U}_{l}\sim\mathrm{Uniform}[0,1],$
  \begin{eqnarray*}
\nonumber P^{(\ell)}_{0,1}(y)&=&  P^{(\ell)}_{0,2}\left(1/2\mathbb{U}(y)+1/2\indic_{\{{U}_{\ell}\leq y\}}\right)\\
                   & = & \tilde{U}_{\ell}P^{(\ell+1)}_{0,1}(y)+(1-\tilde{U}_{\ell})\indic_{\{U_{\ell}\leq y\}}.
\end{eqnarray*}
This seems to be the easiest and most promising. Suggesting a relation between cutting according to the \emph{symmetric random trie} case $\beta\rightarrow\infty,$ and the U-Coalescent. See Goldschmidt and Martin~\cite{Gold} for ideas in this case and Abraham and Delmas~\cite{Abraham1,Abraham2} for the cases of  $\alpha\in [1/2,1)$ and specialized to $\alpha=1/2,$ that is our $\mathrm{PD}(1/2,0)$ Brownian setting.

Of course the relevant variables, as described above, under the $\beta$-splitting scheme and those under the $\mathrm{Beta}(1+\alpha,1-\alpha)
$-Coalescent are not independent.  Confirming, in terms of conditional probability, that one could start a description of a $\mathrm{Beta}(1+\alpha,1-\alpha)
$-Coalescent, in the sense of Pitman~\cite{Pit99} from the $\beta$-splitting scheme. However, in terms of the question raised,  we are looking for the precise interpretation/description of this, within that or a similar continuous time context.  Notice, see Section 6, there are the exact equalities,
$$
S^{-\alpha}_{\alpha,1}=S^{-\alpha}_{\alpha,1+\alpha}\times  \beta^{\alpha}_{1+\alpha,1-\alpha}
=S^{-\alpha}_{\alpha,2}\times
\beta_{\left(\frac{1+\alpha}{\alpha},\frac{1-\alpha}{\alpha}\right)}
$$
and there is
$$
\beta_{\left(\frac{1+\alpha}{\alpha},\frac{1-\alpha}{\alpha}\right)}=\frac{S^{-\alpha}_{\alpha,1}}{S^{-\alpha}_{\alpha,2}};
\beta^{\alpha}_{1+\alpha,1-\alpha}=\frac{S^{-\alpha}_{\alpha,1}}{S^{-\alpha}_{\alpha,1+\alpha}},{\mbox { and }}
P_{\alpha,2}(\beta_{\left(\frac{1+\alpha}{\alpha},\frac{1-\alpha}{\alpha}\right)})=\beta_{1+\alpha,1-\alpha}
$$
with similar equations for $S_{\alpha,1-\alpha}$ and $S_{\alpha,1-2\alpha}$, which can be interpreted in terms of $\alpha$-diversities or local times.
Now from~James~\cite{JamesLamperti}, the density of the random variable
$X_{\alpha,1}\overset{d}=S_{\alpha}/S_{\alpha,1},$ is given explicitly as,
$$
\Delta_{\alpha,1}(x)=\frac{1}{\pi}\frac{\sin\left(\frac{1}{\alpha}\arctan\left(\frac{\sin(\pi
\alpha)}{\cos(\pi\alpha)+x^{\alpha}}\right)\right)}
{{[x^{2\alpha}+2x^{\alpha}\cos(\alpha
\pi)+1]}^{\frac{1}{2\alpha}}}.
$$
and the density of $P_{\alpha,1}(q),$  see also \cite{JLP}, is also explicitly given by
\begin{equation}
\Omega_{\alpha,1}(y|q)=\frac{\Delta_{\alpha,1}(
{(\frac{1-q}{q})}^{1/\alpha}\frac{y}{1-y}
)}{(1-y)q^{1/\alpha}}
\label{density1}
\end{equation}
See James~\cite[Proposition 4.3]{JamesLamperti} for the appearance of the variable $X_{\alpha,1}$ in this context.
This easily leads to the following result.
\begin{prop} Consider the relation
$P_{\alpha,2}(\beta_{\left(\frac{1+\alpha}{\alpha},\frac{1-\alpha}{\alpha}\right)})=\beta_{1+\alpha,1-\alpha}$ where $\beta_{\left(\frac{1+\alpha}{\alpha},\frac{1-\alpha}{\alpha}\right)}$ is independent of $P_{\alpha,2}.$ Then conditional on $\beta_{\left(\frac{1+\alpha}{\alpha},\frac{1-\alpha}{\alpha}\right)}=q,$  $\beta_{1+\alpha,1-\alpha}$ can be expressed as
$$
P_{\alpha,2}(q)=UP_{\alpha,1}(q)+(1-U)P'_{\alpha,1}(q)
$$
where $U:=\tau_{\alpha}(\gamma_{1/\alpha})/\tau_{\alpha}(\gamma_{2/\alpha}),$ and hence $U$ is a $\mathrm{Uniform}[0,1]$ variable independent of the iid $\mathrm{PD}(\alpha,1)$ bridges $P_{\alpha,1}(q), P'_{\alpha,1}(q)$ having common explicit density for fixed $q$ given by~(\ref{density1}). Note when $\alpha=0,$ the Dirichlet case, $q=1/2$ and $P_{0,2}(1/2)$ has a $\mathrm{Uniform}[0,1]$ distribution.
\end{prop}

\section{Which $\mathrm{PK}_{\alpha}(h\cdot f_{\alpha})$ class do $\mathrm{EPG}(\alpha,\zeta)$ models correspond to?}
 By construction, the class of distributions $\mathrm{EPG}(\alpha,\zeta),$ should be larger than $\mathrm{PG}(\alpha,\zeta).$ But can one describe a more precise relationship? Furthermore, it is evident the laws constitute a  sub-class of  $\mathrm{PK}_{\alpha}(h\cdot f_{\alpha}).$ However, identifying precisely what this is is not so obvious. Both issues are important, we shall take up the latter one in this Section and deal with the former one in the next Section.
What is needed is an identification of the variable defined in~(\ref{alphadiversity})
The previous Section on sampling and Remark~\ref{Binomial} aids in obtaining the following result.
\begin{prop}\label{MarkovAlpha} Let $\hat{K}_{n}$ denote the random number of blocks of a partition of $[n]$ following a $\mathrm{EPG}(\alpha,\zeta)$-EPPF, that is derived from a version of the construction in~(\ref{EPGbridge}), for $y\in [0,1],$
\begin{eqnarray}
F_{\alpha,\zeta}(y)& = & Q_{\alpha,\varepsilon_{\alpha}+\zeta}(\lambda_{\alpha,\zeta}(y))\\
                   & = & (1-\tilde{P}^{\dagger}_{\alpha,\zeta})Q_{\alpha,\zeta}(y)+\tilde{P}^{\dagger}_{\alpha,\zeta}\indic_{\{U_{1}\leq y\}}
\end{eqnarray}
Following~\cite[Section 6.1]{Pit02}, write $A_{n}\simeq B_{n}$ if $A_{n}/B_{n}\rightarrow 1$ almost surely as $n\rightarrow \infty.$ Then
$$
\hat{K}_{n}\simeq n^{\alpha}\hat{Z} {\mbox { as }} n\rightarrow \infty
$$
where $\hat{Z}=\hat{T}^{-\alpha}$ is the $\alpha$-diversity of an $\mathrm{EPG}(\alpha,\zeta)$ partition, satisfying~(\ref{alphadiversity}) and
\begin{equation}
\hat{T}=\frac{\tau_{\alpha}(\varepsilon_{\alpha}+\zeta)}{\zeta^{1/\alpha}}.
\label{EPT}
\end{equation}
Furthermore if $Z=T^{-\alpha}$ denotes the $\alpha$-diversity of  $\mathrm{PG}(\alpha,\varepsilon_{\alpha}+\zeta)$, constructed on the same space,
then it is known from Pitman and Yor \cite[p. 877-878]{PY97} that,
$$
T=\frac{\tau_{\alpha}(\varepsilon_{\alpha}+\zeta)}{{(\varepsilon_{\alpha}+\zeta)}^{1/\alpha}}.
$$
These points lead to the decomposition
\begin{equation}
\hat{T}=T\times q_{\alpha,\zeta}^{-1/\alpha}.
\label{LJdecomp}
\end{equation}
which gives an interpretation of an otherwise obvious exact equality. In addition, this shows that $\hat{T}$ can also be expressed in terms of the $\alpha$-diversity of a $\mathrm{PG}(\alpha,\zeta)$ model and
$\tilde{P}^{\dagger}_{\alpha,\zeta}$ as follows
\begin{equation}
\hat{T}=\frac{\tau_{\alpha}(\varepsilon_{\alpha}+\zeta)}{{(\varepsilon_{\alpha}+\zeta)}^{1/\alpha}}\times q_{\alpha,\zeta}^{-1/\alpha}=
\frac{\tau_{\alpha}(\zeta)}{\zeta^{1/\alpha}}\times {(1-\tilde{P}^{\dagger}_{\alpha,\zeta})}^{-1},
\label{diversitycorrespondence}
\end{equation}
with $Q_{\alpha,\varepsilon_{\alpha}+\zeta}(q_{\alpha,\zeta})=1-\tilde{P}^{\dagger}_{\alpha,\zeta}$
\end{prop}
\begin{proof}
From the previous Section and Remark~\ref{Binomial}, unless $\zeta=0,$ for large $n$ it suffices to use the relation,
$$
\hat{K}_{n}=K_{n}-|B'_{1}|+1
$$
where $K_{n}$ are the number of blocks of a $\mathrm{PG}(\alpha,\varepsilon_{\alpha}+\zeta)$ partition of $[n].$ Operating conditionally on $(\varepsilon_{\alpha},\zeta)$ and again setting $1-q_{\alpha,\zeta}=\varepsilon_{\alpha}/(\varepsilon_{\alpha}+\zeta),$ one can write
$$
K_{n}-|B'_{1}|=\sum_{i=1}^{K_{n}}\xi_{i}
$$
where $\xi_{i}$ are iid Bernoulli$(q_{\alpha,\zeta}).$ Hence, setting $\bar{\xi}_{K_{n}}=\sum_{i=1}^{K_{n}}\xi_{i}/K_{n},$
$$
\hat{K}_{n}\simeq
K_{n}\bar{\xi}_{K_{n}}.
$$
By the law of large numbers $\bar{\xi}_{K_{n}}$ converges almost surely to $q_{\alpha,\zeta}=\zeta/(\varepsilon_{\alpha}+\zeta).$
Since $K_{n}$ is determined by $\mathrm{PG}(\alpha,\varepsilon_{\alpha}+\zeta)$ it follows from Pitman~\cite[Section 6.1]{Pit02} that $K_{n}\simeq n^{\alpha}Z,$
where
$$
Z=\frac{(\varepsilon_{\alpha}+\zeta)}{{[\tau_{\alpha}(\varepsilon_{\alpha}+\zeta)]}^{\alpha}}.
$$
hence $\hat{Z}=Z\times q_{\alpha,\zeta}$ concluding the result.
\end{proof}

\begin{rem}
The result suggests that $EPG(\alpha,\zeta)$ models are special cases of $\mathrm{PK}_{\alpha}(h\cdot f_{\alpha})$ with mixing distribution $h(s)f_{\alpha}(s)$ equivalent to that of the random variable in~(\ref{EPT}). However, it is not immediately obvious how to obtain a precise form of the density of the random variable that identifies an explicit form for $h(s).$ We do this next.
\end{rem}
\subsection{$\mathrm{EPG}(\alpha,\zeta)$ laws are sized biased  $\mathrm{PG}(\alpha,\zeta)$ laws}
\begin{prop}\label{PKprop}Consider the setting in Proposition~(\ref{MarkovAlpha}). Conditional on $\zeta,$ the random variable
$$
\hat{T}=\frac{\tau_{\alpha}(\varepsilon_{\alpha}+\zeta)}{\zeta^{1/\alpha}}.
$$
appearing in Proposition~(\ref{MarkovAlpha}), has density
$$
s\frac{1}{\alpha}\zeta^{1/\alpha-1}{\mbox e}^{-(s\zeta^{1/\alpha}-\zeta)}f_{\alpha}(s).
$$
Hence $\hat{T}$ is the size-biased distribution of $T=\tau_{\alpha}(\zeta)/\zeta^{1/\alpha}$ that satisfies (\ref{alphadiversity})
for $\mathrm{PG}(\alpha,\zeta).$ It follows that $\mathrm{EPG}(\alpha,\zeta)$ laws are $\mathrm{PK}_{\alpha}(h\cdot f_{\alpha})$ laws with
\begin{equation}
h(s)=s\mathbb{E}[\frac{1}{\alpha}\zeta^{1/\alpha-1}{\mbox e}^{-(s\zeta^{1/\alpha}-\zeta)}].
\label{sizebiasedh}
\end{equation}
This verifies that $\tilde{P}^{\dagger}_{\alpha,\zeta}=1-(T/\hat{T})$ is the first size-biased pick from $\mathrm{EPG}(\alpha,\zeta).$
\end{prop}
\begin{proof}Write
$$
\hat{T}=\frac{\tau_{\alpha}(\varepsilon_{\alpha}+\zeta)}{{(\varepsilon_{\alpha}+\zeta)}^{1/\alpha}}\times \frac{{(\varepsilon_{\alpha}+\zeta)}^{1/\alpha}}{\zeta^{1/\alpha}}=\frac{\tau_{\alpha}(\zeta)}{\zeta^{1/\alpha}}\times \frac{\tau_{\alpha}(\varepsilon_{\alpha}+\zeta)}
{\tau_{\alpha}(\zeta)}.
$$
Then, using the first equality, conditional on $\varepsilon_{\alpha}$ and $\zeta,$ the density of $\hat{T}$ is of the form
$$
cf_{\alpha}(sc){\mbox e}^{-(s{(\varepsilon_{\alpha}+\zeta)}^{1/\alpha}-(\varepsilon_{\alpha}+\zeta))},
$$
where $c^{\alpha}=\zeta/(\varepsilon_{\alpha}+\zeta).$ Using the fact that $\varepsilon_{\alpha}\overset{d}=\gamma_{(1-\alpha)/\alpha},$ and applying the change of variable $v=\zeta/(y+\zeta),$ leads to the conditional density given $\zeta,$ expressible as,
$$
\zeta^{1/\alpha-1}{\mbox e}^{-(s\zeta^{1/\alpha}-\zeta)}
\frac{1}{\Gamma(\frac{1-\alpha}{\alpha})}\int_{0}^{1}f_{\alpha}(sv^{1/\alpha})(1-v)^{\frac{(1-\alpha)}{\alpha}-1}dv.
$$
Using the identity $S_{\alpha}\overset{d}=S_{\alpha,1}\times \beta^{-1/\alpha}_{1,(1-\alpha)/\alpha},$ see again  James~\cite[p.8, eq. (2.11)]{JamesLamperti}, shows that
$$
\frac{1}{\Gamma(\frac{1-\alpha}{\alpha})}
\int_{0}^{1}f_{\alpha}(sv^{1/\alpha})(1-v)^{\frac{(1-\alpha)}{\alpha}-1}dv=\frac{s}{\alpha}f_{\alpha}(s).
$$
\end{proof}
\begin{rem} The sampling scheme in Section~\ref{samplingcoag} shows how to obtain a partition  ${\{B_{1},\ldots,B_{\hat{K}_{n}}\}},$ from an
$\mathrm{EPG}(\alpha,\zeta)-\mathrm{EPPF}.$ Once we have identified
$$
h(s)=s\mathbb{E}[\frac{1}{\alpha}\zeta^{1/\alpha-1}{\mbox e}^{-(s\zeta^{1/\alpha}-\zeta)}],
$$
it is straightforward to obtain expressions for the EPPF (and other quantities), see \cite[Theorem 4.5, p.86]{Pit06}, and also~\cite{GnedinPitmanI,Pit02}. One can apply the identities in Section 3.
\end{rem}
\subsection{When are EPG models  $\mathrm{PG}(\alpha,\zeta)$ models}
The next result, which allows one to  identify when an extended model is a $PG(\alpha,\zeta)$ model, reveals important recursive relations.
\begin{prop}\label{recursiverelations}For $0\leq \alpha<1$, and $\zeta$ a non-negative random variable there is the relation
$$
\mathrm{EPG}(\alpha,\gamma_{1}+\zeta)=\mathrm{PG}(\alpha,\zeta).
$$
Noting that $\varepsilon_{\alpha}+\gamma_{1}\overset{d}=\gamma_{1/\alpha}$ gives $\mathrm{EPG}(\alpha,\gamma_{1/\alpha}+\zeta)=\mathrm{PG}(\alpha,\varepsilon_{\alpha}+\zeta).$
Leading to the following implications,
\begin{eqnarray}
F_{\alpha,\zeta}(y)& = & F_{\alpha,\gamma_{1}+\varepsilon_{\alpha}+\zeta}(\lambda_{\alpha,\zeta}(y))
\label{EPGrecurse}\\
& =&\frac{\tau_{\alpha}(\zeta)}{\tau_{\alpha}(\varepsilon_{\alpha})+\tau_{\alpha}(\zeta)}F_{\alpha,\gamma_{1}+\zeta}(y)
+\frac{\tau_{\alpha}(\varepsilon_{\alpha})}{\tau_{\alpha}(\varepsilon_{\alpha})+\tau_{\alpha}(\zeta)}\indic_{\{U_{1}\leq y\}}.
\label{EPGrecurse2}
\end{eqnarray}
with
$$
F_{\alpha,\gamma_{1}+\varepsilon_{\alpha}+\zeta}(y)=Q_{\alpha,\varepsilon_{\alpha}+\zeta}(y)=\frac{\tau_{\alpha}((\varepsilon_{\alpha}+\zeta)y)}
{\tau_{\alpha}(\varepsilon_{\alpha}+\zeta)}
$$
and where $\tau_{\alpha}(\varepsilon_{\alpha})\overset{d}=\gamma_{1-\alpha}$ is independent of $F_{\alpha,\gamma_{1}+\zeta}(y)=Q_{\alpha,\zeta}(y).$
\end{prop}
\begin{proof}Recall the form of $h(s)$ in~(\ref{sizebiasedh})
$$
h(s)=s\mathbb{E}[\frac{1}{\alpha}\zeta^{1/\alpha-1}{\mbox e}^{-(s\zeta^{1/\alpha}-\zeta)}].
$$
Then substituting $\zeta$ with $\gamma_{1}+\zeta$ and integrating with respect to $\gamma_{1},$ conditioned on $\zeta,$
leads to the evaluation of the integral
$$
s\int_{0}^{\infty}(x+\zeta)^{1/\alpha-1}{\mbox e}^{-s{(x+\zeta)}^{1/\alpha}}dx=\alpha {\mbox e}^{-s\zeta^{1/\alpha}}
$$
completing the result.
\end{proof}

\section{A homogeneous discrete Markov Chain for $\mathrm{PK}_{\alpha}(h\cdot f_{\alpha})$ models}\label{MCcoag}

In this Section we will show that Propositions~\ref{MarkovAlpha}, \ref{PKprop} and ~\ref{recursiverelations}, helps to identify a homogeneous Markov chain that is complementary to Perman,Pitman,Yor~\cite[Theorem 2.1]{PPY92} in the $\alpha-$stable case, and also relates to the results in~\cite{BertoinGoldschmidt2004,Dong2006,Haas}. In fact we argue that while the Markov Chains exhibited in Perman,Pitman,Yor~\cite[Theorem 2.1,Corollary 2.3, Corollary 3.15]{PPY92} can be seen as one derived from (stick-breaking) size biased sampling and excising of excursion intervals, the \emph{related} Chain that we shall discuss is one that can be derived from repeated application of the fragmentation or coagulation operations in~BDGM~\cite{BertoinGoldschmidt2004,Dong2006}. To clarify what we mean, we find it useful to first illustrate some relevant facts from the work of~\cite{PPY92} in case (I) below and then in (II) outline some key points about parallel chains derived from the operations in~\cite{BertoinGoldschmidt2004,Dong2006}, that we shall subsequently establish.

\begin{enumerate}
\item[(I)]\textsc{Markov Chains induced by stick-breaking in PPY\cite{PPY92}}
\item[(i)]Let $(P^{(0)}_{i})\in \mathcal{P}_{\infty}$, denote a mass partition following a $\mathrm{PD}(\alpha,0)$ law and let $\hat{T}^{-\alpha}_{\alpha,0}\overset{d}=S^{-\alpha}_{\alpha},$ denote its $\alpha$-diversity. We may assume that $(P^{(0)}_{i})\in \mathcal{P}_{\infty}$ are the normalized ranked lengths of excursion of $\mathcal{B}_{(0)}:=(B^{(0)}_{t},t>0)$ denoting a recurrent Bessel Process of dimension $2-2\alpha$ on $\mathbb{R}^{+},$ starting at $B^{(0)}_{0}=0$ or for $\alpha=1/2$ a Brownian Motion.
\item[(ii)]As in~\cite[Section 3]{PPY92},  starting from $(\mathcal{B}_{(0)},(P^{(0)}_{i})),$ one may construct a family of generalized Bessel bridges, with corresponding excursion lengths and $\alpha$-diversities denoted as ${((P_{i,k-1}),\hat{T}_{\alpha,{(k-1)}\alpha})}_{\{k\ge 1\}},$
$(P_{i,0})=(P^{(0)}_{i}),$
formed by the successive deletion of excursion lengths discovered by size-biased sampling(that is using the \emph{first size biased pick}, say $1-W_{k},$ from each subsequently formed~$(P_{i,k-1})$) and then re-sizing/re-scaling such that the subsequent interval has unit length.
\item[(iii)]In this setting, the collections $(W_{1},W_{2},\ldots)$ are independent with $W_{k}\overset{d}=\beta_{k\alpha,1-\alpha}.$ $\hat{T}^{-\alpha}_{\alpha,k\alpha}\overset{d}=S^{-\alpha}_{\alpha,k\alpha},$ is independent of $(W_{1},\ldots,W_{k}),$ and furthermore satisfies  $\hat{T}_{\alpha,k\alpha}=\hat{T}_{\alpha,(j-1)\alpha}\times \prod_{l=1}^{j}W_{l},$ for $j=1,\ldots, k.$ Note while $W_{1}$ is independent of $\hat{T}_{\alpha,\alpha},$ $W_{2}$ is not.
\item[(iv)]The family ${((P_{i,k-1}))}_{\{k\ge 1\}}$ forms a discrete Markov chain with states indicated by the family of laws
${(\mathrm{PD}(\alpha,(k-1)\alpha))}_{\{k\ge 1\}}.$ That is,
$$
\mathrm{PD}(\alpha,0),\mathrm{PD}(\alpha,\alpha), \mathrm{PD}(\alpha,2\alpha),\ldots
$$
\item[(v)]That is, reading left to right this is a chain that can be explained by the deletion of excursion lengths in the sense of \cite[Corollary 3.15]{PPY92} or equivalently by the deletion operation described in~Pitman and Yor~\cite[Propositions 34, Section 6.1]{PY97}. This description, although independence no longer applies in general, holds without much modification under a change of measure in the general  $\mathrm{PK}_{\alpha}(h\cdot f_{\alpha})$ setting, as it involves size-biased picks.
\item[(vi)]\cite[Corollary 3.15]{PPY92} shows that under the $\mathrm{PD}(\alpha,\theta),$ $\theta>-\alpha$ setting independence still holds and the description of relevant variables is obtained by replacing $(k-1)\alpha$ with $\theta+(k-1)\alpha.$ This generalizes the known Dirichlet process case, $\alpha=0,$ where starting with $\mathrm{PD}(0,\theta)$ induces a chain exhibiting the following well-known and characterizing invariance property
$$
\mathrm{PD}(0,\theta),\mathrm{PD}(0,\theta), \mathrm{PD}(0,\theta),\ldots
$$
\item[(vii)] Conversely a dual chain can be read from right to left by using a modification of the insertion operation in \cite[Proposition 35, Section 6.1]{PY97}. Note, a nice description of the insertion operation, in the general  $\mathrm{PK}_{\alpha}(h\cdot f_{\alpha})$ setting, requires some thought as the description in \cite[Proposition 35, Section 6.1]{PY97} relies heavily on independence in the $\mathrm{PD}(\alpha,\theta)$ setting.\qed
\end{enumerate}

We now outline some points that we shall establish and expand upon.

\begin{enumerate}
\item[(II)]\textsc{Dual Markov Chains induced by successive coag/frag operations in BDGM~\cite{BertoinGoldschmidt2004,Dong2006}.}

\item[(viii)]Starting from the same $(\mathcal{B}_{(0)},(P^{(0)}_{i})),$ under a $\mathrm{PD}(\alpha,0)$ law, one may construct another family of generalized Bessel bridges, with corresponding excursion lengths and $\alpha$-diversities denoted as ${((P^{(k-1)}_{l}),\hat{T}_{\alpha,{(k-1)}})}_{\{k\ge 1\}},$
formed by applying  successively a particular variation of the coagulation/fragmentation operation in BDGM~\cite{BertoinGoldschmidt2004,Dong2006}, that we shall describe.
\item[(ix)]These operations are encoded by the equality, for $k=1,2,\ldots,$ which is deduced from Proposition~\ref{MarkovAlpha},
$$
\hat{T}_{\alpha,k-1}=\hat{T}_{\alpha,k}\times V^{-1/\alpha}_{k},
$$
leading to a Markov Chain that complements \cite[Theorem 2.1, Corollary 2.3]{PPY92} in a $\mathrm{PD}(\alpha,0),$ hence $\mathrm{PK}_{\alpha}(h\cdot f_{\alpha}),$ setting.
\item[(x)]Under $\mathrm{PD}(\alpha,0),$ $(V_{1},V_{2},\ldots)$ are independent with $V_{k}\overset{d}=\beta_{((\alpha+k-1)/\alpha,(1-\alpha)/\alpha)}.$
$\hat{T}^{-\alpha}_{\alpha,k}\overset{d}=S^{-\alpha}_{\alpha,k},$ is independent of $(V_{1},\ldots,V_{k}),$ and furthermore satisfies
$\hat{T}^{-\alpha}_{\alpha,k}=\hat{T}^{-\alpha}_{\alpha,(j-1)}\times \prod_{l=1}^{j}V_{l},$ for $j=1,\ldots, k.$ Note again while $V_{1}$ is independent of $\hat{T}_{\alpha,1},$ $V_{2}$ is not.
\item[(xi)]Under $\mathrm{PD}(\alpha,\theta),$ $\hat{T}_{\alpha,0}\overset{d}=S_{\alpha,\theta},$ independence is preserved and otherwise the distributions are adjusted by replacing $k-1$ with $\theta+k-1,$ The family ${((P^{(k-1)}_{l}))}_{\{k\ge 1\}}$ forms a discrete Markov chain with states indicated by the family of laws
${(\mathrm{PD}(\alpha,\theta+(k-1)))}_{\{k\ge 1\}}.$ That is,
$$
\mathrm{PD}(\alpha,\theta),\mathrm{PD}(\alpha,\theta+1), \mathrm{PD}(\alpha,\theta+2),\ldots
$$
Reading left to right this is the fragmentation chain. Reading right to left is the dual coagulation chain. This naturally agrees with chains formed by the constructions in~BDGM~\cite{BertoinGoldschmidt2004,Dong2006}. However, points  [(ix)] and [(x)], which we believe offers further interpretation and provides a blueprint for constructing processes with Markovian properties,  are not emphasized .
\item[(xii)]Note, in the general setting,  a relation between the chains in (I) and (II) is seen in the following equality, again deduced from Proposition~\ref{MarkovAlpha},
$$
\hat{T}_{\alpha,0}=\hat{T}_{\alpha,1}\times V^{-1/\alpha}_{1}=\hat{T}_{\alpha,\alpha}\times W^{-1}_{1}
$$
Importantly this relation between $(V_{1},W_{1})$ does not extend in the same manner to $(V_{k},W_{k})$ for $k=2,\ldots,.$
\item[(xiii)]Lastly the typical $\mathrm{PD}(\alpha,\theta)$ constructions, in terms of notation, mask the fact that these processes are naturally connected to waiting times. This will be seen from the more general $\mathrm{EPG}(\alpha,\zeta)$ constructions
\qed
\end{enumerate}

The next result in the $\mathrm{PD}(\alpha,\theta)$ setting follows as a Corollary of Propositions~\ref{MarkovAlpha},\ref{PKprop} and ~\ref{recursiverelations}. Importantly it reveals the structure of the Markcov Chains we have in mind.

\begin{cor}Setting $\zeta=\gamma_{(\theta+\alpha)/\alpha}$ in Propositions~\ref{MarkovAlpha},\ref{PKprop} and ~\ref{recursiverelations}, leads to results for the $\mathrm{PD}(\alpha,\theta)$ case, for $\theta>-\alpha$
\begin{enumerate}
\item[(i)]There is the distributional identity,
$$
S_{\alpha,\theta}\overset{d}=\frac{\tau_{\alpha}(\varepsilon_{\alpha}+\gamma_{\frac{\theta+\alpha}{\alpha}})}{\gamma_{\frac{\theta+\alpha}{\alpha}}^{1/\alpha}}
=
\frac{\tau_{\alpha}(\varepsilon_{\alpha}+\gamma_{\frac{\theta+\alpha}{\alpha}})}{{(\varepsilon_{\alpha}+
\gamma_{\frac{\theta+\alpha}{\alpha}})}^{1/\alpha}}
\times {(\beta_{\left(\frac{\theta+\alpha}{\alpha},\frac{1-\alpha}{\alpha}\right)})}^{-1/\alpha},
$$
where the terms in the product are independent.
\item[(ii)] Indicating there is version of the variables $(S^{-\alpha}_{\alpha,\theta},S^{-\alpha}_{\alpha,1+\theta},S^{-\alpha}_{\alpha,\theta+\alpha}),$ corresponding to the $\alpha$-diversities of  $(\mathrm{PD}(\alpha,\theta),\mathrm{PD}(\alpha,1+\theta),\mathrm{PD}(\alpha,\theta+\alpha)),$ related by,
$$
S^{-\alpha}_{\alpha,\theta}=S^{-\alpha}_{\alpha,\theta+1}\times \beta_{\left(\frac{\theta+\alpha}{\alpha},\frac{1-\alpha}{\alpha}\right)}=
S^{-\alpha}_{\alpha,\theta+\alpha}\times \beta^{\alpha}_{\theta+\alpha,1-\alpha},
$$
\item[(iii)]This can also be expressed by the relationship
\begin{eqnarray*}
\nonumber P_{\alpha,\theta}(y)& = & P_{\alpha,1+\theta}
\left((1-\beta_{(\frac{1-\alpha}{\alpha},\frac{\theta+\alpha}{\alpha})})\mathbb{U}(y)+\beta_{(\frac{1-\alpha}{\alpha},
\frac{\theta+\alpha}{\alpha})}\indic_{(U_{1}\leq y)}\right)\\
                   & = & \beta_{\theta+\alpha,1-\alpha}P_{\alpha,\theta+\alpha}(y)+(1-\beta_{\theta+\alpha,1-\alpha})\indic_{\{U_{1}\leq y\}},
\end{eqnarray*}
where the relevant quantities on the right hand side are independent of one another.
\item[(iv)]Hence for each $n$; there exists a version of $(S_{\alpha,\theta+n-1},S_{\alpha,\theta+n})$ such that
$$
S_{\alpha,\theta+n-1}=S_{\alpha,\theta+n}
\times {(\beta_{\left(\frac{\theta+n-1+\alpha}{\alpha},\frac{1-\alpha}{\alpha}\right)})}^{-1/\alpha}.
$$
Note: Indicating the appropriate transitions for a Markov Chain.
\item[(v)] Furthermore, repeated application yields
\begin{equation}
S^{-\alpha}_{\alpha,\theta}=S^{-\alpha}_{\alpha,\theta+n}\times
\prod_{i=1}^{n}\beta_{\left(\frac{\theta+i-1+\alpha}{\alpha},\frac{1-\alpha}{\alpha}\right)}
\label{jamesid}
\end{equation}
\end{enumerate}
\end{cor}
\begin{proof} The bulk of this result follows immediately by substituting appropriately in Propositions~\ref{MarkovAlpha},\ref{PKprop} and~\ref{recursiverelations}.
However we find it illustrative to give some details about statement [(i)] and hence [(ii)].  The right hand side of statement [(i)] can be written more explicitly as
\begin{equation}
\frac{\tau_{\alpha}(\varepsilon_{\alpha}+\gamma_{\frac{\theta+\alpha}{\alpha}})}{\gamma^{1/\alpha}_{\frac{\theta+\alpha}{\alpha}}}=\frac{\tau_{\alpha}
(\varepsilon_{\alpha}+\gamma_{\frac{\theta+\alpha}{\alpha}})
}{{(\varepsilon_{\alpha}+\gamma_{\frac{\theta+\alpha}{\alpha}})}^{1/\alpha}}
\times \frac{{(\varepsilon_{\alpha}+\gamma_{\frac{\theta+\alpha}{\alpha}})}^{1/\alpha}}
{\gamma^{1/\alpha}_{\frac{\theta+\alpha}{\alpha}}}.
\label{basicChain1}
\end{equation}
where independence is seen to follow from the standard beta-gamma calculus of Lukacs~\cite{Lukacs}. The distributional identity for $S_{\alpha,\theta}$ can be deduced by substituting $\zeta=\gamma_{(\theta+\alpha)/\alpha}$ in~(\ref{sizebiasedh}). The distributional equivalences for $S_{\alpha,\theta+1}$ and $S_{\alpha,\theta+\alpha}$ to their $\mathrm{PG}(\alpha,\zeta)$ forms are already known from~Pitman and Yor\cite{PY97}. Otherwise, the  \textit{equivalent in distribution} version of these points can be read from James~\cite[p.8, eq. (2.11)]{JamesLamperti}.
\end{proof}

Although the independent structures above will be replaced by dependent quantities, similar relations hold for all $\mathrm{PK}_{\alpha}(h\cdot f_{\alpha})$
which is encoded and formalized in the next series of results. The first result describes a Markov Chain of $\alpha$-diversities
labeled in the general $\mathrm{PK}_{\alpha}(h\cdot f_{\alpha})$
setting as $(\hat{T}^{-\alpha}_{\alpha,0},\hat{T}^{-\alpha}_{\alpha,1},\hat{T}^{-\alpha}_{\alpha,2},\ldots)$
which complements Perman,Pitman,Yor~\cite[Theorem 2.1,Corollary 2,3]{PPY92} in the $\mathrm{PD}(\alpha,0)$ setting. In order to do this we only need to use the decomposition in the  canonical $\mathrm{PD}(\alpha,0)$ case,
\begin{equation}
S_{\alpha}=S_{\alpha,1}
\times {(\beta_{\left(1,\frac{1-\alpha}{\alpha}\right)})}^{-1/\alpha},
\label{basicChain1}
\end{equation}
which reads in terms of an $\mathrm{EPG}(\alpha,\gamma_{1})$ model as
\begin{equation}
\frac{\tau_{\alpha}(\varepsilon_{\alpha}+\gamma_{1})}{\gamma_{1}^{1/\alpha}}=\frac{\tau_{\alpha}(\varepsilon_{\alpha}+\gamma_{1})}{{(\varepsilon_{\alpha}+\gamma_{1})}^{1/\alpha}}
\times \frac{{(\varepsilon_{\alpha}+\gamma_{1})}^{1/\alpha}}{\gamma_{1}^{1/\alpha}}.
\label{basicChain1}
\end{equation}
The independence between the variables on the right hand side make this simple.

\begin{prop}\label{JPPY}For each $n,$ let $(\hat{T}_{\alpha,0}, \hat{T}_{\alpha,1},\ldots,\hat{T}_{\alpha,n})$ denote a vector of random variables such that $\hat{T}_{\alpha,0}\overset{d}=S_{\alpha}$ and there is the relationship for each integer $k$
\begin{equation}
\hat{T}_{\alpha,(k-1)}=\hat{T}_{\alpha,k}\times V^{-1/\alpha}_{k}
\label{Vequation}
\end{equation}
where $V_{k}\overset{d}=\beta_{((\alpha+k-1)/\alpha,(1-\alpha)/\alpha)}$ independent of $\hat{T}_{\alpha,k}$ and marginally
$\hat{T}_{\alpha,k}\overset{d}=S_{\alpha,k}.$ Then, the conditional distribution of $\hat{T}_{\alpha,k}$ given  $\hat{T}_{\alpha,k-1}=t$ is the same for all $k$ and equates to the density,
\begin{equation}
P(\hat{T}_{\alpha,1}\in ds|\hat{T}_{\alpha,0}=t)/ds=\frac{\alpha^{2}}{\Gamma(\frac{1-\alpha}{\alpha})}
\frac{(s/t)^{\alpha-1}(1-(s/t)^{\alpha})^{\frac{(1-\alpha)}{\alpha}-1}f_{\alpha}(s)}{t^{2}f_{\alpha}(t)},
\label{transitionV}
\end{equation}
for $s<t.$
By a change of variable $v=(s/t)^{\alpha}$ the density of $V_{1}|\hat{T}_{\alpha,0}=t$ is given by
\begin{equation}
P(V_{1}\in dv|\hat{T}_{\alpha,0}=t)/dv=\frac{\alpha}{\Gamma(\frac{1-\alpha}{\alpha})}
\frac{(1-v)^{\frac{(1-\alpha)}{\alpha}-1}f_{\alpha}(v^{1/\alpha}t)}{tf_{\alpha}(t)}.
\label{transitionV2}
\end{equation}
Furthermore $(V_{1},\ldots,V_{n})$ are independent variables, independent of $\hat{T}_{\alpha,n}.$
Hence, the sequence is a Markov Chain, governed by a $\mathrm{PD}(\alpha,0)$ law, with interpretations given by the discussion above.
\end{prop}
\begin{proof}
Because of the independence between $V_{k}$ and $\hat{T}_{\alpha,k}$ the proof just reduces to an elementary Bayes rule argument. Details are presented for clarity. The distribution of $\hat{T}_{\alpha,k-1}|\hat{T}_{\alpha,k}=s$ is just $V^{-1/\alpha}_{k}s,$ where $V_{k}\overset{d}=\beta_{((\alpha+k-1)/\alpha,(1-\alpha)/\alpha)}.$ Use the fact that for each $k,$  $\hat{T}_{\alpha,k}$ has density
$$
f_{\alpha,k}(s)=\frac{\Gamma(k+1)}{\Gamma(\frac{k+\alpha}{\alpha})}s^{-k}f_{\alpha}(s),
$$
to show that the joint density of $(\hat{T}_{\alpha,k-1},\hat{T}_{\alpha,k})$ is,
\begin{equation}
\frac{\alpha^{2}\Gamma(k)t^{-(k+1)}}{\Gamma(\frac{k+\alpha-1}{\alpha})\Gamma(\frac{1-\alpha}{\alpha})}
(s/t)^{\alpha-1}(1-(s/t)^{\alpha})^{\frac{(1-\alpha)}{\alpha}-1}f_{\alpha}(s).
\label{jointSk}
\end{equation}
 Now divide~(\ref{jointSk}) by the $f_{\alpha,k-1}(t)$ density of $\hat{T}_{\alpha,k-1},$ to obtain (\ref{transitionV}). Note again that the \textit{in distribution} version of~(\ref{Vequation}), which yields the same conditional density in~(\ref{transitionV}), can be read from James~\cite[p.8, eq. (2.11)]{JamesLamperti}. The Markov chain is otherwise evident from the exact equality statement.
\end{proof}

\begin{cor}\label{corollary1}
As consequences of Proposition~\ref{JPPY} the distribution of the quantities above with respect to a $PK(\rho_{\alpha};h\cdot f_{\alpha})$ are given by (\ref{transitionV}) and specifying $\hat{T}_{\alpha,0}$ to have density $h(t)f_{\alpha}(t).$
\begin{enumerate}
\item[(i)]
 In particular, the joint law of $(V_{1},\ldots,V_{n},\hat{T}_{\alpha,n})$ is given by,
\begin{equation}
\left[\prod_{k=1}^{n}f_{\beta_{k}}(v_{k})\right]h(s/\prod_{l=1}^{n}v^{1/\alpha}_{l})f_{\alpha,n}(s)ds
\label{jointV}
\end{equation}
where $f_{\beta_{k}}$ denotes the density of a $\beta_{(\alpha+k-1)/\alpha,(1-\alpha)/\alpha},$ variable.
\item[(ii)]Relative to  $\hat{T}_{\alpha,n},$ for each $j=1,2,\ldots,n$
$$
\hat{T}^{-\alpha}_{\alpha,j-1}=\hat{T}^{-\alpha}_{\alpha,n}\times \prod_{l=j}^{n}V_{l}
$$
\end{enumerate}
\end{cor}
\begin{rem}\label{Expecth} The fact that the quantity in~(\ref{jointV}) integrates to $1,$ follows from the identity~(\ref{jamesid}). Which reads as
$$
\mathbb{E}_{\alpha,0}[h(S_{\alpha})]=\mathbb{E}_{\alpha,0}[h(S_{\alpha,n}\times \prod_{i=1}^{n}\beta^{-1/\alpha}_{\left(\frac{i-1+\alpha}{\alpha},\frac{1-\alpha}{\alpha}\right)})]=1.
$$
\end{rem}

\begin{rem}\label{randomtrees}Comparing (\ref{transitionV}) with Haas, Miermont, Pitman and Winkel~\cite[Proposition 18, (b),(c)]{Haas} shows that under a $\mathrm{PD}(\alpha,1-\alpha)$ model, where for each $k=1,2,\ldots;$ $\hat{T}_{\alpha,k-1}\overset{d}=S_{\alpha,k-\alpha},$ $\hat{T}^{-\alpha}_{\alpha,k-1}$ equates to the total length  of $\mathcal{\tilde{R}}^{\mathrm{ord}}_{k},$ say $
\mathbb{D}(\mathcal{\tilde{R}}^{\mathrm{ord}}_{k})=\hat{T}^{-\alpha}_{\alpha,k-1}.$
Where $\mathcal{\tilde{R}}^{\mathrm{ord}}_{k}$ is a member of an increasing family
$(\mathcal{\tilde{R}}^{\mathrm{ord}}_{k})$
of leaf-labeled $\mathbb{R}$-trees with edge lengths, arising as limits in Ford's sequential construction. It follows from~(\ref{Vequation}) that, in this setting, $(V_{k})$ can be interpreted as
$$
V_{k}=\frac{\mathbb{D}(\mathcal{\tilde{R}}^{\mathrm{ord}}_{k})}{
\mathbb{D}(\mathcal{\tilde{R}}^{\mathrm{ord}}_{k+1})}\overset{d}=\beta_{(\frac{k}{\alpha},\frac{1-\alpha}{\alpha})},
$$
which is independent of $\mathbb{D}(\mathcal{\tilde{R}}^{\mathrm{ord}}_{k+1})=\hat{T}^{-\alpha}_{\alpha,k}\overset{d}=S^{-\alpha}_{\alpha,k+1-\alpha}.$ In fact $(V_{1},\ldots,V_{k})$ are mutually independent and independent of $\mathbb{D}(\mathcal{\tilde{R}}^{\mathrm{ord}}_{k+1}).$
See~\cite{Haas} for a more precise interpretation of $(\mathcal{\tilde{R}}^{\mathrm{ord}}_{k}).$ See also \cite{Dong2006} for related discussions involving fragmentation by  $\mathrm{PD}(\alpha,1-\alpha)$ models.
\end{rem}
\subsection{Results for $\mathrm{PG}(\alpha,\zeta)$} We now address the specific cases of $\mathrm{PG}(\alpha,\zeta)$ and $\mathrm{EPG}(\alpha,\zeta).$
Let $(\mathrm{e}_{i})$ denote a collection of iid exponential~$(1)$ random variables and, independent of this, let $(\varepsilon_{\alpha,i})$ be a
collection of iid variables such that~$\varepsilon_{\alpha,i}\overset{d}=\gamma_{(1-\alpha)/\alpha}$
Note that for each $i,$
$$
\mathrm{e}_{i}+\varepsilon_{\alpha,i}\overset{d}=\gamma_{\frac{1}{\alpha}}
$$
Define a sequence of nested sums, such that $\zeta_{\alpha,0}:=\zeta,$ and
\begin{equation}
\zeta_{\alpha,k}=\sum_{i=1}^{k}(\mathrm{e}_{i}+\varepsilon_{\alpha,i})+\zeta\overset{d}=\gamma_{\frac{k}{\alpha}}+\zeta
\label{partialsumalpha}
\end{equation}
Hence $\zeta_{\alpha,i}=\mathrm{e}_{i}+\varepsilon_{\alpha,i}+\zeta_{\alpha,i-1}.$ The next result arises in the description of $(V_{k})$ and follows from an elementary conditioning argument.
\begin{lem}\label{lemmaphi}Let $\zeta_{\alpha,k}$ be random variables described in~(\ref{partialsumalpha}) then for any $\zeta>0,$ the conditional joint density of the vector $(\vartheta_{\alpha,1},\ldots,\vartheta_{\alpha,n}),$ given $\zeta,$ where
$$\vartheta_{\alpha,k}:=\frac{{\zeta_{\alpha,k-1}}}{{\zeta_{\alpha,k}}};{\mbox { for }} k=1,\ldots,n
$$
is given by
$$
\frac{\zeta^{\frac{n}{\alpha}}{\mbox e}^{\zeta}}{{[\Gamma(\frac{1}{\alpha})]}^{n}}
{\mbox e}^{-\zeta/(\prod_{l=1}^{n}y_{l})}\prod_{l=1}^{n}y^{-\frac{(n-l+1)}{\alpha}-1}_{l}(1-y_{l})^{\frac{1}{\alpha}-1}.
$$
\begin{enumerate}
\item[(i)]When $\zeta\overset{d}=\gamma_{\theta/\alpha},$ corresponding to the $\mathrm{PD}(\alpha,\theta)$ case for $\theta \ge 0,$ it follows that the
$\vartheta_{\alpha,k}$ are independent with
$$\vartheta_{\alpha,k}\overset{d}=\beta_{(\frac{(\theta+k-1)}{\alpha},\frac{1}{\alpha})}
$$
 for $k=1,2,\ldots$.
\item[(ii)]The case $\zeta=0,$ is equivalent to setting $\theta=0.$ Hence $\vartheta_{\alpha,1}=0,$ and otherwise
$$\vartheta_{\alpha,k}\overset{d}=\beta_{(\frac{(k-1)}{\alpha},\frac{1}{\alpha})}.
$$
Note conditioning in the case $\zeta=0$ not only corresponds to the $\mathrm{PD}(\alpha,0)$ model but any $\mathrm{PG}(\alpha,\zeta)$ model where $P(\zeta=0)>0.$
\end{enumerate}
\end{lem}

\begin{prop}
Consider the setting in Proposition~\ref{JPPY} then in the $\mathrm{PG}(\alpha,\zeta)$ case the  joint density of $(V_{1},\ldots,V_{n},\hat{T}_{\alpha,n})$ conditioned on $\zeta,$ is expressed as follows
\begin{equation}
\left[\prod_{k=1}^{n}f_{\beta_{k}}(v_{k})\right]{\mbox e}^{\zeta}{\mbox e}^{-s{(\frac{\zeta}{\prod_{l=1}^{n}v_{l}})}^{1/\alpha}}f_{\alpha,n}(s)
\label{jointVG1}
\end{equation}
and the joint density of $(V_{1},\ldots,V_{n})|\zeta$ can be expressed as
\begin{equation}
\left[\prod_{k=1}^{n}f_{\beta_{k}}(v_{k})\right]{\mbox e}^{\zeta}\mathbb{S}_{\alpha,n}\left({\left(\frac{\zeta}{\prod_{l=1}^{n}v_{l}}\right)}^{1/\alpha}\right)
\label{jointVG2}
\end{equation}
where $f_{\beta_{k}}$ denotes the density of a $\beta_{(\alpha+k-1)/\alpha,(1-\alpha)/\alpha},$ variable. See (\ref{Sid}) for various forms of $\mathbb{S}_{\alpha,n}(y).$
\begin{enumerate}
\item[(i)]When $\zeta=0,$ it follows from~(\ref{jointVG1}),(\ref{jointVG2}), that for $k=1,2,\ldots$
$$
V_{k}\overset{d}=\beta_{(\frac{(\alpha+k-1)}{\alpha},\frac{(1-\alpha)}{\alpha})}
$$
and are independent.
\end{enumerate}
\end{prop}
\begin{rem}Note for clarity, through various cancellations the joint density in~(\ref{jointVG1}) has the precise form
$$
{\mbox e}^{\zeta}{\mbox e}^{-s{(\frac{\zeta}{\prod_{l=1}^{n}v_{l}})}^{1/\alpha}}s^{-n}f_{\alpha}(s)\prod_{k=1}^{n}
\frac{(1-\alpha)}{{[\Gamma(\frac{1}{\alpha})]}}
v^{\frac{k-1}{\alpha}}_{k}{(1-v_{k})}^{\frac{1-\alpha}{\alpha}-1}.
$$

\end{rem}

We now identify the random variables in the following theorem.
\begin{thm}\label{TheoremPGV}In the $\mathrm{PG}(\alpha,\zeta)$ case, where $\zeta$ is a non-negative random variable,  the variables defined in Proposition~\ref{JPPY} are identified explicitly as follows.

\begin{enumerate}
\item[(i)]The equation in~(\ref{Vequation}), $\hat{T}_{\alpha,(k-1)}=\hat{T}_{\alpha,k}\times V^{-1/\alpha}_{k},$
reads in this case as,
\begin{equation}
\frac{\tau_{\alpha}(\mathrm{e}_{k}+\varepsilon_{\alpha,k}+\zeta_{\alpha,k-1})}{{(\mathrm{e}_{k}+\zeta_{\alpha,k-1})}^{1/\alpha}}
=\frac{\tau_{\alpha}(\zeta_{\alpha,k})}{{(\zeta_{\alpha,k})}^{1/\alpha}}\times V^{-1/\alpha}_{k}
\end{equation}
\item[(ii]For each $k,$
$$
V_{k}=\frac{\mathrm{e}_{k}+\zeta_{\alpha,k-1}}{\zeta_{\alpha,k}}:=1-\beta^{(k)}_{(\frac{(1-\alpha)}{\alpha},1)}[1-\vartheta_{\alpha,k}],
$$
where $(\beta^{(k)}_{({(1-\alpha)}/{\alpha},1)})$ are a collection of iid $\mathrm{Beta}((1-\alpha)/\alpha,1)$ variables  independent of $(\vartheta_{\alpha,k}).$
More specifically, $1-V_{k}$ decomposes in terms of
$$
\frac{\varepsilon_{\alpha,k}}{\mathrm{e}_{k}+\varepsilon_{\alpha,k}}:=\beta^{(k)}_{(\frac{(1-\alpha)}{\alpha},1)}{\mbox { and }}\vartheta_{\alpha,k}=\frac{{\zeta_{\alpha,k-1}}}{{\zeta_{\alpha,k}}}
$$
and hence $(V_{1},\ldots,V_{n})$ are conditionally independent given $(\vartheta_{\alpha,1},\ldots, \vartheta_{\alpha,n})$

\end{enumerate}
\end{thm}
\begin{proof}We show that the $(V_{k})$ variables defined in [(ii)] satisfy (\ref{jointVG2}). The remaining results are then easily checked. The description in [(ii)] shows that the joint density of the variables given $(\vartheta_{\alpha,1}=y_{1},\ldots, \vartheta_{\alpha,n}=y_{n},\zeta)$ is of the form
$$
\alpha^{-n}\prod_{l=1}^{n}(1-v_{l})^{\frac{1-\alpha}{\alpha}-1}{(1-y_{l})}^{-(\frac{1-\alpha}{\alpha})}
$$
for $0<y_{l}<v_{l}<1, l=1,\ldots,n.$ Integrating with respect to the density in Lemma~\ref{lemmaphi} amounts to evaluating the integral
$$
\int_{0}^{v_{1}}\cdots\int_{0}^{v_{n}}{\mbox e}^{-\zeta/(\prod_{l=1}^{n}y_{l})}\prod_{l=1}^{n}y^{-\frac{(n-l+1)}{\alpha}-1}_{l}dy_{n}\ldots dy_{1}
$$
Now use the fact that
$$
{\mbox e}^{-\zeta/(\prod_{l=1}^{n}y_{l})}=\int_{0}^{\infty}{\mbox e}^{-s{(\frac{\zeta}{\prod_{l=1}^{n}y_{l}})}^{1/\alpha}}f_{\alpha}(s)ds.
$$
Augmenting this expression within the multiple integral shows that what is left is to evaluate,
$$
\int_{0}^{v_{1}}\cdots\int_{0}^{v_{n}}{\mbox e}^{-s{(\frac{\zeta}{\prod_{l=1}^{n}y_{l}})}^{1/\alpha}}\prod_{l=1}^{n}y^{-\frac{(n-l+1)}{\alpha}-1}_{l}dy_{n}\ldots dy_{1}.
$$
Using the change of variable $r=y^{-1/\alpha},$ leads to a multiple gamma integral,
$$
\alpha^{n}\int_{v^{-1/\alpha}_{1}}^{\infty}\cdots\int_{v^{-1/\alpha}_{n}}^{\infty}{\mbox e}^{-s{({\zeta^{1/\alpha}}{\prod_{l=1}^{n}r_{l}})}}\prod_{l=1}^{n}r^{{n-l}}_{l}dr_{n}\ldots dr_{1}
$$
where, when first integrating with respect to $r_{n},$ gives
$$
\left[\prod_{l=1}^{n-1}r^{{n-l}}_{l}\right]\int_{v^{-\frac{1}{\alpha}}_{n}}^{\infty}{\mbox e}^{-s{({\zeta^{1/\alpha}}{\prod_{l=1}^{n}r_{l}})}}dr_{n},
$$
is equal to
$$
s^{-1}\zeta^{-1/\alpha}{\mbox e}^{-{(\frac{\zeta}{v_{n}})}^{1/\alpha}{(s{\prod_{l=1}^{n-1}r_{l}})}}\prod_{l=1}^{n-2}r^{{n-l-1}}_{l}.
$$
Repeating this procedure for $r_{n-1},$ generates an expression of the form 
$$
s^{-2}\zeta^{-2/\alpha}v^{1/\alpha}_{n}{\mbox e}^{-{(\frac{\zeta}{v_{n}v_{n-1}})}^{1/\alpha}{(s{\prod_{l=1}^{n-2}r_{l}})}}\prod_{l=1}^{n-3}r^{{n-l-2}}_{l}.
$$
Continuing with $r_{n-2}$ etc. reveals the result.
\end{proof}
The proof reveals a joint density of $(V_{1},\ldots,V_{n},\vartheta_{\alpha,1},\ldots,\vartheta_{\alpha,n},\hat{T}_{\alpha,n})$ given $\zeta,$ with arguments $(v_{1},\ldots,v_{n},y_{1},\ldots,y_{n},s),$ that is proportional to
$$
f_{\alpha}(s){\mbox e}^{-s{(\frac{\zeta}{\prod_{l=1}^{n}y_{l}})}^{1/\alpha}}\prod_{l=1}^{n}y^{-\frac{(n-l+1)}{\alpha}-1}_{l}\prod_{l=1}^{n}(1-v_{l})^{\frac{1-\alpha}{\alpha}-1},
$$
 $y_{l}<v_{l}<1, l=1,\ldots,n.$ Furthermore, there is the relation $\zeta_{\alpha,k}=\zeta/(\prod_{l=1}^{k},\vartheta_{\alpha,l}.$ Leading to the following Corollary.

\begin{cor}From Theorem~\ref{TheoremPGV} there are the following conditional relations.
\begin{enumerate}
\item[(i)] Given $(\zeta_{\alpha,k},\vartheta_{\alpha,k})$ the random variables $(V_{k},\hat{T}_{\alpha,k})$, are conditionally independent with joint density, having respective arguments $(v,s),$
$$
\frac{1-\alpha}{\alpha}{(1-v)}^{\frac{1-\alpha}{\alpha}-1}{(1-\vartheta_{\alpha,k})}^{-\frac{1-\alpha}{\alpha}}
{\mbox e}^{-(s\zeta^{1/\alpha}_{\alpha,k}-\zeta_{\alpha,k})}f_{\alpha}(s),
$$
for $\vartheta_{\alpha,k}<v<1.$
\item[(ii)] The joint density of   $(\vartheta_{\alpha,k},\zeta_{\alpha,k})|\zeta,$ with arguments $(u,y),$ is
$$
\frac{1}{\Gamma(\frac{k-1}{\alpha})\Gamma(\frac{1}{\alpha})}
{u}^{\frac{k-1}{\alpha}-1}{(1-u)}^{\frac{1-\alpha}{\alpha}}
{(y-\zeta/u)}^{\frac{k-1}{\alpha}-1}y^{\frac{1}{\alpha}}
{\mbox e}^{-(y-\zeta)}
$$
for $y>\zeta/u.$
\end{enumerate}
\end{cor}
\subsection{Results for $\mathrm{EPG}(\alpha,\zeta)$}
We now describe results in the $\mathrm{EPG}(\alpha,\zeta)$ case which builds on Theorem~\ref{TheoremPGV}.
For the $\mathrm{EPG}(\alpha,\zeta)$ case we relabel, the general variables $(V_{1},\ldots,V_{n},\hat{T}_{\alpha,n})$  appearing in Proposition~\ref{JPPY}. as $(q_{\alpha,1},\ldots,q_{\alpha,n},\tilde{S}_{\alpha,n}).$

\begin{prop}\label{propEPGV}
In the $\mathrm{EPG}(\alpha,\zeta)$ case of  Proposition~\ref{JPPY} the  joint density of $(q_{\alpha,1},\ldots,q_{\alpha,n},\tilde{S}_{\alpha,n})$ conditioned on $\zeta,$ is expressed as follows.
\begin{equation}
\frac{1}{\alpha}\left[\prod_{k=1}^{n}v^{-1/\alpha}_{k}f_{\beta_{k}}(v_{k})\right]\zeta^{\frac{1}{\alpha}-1}{\mbox e}^{\zeta}{\mbox e}^{-s{(\frac{\zeta}{\prod_{l=1}^{n}v_{l}})}^{1/\alpha}}sf_{\alpha,n}(s)ds
\label{jointVEPG}
\end{equation}

\begin{enumerate}
\item[(i)]The conditional density of $q_{\alpha,1}$ given $\zeta,$ has the explicit form,
$$
f_{q_{\alpha,1}}(v)=\frac{\zeta^{\frac{1-\alpha}{\alpha}}}{\Gamma(\frac{1-\alpha}{\alpha})}v^{-1/\alpha}{(1-v)}^{\frac{1-\alpha}{\alpha}-1}{\mbox e}^{-(\zeta/v-\zeta)}
$$
verifying $q_{\alpha,1}=\zeta/({\varepsilon_{\alpha}+\zeta}).$
\item[(ii)] For each $n,$ the conditional density of  $(q_{\alpha,2},\ldots,q_{\alpha,n},\tilde{S}_{\alpha,n})$  given $(q_{\alpha,1},\zeta)$ can be expressed as.
\begin{equation}
\left[\prod_{k=2}^{n}f_{\beta_{k-1}}(v_{k})\right]{\mbox e}^{\varepsilon_{\alpha}+\zeta}{\mbox e}^{-s{(\frac{\varepsilon_{\alpha}+\zeta}{\prod_{l=2}^{n}v_{l}})}^{1/\alpha}}f_{\alpha,n-1}(s)
\label{jointEVG1}
\end{equation}
\end{enumerate}
\end{prop}
\begin{proof}
Statement [(i)] follows by setting n=1 in~(\ref{jointVEPG}), using $sf_{\alpha,1}(s)=\alpha f_{\alpha}(s)$ and the Laplace transform of $f_{\alpha}.$
Dividing~(\ref{jointVEPG}) by $f_{q_{\alpha,1}}(v_{1})$ leads to~(\ref{jointEVG1})
\end{proof}

For shorthand write $\tilde{\zeta}_{\alpha,0}=\zeta$ and for $k=1,2,\ldots,$
$$
\tilde{\zeta}_{\alpha,k}=\mathrm{e}_{k}+\zeta_{\alpha,k-1}+\varepsilon_{\alpha}=\mathrm{e}_{k}+\sum_{l=1}^{k-1}(\mathrm{e}_{l}+\varepsilon_{\alpha,l})+\varepsilon_{\alpha}+\zeta.
$$
or
$$
\tilde{\zeta}_{\alpha,k}=\zeta_{k}+\sum_{l=1}^{k}\varepsilon_{\alpha,l-1}
$$
with $\varepsilon_{\alpha,0}=\varepsilon_{\alpha},$ and there is the relation
$$
\tilde{\zeta}_{\alpha,k-1}+\varepsilon_{\alpha,k-1}=\zeta_{\alpha,n-1}+\varepsilon_{\alpha}.
$$

\begin{thm}\label{TheoremEPGV}
In the $\mathrm{EPG}(\alpha,\zeta)$ case of Proposition~\ref{JPPY} the variables are identified explicitly as follows.

\begin{enumerate}
\item[(i)]The equation in~(\ref{Vequation}), here written as $\tilde{S}_{\alpha,(k-1)}=\tilde{S}_{\alpha,k}\times q^{-1/\alpha}_{\alpha,k},$
reads for  $k=1,2\ldots,$
\begin{equation}
\frac{\tau_{\alpha}(
\tilde{\zeta}_{\alpha,k-1}+\varepsilon_{\alpha,k-1})}
{{(\tilde{\zeta}_{\alpha,k-1})}^{1/\alpha}}
=\frac{\tau_{\alpha}(
\tilde{\zeta}_{\alpha,k-1}+\varepsilon_{\alpha,k-1})}
{{(\tilde{\zeta}_{\alpha,k-1}+\varepsilon_{\alpha,k-1})}^{1/\alpha}}
\times q^{-1/\alpha}_{\alpha,k}
\label{qspec}
\end{equation}
\item[((ii)] For $k=2,3,\ldots,$
$$
q_{\alpha,k}=\frac{\tilde{\zeta}_{\alpha,k-1}}
{\tilde{\zeta}_{\alpha,k-1}+\varepsilon_{\alpha,k-1}}:=1-\beta^{(k)}_{(\frac{1-\alpha}{\alpha},1)}[1-\hat{\vartheta}_{\alpha,k}]
$$
where $(\beta^{(k)}_{({(1-\alpha)}/{\alpha},1)})$ are a collection of iid $\mathrm{Beta}((1-\alpha)/\alpha,1)$ variables  independent of $(\hat{\vartheta}_{\alpha,k}).$
Specifically, the decomposition of $1-q_{\alpha,k}$ is given by,
$$
\frac{\varepsilon_{\alpha,k-1}}{\mathrm{e}_{k-1}+\varepsilon_{\alpha,k-1}}:=\beta^{(k)}_{(\frac{1-\alpha}{\alpha},1)}
{\mbox { and }}
\hat{\vartheta}_{\alpha,k}=
\frac{\tilde{\zeta}_{\alpha,k-2}}{\tilde{\zeta}_{\alpha,k-1}}
=\frac{\zeta_{\alpha,k-2}+\varepsilon_{\alpha}}{\zeta_{\alpha,k-1}+\varepsilon_{\alpha}}.
$$
Note $\hat{\vartheta}_{\alpha,1}:=0.$
\end{enumerate}
\end{thm}
\begin{proof} First notice that (\ref{qspec}) can be written as,
$$
\frac{\tau_{\alpha}(\varepsilon_{\alpha,k-1}+\mathrm{e}_{k-1}+\zeta_{\alpha,k-2}
+\varepsilon_{\alpha})}{{(\mathrm{e}_{k-1}+\zeta_{\alpha,k-2}+\varepsilon_{\alpha})}^{1/\alpha}}
=\frac{\tau_{\alpha}(\zeta_{\alpha,k-1}+\varepsilon_{\alpha})}{{(\zeta_{\alpha,k-1}+\varepsilon_{\alpha})}^{1/\alpha}}\times q^{-1/\alpha}_{\alpha,k}.
$$
Comparing ~(\ref{jointEVG1}) with~(\ref{jointVG1}) shows that the conditional density of  $(q_{\alpha,2},\ldots,q_{\alpha,n},\tilde{S}_{\alpha,n})$  given $(q_{\alpha,1},\zeta)$ is equivalent, jointly and component-wise, to that of the conditional joint density of  $(V_{1},\ldots,V_{n-1},\hat{T}_{\alpha,n-1})$ given $\varepsilon_{\alpha}+\zeta,$ under a  $\mathrm{PG}(\alpha,\varepsilon_{\alpha}+\zeta)$ model. It remains to apply Theorem~\ref{TheoremPGV} for $\mathrm{PG}(\alpha,\varepsilon_{\alpha}+\zeta)$ variables.
\end{proof}
\subsection{BDGM Coag/Frag duality for $\mathrm{EPG}(\alpha,\zeta)$ models}\label{BDGMSection}
We now use Theorem~\ref{TheoremEPGV} to demonstrate that Proposition~\ref{JPPY} and surrounding results can be cast in terms of Markov chains induced by coagulation/fragmentation operations on
$\mathcal{P}_{\infty}$ as described in  BDGM~\cite{BertoinGoldschmidt2004,Dong2006}. and related results, explicitly for the $\mathrm{EPG}(\alpha,\zeta)$ case. As a reminder, we work with a bijective notion of operations on bridges. Coagulation in this setting equates to compositions of $\mathrm{PG}(\alpha,\zeta)$ bridges with simple bridges. See~Bertoin\cite{BerFrag} and  Bertoin and LeGall\cite{BerLegall03}.  Recall that,
$$
\tilde{\zeta}_{\alpha,k}=\zeta_{k}+\sum_{l=1}^{k}\varepsilon_{\alpha,l-1}.
$$
Theorem~\ref{TheoremEPGV} indicates a Markov chain on $\mathcal{P}_{\infty}$ with states described by the sequence of laws, $(\mathrm{EPG}(\alpha, \tilde{\zeta}_{\alpha,k-1})_{\{k\ge 1\}}),$ that is for each $n$
\begin{equation}
\mathrm{EPG}(\alpha,\zeta),\mathrm{EPG}(\alpha,\mathrm{e}_{1}+\varepsilon_{\alpha}+\zeta),\ldots,
\mathrm{EPG}(\alpha,\zeta_{n}+\sum_{l=0}^{n-1}\varepsilon_{\alpha,l})
\label{EPGBDGMchain}.
\end{equation}
Applying Proposition~\ref{recursiverelations}, (\ref{EPGBDGMchain}) can also be read as,
$$
\mathrm{EPG}(\alpha,\zeta),\mathrm{PG}(\alpha,\varepsilon_{\alpha}+\zeta),\ldots,
\mathrm{PG}(\alpha,\zeta_{\alpha,n-1}+\varepsilon_{\alpha})
$$
where $\tilde{\zeta}_{\alpha,n-1}+\varepsilon_{\alpha,n-1}=\zeta_{\alpha,n-1}+\varepsilon_{\alpha}.$ (\ref{EPGBDGMchain}) read left to right corresponds to a discrete fragmentation process. A coagulation process is then read right to left. A chain with states entirely in the $\mathrm{PG}(\alpha,\zeta)$
class of laws, specifically the collection $(\mathrm{PG}(\alpha, \zeta_{\alpha,k-1})_{\{k\ge 1\}}),$ is obtained by replacing $\zeta$ with $\gamma_{1}+\zeta.$ We first highlight a few pertinent points and constructions.

\begin{enumerate}
\item[(III)] \textsc{Constructions and Pertinent Points}
\item[(i)] Define simple bridges,
\begin{equation}
\tilde{\lambda}_{\alpha,k}(y)=q_{\alpha,k}\mathbb{U}(y)+(1-q_{\alpha,k})\indic_{\{U_{k}\leq y\}}
\label{qbridge}
\end{equation}
where $q_{\alpha,k}=\tilde{S}^{-\alpha}_{\alpha,k-1}/\tilde{S}^{-\alpha}_{\alpha,k},$ are specified by Theorem~\ref{TheoremEPGV}.
\item[(ii)]For each $k,$
$$
\tilde{P}^{\dagger}_{\alpha,\tilde{\zeta}_{\alpha,k-1}}=
\frac{\tau_{\alpha}(\varepsilon_{\alpha,k-1})}{({\tau_{\alpha}(\varepsilon_{\alpha,k-1})+
\tau_{\alpha}(\tilde{\zeta}_{\alpha,k-1})})},$$ is the first size biased pick
from a $\mathrm{EPG}(\alpha,\tilde{\zeta}_{\alpha,k-1})$ mass partition denoted as $(P^{(k-1)}_{i}).$
\item[(iv)]There is the decomposition $$(P^{(k-1)}_{i})=\mathrm{Rank}((P^{*}_{l,k-1}),\tilde{P}^{\dagger}_{\alpha,\tilde{\zeta}_{\alpha,k-1}}),$$ where $(P^{*}_{l,k-1})$ denotes the points of $(P^{(k-1)}_{i})$ remaining after selecting $\tilde{P}^{\dagger}_{\alpha,\tilde{\zeta}_{\alpha,k-1}}.$
\item[(ii)]Construct a sequence of iid $\mathrm{PD}(\alpha,1-\alpha)$ bridges as follows
$$
P^{(k-1)}_{\alpha,1-\alpha}(y)=\frac{\tau_{\alpha}(\varepsilon_{\alpha,k-1}y)}{\tau_{\alpha}(\varepsilon_{\alpha,k-1})}
$$
where $\varepsilon_{\alpha,0}=\varepsilon_{\alpha}.$
\item[(v)]Recall from (\ref{PYlamperti})
that $\tau_{\alpha}(\varepsilon_{\alpha,k-1})\overset{d}=\gamma_{1-\alpha}$ is independent of $P^{(k-1)}_{\alpha,1-\alpha}.$
\item[(vi)]From Proposition~\ref{recursiverelations}, $F_{\alpha,\tilde{\zeta}_{\alpha,k}}=Q_{\alpha,\varepsilon_{\alpha}+\zeta_{\alpha,k-1}}=Q_{\alpha,\varepsilon_{\alpha,k-1}+\tilde{\zeta}_{\alpha,k-1}}.$\qed
\end{enumerate}
\begin{prop}\label{bijectDGM}
Let $(F_{\alpha,\tilde{\zeta}_{\alpha,0}}, F_{\alpha,\tilde{\zeta}_{\alpha,1}},\ldots)$ denote a family of $(\mathrm{EPG}(\alpha,\tilde{\zeta}_{\alpha,k-1}))_{{k\ge 1}}$ bridges,
with corresponding $\alpha$-diversities $(\tilde{S}^{-\alpha}_{\alpha,0},\tilde{S}^{-\alpha}_{\alpha,1},\ldots)$ satisfying a Markov Chain described in Theorem~\ref{TheoremEPGV}. Then the bridges have the following properties,
\begin{enumerate}
\item[(i)]
$
F_{\alpha,\zeta}(y)=F_{\mathrm{e}_{1}+\zeta+\varepsilon_{\alpha}}(\tilde{\lambda}_{\alpha,1}(y))
$
and generally for $k=1,2,\ldots,$
\begin{equation}
F_{\alpha,\tilde{\zeta}_{\alpha,k-1}}(y)=F_{\alpha,\tilde{\zeta}_{\alpha,k}}(\tilde{\lambda}_{\alpha,k}(y))
\label{coagF}
\end{equation}
\item[(ii)]Thus for each $n,$ $\tilde{S}^{-\alpha}_{\alpha,0}=\tilde{S}^{-\alpha}_{\alpha,n}\times \prod_{j=1}^{n}q_{\alpha,j}$ and
$$
F_{\alpha,\zeta}=F_{\alpha,\tilde{\zeta}_{\alpha,n}}\circ \tilde{\lambda}_{\alpha,n}\circ\cdots\circ\tilde{\lambda}_{\alpha,1}.
$$
\item[(iii)]There is the following identity
\begin{equation}
F_{\alpha,\tilde{\zeta}_{\alpha,k}}(y)=
\frac{\tau_{\alpha}(\tilde{\zeta}_{\alpha,k-1}y)}
{\tau_{\alpha}(\varepsilon_{\alpha,k-1})+\tau_{\alpha}(\tilde{\zeta}_{\alpha,k-1})}
+\tilde{P}^{\dagger}_{\alpha,\tilde{\zeta}_{\alpha,k-1}}P^{(k-1)}_{\alpha,1-\alpha}(y),
\label{bridgefrag}
\end{equation}
where $P^{(k-1)}_{\alpha,1-\alpha}$ is independent of $F_{\alpha,\tilde{\zeta}_{\alpha,j}}$ for $j=1,\ldots,k-1.$
\end{enumerate}
\end{prop}

\begin{proof}Theorem~\ref{TheoremEPGV} serves as a blueprint to identify specifically the $\mathrm{EPG}(\alpha,\zeta)$ models involved, the results are then mainly concluded from our explicit construction of $\mathrm{EPG}(\alpha,\zeta)$ processes. One can further use Proposition~\ref{MarkovAlpha} to verify the correspondence between the families $(\hat{S}^{-\alpha}_{\alpha,k-1})_{\{k\ge 1\}},$ and $(F_{\alpha,\tilde{\zeta}_{\alpha,k-1}})_{\{k\ge 1\}}.$ [(iii)] relies on the construction~(\ref{EPGbridge2}), and the explicit construction of  $P^{(k-1)}_{\alpha,1-\alpha}$ which when multiplied by
$\tilde{P}^{\dagger}_{\alpha,\tilde{\zeta}_{\alpha,k-1}}$ leads to a cancellation of the $\tau_{\alpha}(\varepsilon_{\alpha,k-1})$ term.
\end{proof}
Now, for each $k=1,2,\ldots,$ let
$(P^{(k-1)}_{l})$
and $(P^{(k)}_{i}),$ living in $\mathcal{P}_{\infty},$ denote respectively the $\mathrm{EPG}(\alpha,\tilde{\zeta}_{\alpha,k-1})$ and $\mathrm{EPG}(\alpha,\tilde{\zeta}_{\alpha,k})$ mass partitions of $F_{\alpha,\tilde{\zeta}_{\alpha,k-1}},$ and $F_{\alpha,\tilde{\zeta}_{\alpha,k}}.$ Let $(U_{l,k})_{\{l\ge 1\}}$ denote the atoms of $F_{\alpha,\tilde{\zeta}_{\alpha,k}}$ and for each $l,$ define indicators
$I_{l,k}:=\indic_{\{\tilde{\lambda}^{-1}_{\alpha,k}(U_{l,k})=U_{k}\}}.$ which are conditionally iid $\mathrm{Bernoulli}(1-q_{\alpha,k}).$ We now summarize how Proposition~\ref{bijectDGM} describes  sequential  fragmentation and coagulation operations on $\mathcal{P}_{\infty}$

\begin{itemize}
\item \textsc{COAG} The relation~(\ref{coagF}) corresponds to the coagulation operation of \cite{BertoinGoldschmidt2004,Dong2006} as follows:
$$
(P^{(k-1)}_{i})=\mathrm{Rank}((P^{(k)}_{l}:I_{l,k}=0);\sum_{\{j:I_{j,k}=1\}}P^{(k)}_{j} ),
$$
for the input $(P^{(k)}_{l}).$
\item \textsc{FRAG}
The equation~(\ref{bridgefrag}) shows that $(P^{(k)}_{i}),$ the $\mathrm{EPG}(\alpha,\tilde{\zeta}_{\alpha,k})$ mass partition of $F_{\alpha,\tilde{\zeta}_{\alpha,k}},$ can be represented as
$$
(P^{(k)}_{i})=\mathrm{Rank}((P^{*}_{l,k-1});(Q^{(k-1)}_{l})\tilde{P}^{\dagger}_{\alpha,\tilde{\zeta}_{\alpha,k-1}}),
$$
where $(Q^{(k-1)}_{l})$ is the  $\mathrm{PD}(\alpha,1-\alpha)$ mass partition of  $P^{(k-1)}_{\alpha,1-\alpha}$ independent of $(P^{(k-1)}_{i}).$ This description of
$(P^{(k)}_{i})$ corresponds to the $\mathrm{PD}(\alpha,1-\alpha)$-fragmentation,  operation in \cite{BertoinGoldschmidt2004,Dong2006}, applied to the input $(P^{(k-1)}_{i}).$
\end{itemize}

We now describe some specifics in the $\mathrm{PD}(\alpha,\theta)$ case, which corresponds to the chain $\mathrm{PD}(\alpha,\theta),\mathrm{PD}(\alpha,\theta+1),\ldots$
\begin{cor}\label{CorollaryPDFrag} Setting $\zeta=\gamma_{(\theta+\alpha)/\alpha},$
in Proposition~\ref{bijectDGM} leads to results for $\mathrm{PD}(\alpha,\theta)$.  For $\ell=k-1=0,1,2,\ldots,$
\begin{equation}
q_{\alpha,\ell+1}=\tilde{S}^{-\alpha}_{\alpha,\ell}/\tilde{S}^{-\alpha}_{\alpha,\ell+1}:=S^{-\alpha}_{\alpha,\theta+\ell}/S^{-\alpha}_{\alpha,\theta+\ell+1}:=\beta_{\left(\frac{\theta+\ell+\alpha}{\alpha},\frac{1-\alpha}{\alpha}\right)}
\label{qrecursePD}
\end{equation}
are independent and define independent simple bridges $(\tilde{\lambda}_{\alpha,\ell+1}),$ such that for each fixed $l,$ $\tilde{\lambda}_{\alpha,\ell+1}$ is independent of $P_{\alpha,\theta+\ell+1},$ and hence its $\alpha$-diversity denoted as $\tilde{S}^{-\alpha}_{\alpha,\ell+1}:=S^{-\alpha}_{\alpha,\theta+\ell+1}.$ The relations in (\ref{coagF}) and (\ref{bridgefrag}) specialize to
\begin{eqnarray*}
\nonumber P_{\alpha,\theta+\ell}(y)&=&  P_{\alpha,\theta+\ell+1}\left(\beta_{\left(\frac{\theta+\ell+\alpha}{\alpha},\frac{1-\alpha}{\alpha}\right)}\mathbb{U}(y)+(1-\beta_{\left(\frac{\theta+\ell+\alpha}{\alpha},\frac{1-\alpha}{\alpha}\right)})\indic_{\{U_{\ell+1}\leq y\}}\right)\\
                   & = & \beta_{\theta+\alpha+\ell,1-\alpha}P_{\alpha,\theta+\alpha+\ell}(y)+(1-\beta_{\theta+\alpha+\ell,1-\alpha})\indic_{\{U_{\ell+1}\leq y\}},
\end{eqnarray*}
and
$
P_{\alpha,\theta+\ell+1}(y) =\beta_{\theta+\alpha+\ell,1-\alpha}P_{\alpha,\theta+\alpha+\ell}(y)+(1-\beta_{\theta+\alpha+\ell,1-\alpha})
                   P^{(\ell)}_{\alpha,1-\alpha}(y),
$
where terms on the right-hand side are independent of one another.
There is the relationship among the  the $\alpha$-diversities, of  $(\mathrm{PD}(\alpha,\theta+\ell),\mathrm{PD}(\alpha,1+\theta+\ell),\mathrm{PD}(\alpha,\theta+\alpha+\ell)),$
\begin{equation}
S^{-\alpha}_{\alpha,\theta+\ell}=S^{-\alpha}_{\alpha,\theta+\ell+1}\times \beta_{\left(\frac{\theta+\alpha+\ell}{\alpha},\frac{1-\alpha}{\alpha}\right)}=
S^{-\alpha}_{\alpha,\theta+\alpha+\ell}\times \beta^{\alpha}_{\theta+\alpha+\ell,1-\alpha},
\label{JPPYequality2}
\end{equation}
where respective terms in the products are independent. Furthermore
$$
P_{\alpha,\theta+\ell+1}(\beta_{\left(\frac{\theta+\alpha+\ell}{\alpha},\frac{1-\alpha}{\alpha}\right)})= \beta_{\theta+\alpha+\ell,1-\alpha}.
$$
is one minus the size biased pick from $\mathrm{PD}(\alpha,\theta+\ell)$
\end{cor}

\begin{rem}
Setting $\zeta=0,$ yields $\mathrm{EPG}(\alpha,0):=\mathrm{PD}(\alpha,-\alpha),$ and states represented by the sequence
$$
\mathrm{PD}(\alpha,-\alpha),\mathrm{PD}(\alpha,1-\alpha),\mathrm{PD}(\alpha,2-\alpha),\ldots
$$
where $\mathrm{PD}(\alpha,-\alpha)((1,0,0,\ldots))=1.$ This indicates a $\mathrm{PD}(\alpha,1-\alpha)$ fragmentation procedure with natural initial state $(1,0,0,\ldots).$
\end{rem}
\begin{rem} It is interesting to note how the sequence of iid $\mathrm{PD}(\alpha,1-\alpha)$ \emph{fragmenting} bridges $(P^{(k-1)}_{\alpha,1-\alpha})$ have been carefully constructed from independent $\mathrm{Gamma}(1-\alpha)$ variables $(\tau_{\alpha}(\varepsilon_{\alpha,k-1})),$ where a nice interpretation of these, in terms of Ages of excursions at marked independent exponential times,  should be obtainable from the Big Poisson formulation in \cite[Section 3]{PY92}.
\end{rem}
\begin{rem}
Furthermore while these bridges are independent and identically distributed amongst themselves they, as should be expected, do not have the same joint behavior among other variables. For instance, in the $\mathrm{PD}(\alpha,\theta)$ setting,$ P^{(0)}_{\alpha,1-\alpha}$ is independent of $P_{\alpha,\theta}$ but not of the subsequent bridges $(P_{\alpha,\theta+k})_{\{k\ge 1\}}.$
\end{rem}

\begin{rem}\label{tagtime}Reading~(\ref{EPGBDGMchain}). and noting that $\tilde{\zeta}_{\alpha,k}-\tilde{\zeta}_{\alpha,k=1}=\mathrm{e}_{k}+\varepsilon_{\alpha,k-1}\overset{d}=\gamma_{1/\alpha},$ not depending on $\zeta,$ shows that the time of transitions can be tagged to independent exponential waiting times defined for $k=1,2,\ldots$ as
$$
E_{k}=\tau_{\alpha}(\tilde{\zeta}_{\alpha,k})-\tau_{\alpha}(\tilde{\zeta}_{\alpha,k-1})=\tau_{\alpha}(\mathrm{e}_{k}+\varepsilon_{\alpha,k-1})\overset{d}=\gamma_{1}.
$$
Leading to a continuous time description of the fragmentation and coagulation Markov chains.
\end{rem}
\subsection{Approximating $\mathrm{EPG}(\alpha,\zeta)$-bridges by flows of simple bridges}
Statement [(ii)] of Proposition~\ref{bijectDGM} shows that for $n=1,2,\ldots,$
\begin{equation}
F_{\alpha,\zeta}=F_{\alpha,\tilde{\zeta}_{\alpha,n}}\circ \tilde{\lambda}_{\alpha,n}\circ\cdots\circ\tilde{\lambda}_{\alpha,1}.
\label{Fid}
\end{equation}
Removing $F_{\alpha,\tilde{\zeta}_{\alpha,n}}$ leaves a composition of simple bridges
$$
\Lambda^{(n)}_{\alpha,\zeta}:=\tilde{\lambda}_{\alpha,n}\circ\cdots\circ\tilde{\lambda}_{\alpha,1}
$$
that is similar to the concept of flows of simple bridges discussed in~\cite{BerFrag,BerLegall00,BerLegall03} and also iterated random functions as discussed in Diaconis and Freedman~\cite{DiaconisFreedmanIt}. See also \cite[Section 2.3, Remarks a,b]{BertoinGoldschmidt2004} for related comments.
Here, it is shown that they converge weakly to a random measure having the same distribution as $F_{\alpha,\zeta},$ where weak convergence is in the usual sense for random measures on $[0,1].$ One can see for instance \cite{James2008} for specifics, as we shall use that work in the  next result. See also \cite[Section 4]{BerFrag}.
\begin{prop}\label{iteratedlimit}Let $(\tilde{\lambda}_{\alpha,k})$ denote the collection of simple bridges as defined in~(\ref{qbridge}) with $(q_{\alpha,k})$ specified by Theorem~\ref{TheoremEPGV}. That is the $\mathrm{EPG}(\alpha,\zeta)$ case. Then the exchangeable bridge defined as
$$
\Lambda^{(n)}_{\alpha,\zeta}:=\tilde{\lambda}_{\alpha,n}\circ\cdots\circ\tilde{\lambda}_{\alpha,1}
$$
converges weakly to $F_{\alpha,\zeta},$ an $\mathrm{EPG}(\alpha,\zeta)$-bridge, as $n\rightarrow \infty.$ This is written as $\lim_{n\rightarrow\infty}\Lambda^{(n)}_{\alpha,\zeta}(\cdot)\overset{d}=F_{\alpha,\zeta}(\cdot).$
\end{prop}
\begin{proof}
Notice that~(\ref{Fid}) is an equality that holds for each $n,$ and otherwise this is a relationship between distribution functions. Focusing on $F_{\alpha,\tilde{\zeta}_{\alpha,n}},$ it is not difficult to see that this process behaves asymptotically like $P_{\alpha,n}.$ Hence applying James~\cite[Proposition 2.3]{James2008} shows that $\lim_{n\rightarrow\infty}F_{\alpha,\tilde{\zeta}_{\alpha,n}}(\cdot)\overset{d}=\mathbb{U}(\cdot).$ Concluding the result.
\end{proof}
\begin{rem} Note that the result above leaves some wiggle room for the approximation of objects having an $\mathrm{EPG}(\alpha,\zeta)$ law but not necessarily arising from the constructions we have discussed. Obviously,
$$
F^{-1}_{\alpha,\zeta}=\tilde{\lambda}^{-1}_{\alpha,1}\circ \tilde{\lambda}^{-1}_{\alpha,2}\circ\cdots\circ \mathbb{U},
$$
which indicates a coalescent process starting at the trivial partition, i.e. all singletons, on $\mathbb{N}.$
\end{rem}
\subsubsection{Approximating Pitman-Yor processes aka $\mathrm{PD}(\alpha,\theta)$ bridges.}
Proposition~\ref{iteratedlimit} indicates that when $\zeta\overset{d}=\gamma_{(\theta+\alpha)/\alpha},$  $\lim_{n\rightarrow\infty}\Lambda^{(n)}_{\alpha,\zeta}(\cdot)\overset{d}=P_{\alpha,\theta}(\cdot),$ for $\theta>-\alpha.$  Thus indicating another method to approximate Pitman-Yor processes. These also are potentially useful for novel applications involving  hierarchical structures built from simple bridges with explicit limits. The same can be said of more general $\mathrm{EPG}(\alpha,\zeta)$-bridges. However, in this $\mathrm{PD}(\alpha,\theta)$ setting one can obtain much more explicit results as shown next.
As a reminder, here for $k=1,2,\ldots$
$$
\tilde{\lambda}_{\alpha,k}(y)=\beta_{\left(\frac{\theta+k-1+\alpha}{\alpha},\frac{1-\alpha}{\alpha}\right)}\mathbb{U}(y)+(1-\beta_{\left(\frac{\theta+k-1+\alpha}{\alpha},\frac{1-\alpha}{\alpha}\right)})\indic_{\{U_{k}\leq y\}}.
$$

\begin{prop}\label{idposterior} Let $K_{n}$ denote the random number of blocks of an
$\mathrm{PD(}\alpha,\theta)$-partition of $[n]=\{1,2,\ldots,n\},$ and consider the
random process
$$
b^{(K_{n})}_{\alpha,\theta}(\cdot)=\beta_{(\frac{\theta}{\alpha}+K_{n},\frac{n}{\alpha}-K_{n})}\mathbb{U}(\cdot)+\sum_{j=1}^{K_{n}}S_{j:n}\indic_{(U_{j}\leq\cdot)}
$$
where, conditioned on $K_{n}=k$, the vector
$(\beta_{(\frac{\theta}{\alpha}+k,\frac{n}{\alpha}-k)},
(S_{j:n})_{j=1}^{k}))$  has law $\mathrm{Dir}(\theta/\alpha+k,(n_{1}-\alpha)/\alpha,\ldots,(n_{k}-\alpha)/\alpha),$  denoting a Dirichlet vector with parameters as indicated.
$(n_{1},\ldots,n_{k})$ are the sizes of the $k$ blocks.
Hence,
$\beta_{(\frac{\theta}{\alpha}+K_{n},\frac{n}{\alpha}-K_{n})}$
represents the total mass of \emph{dust} of a random mass
partition with $K_{n}$ non-zero elements.
\begin{enumerate}
\item[(i)]$P_{\alpha,\theta}\overset{d}=P_{\alpha,\theta+n}\circ
b^{(K_{n})}_{\alpha,\theta}.$
\item[(ii)]Hence,
$$
b^{(K_{n})}_{\alpha,\theta}\overset{d}=\Lambda^{(n)}_{\alpha,\gamma_{\frac{\theta+\alpha}
{\alpha}}}.$$
\item[(iii)]For $q_{k}\overset{d}=\beta_{(\theta+\alpha+k-1)/\alpha,(1-\alpha)/\alpha)},$ there is the identity,
$$
\beta_{(\frac{\theta}{\alpha}+K_{n},\frac{n}{\alpha}-K_{n})}\overset{d}=\prod_{k=1}^{n}q_{k}.
$$
\item[(iv)]$\lim_{n\rightarrow\infty}b^{(K_{n})}_{\alpha,\theta}(\cdot)\overset{d}=P_{\alpha,\theta}(\cdot).$
\item[(v)]When $\alpha=0,$ $q_{k}=(\theta+k-1)/(\theta+k),$
\begin{equation}
b^{(K_{n})}_{0,\theta}(\cdot)\overset{d}=\frac{\theta}{\theta+n}\mathbb{U}(\cdot)+\frac{1}{\theta+n}\sum_{i=1}^{n}\indic_{(\tilde{U}_{i}\leq\cdot)}.
\label{BMQ}
\end{equation}
where the distribution of
$\mathbf{U}={(\tilde{U}_{1},\ldots,\tilde{U}_{n})}$ is the
exchangeable Blackwell-MacQueen~\cite{Blackwell} P\'olya urn
distribution.
$$
\pi_{\mathbf{U}}(du_{1},\ldots,du_{n})=\frac{\Gamma(\theta)}{\Gamma(\theta+n)}\prod_{i=1}^{n}
(\theta\mathbb{U}+\sum_{j=1}^{i-1}\delta_{u_{j}})(du_{i}).
$$
\end{enumerate}
\end{prop}
\begin{proof}
From the description of the posterior distribution of $P_{\alpha,\theta}$ in Pitman~\cite{Pit96} one can obtain a mixture representation of $P_{\alpha,\theta}$ as follows,
\begin{equation}
P_{\alpha,\theta}(\cdot)\overset{d}=R_{K_{n}}P_{\alpha,\theta+K_{n}\alpha}(\cdot)+\sum_{j=1}^{K_{n}}\tilde{P}_{j:n}\indic_{(U_{j}\leq
\cdot)} \label{decomp1}
\end{equation}
where $
R_{K_{n}}=(1-\sum_{j=1}^{K_{n}}\tilde{P}_{j:n})\overset{d}=\beta_{\theta+K_{n}\alpha,n-K_{n}\alpha},$
and conditioned on the data, with $K_{n}=k,$
$P_{\alpha,\theta+k\alpha}$ is a
$\mathrm{PD}(\alpha,\theta+k\alpha)$-bridge and is independent of
the random Dirichlet vector,
$$
(R_{k},\tilde{P}_{1,n},\ldots,\tilde{P}_{k,n})\sim
\mathrm{Dir}(\theta+k\alpha,n_{1}-\alpha,\ldots,n_{k}-\alpha).
$$
By construction, similar to Section~\ref{construction}, it follows that the expressions in [(i)] are equivalent in distribution to~(\ref{decomp1}). [(ii)] now follows from
[(i)]. One can obtain [(iii)] by checking moments, there is also an identity deduced from~\cite{HJL} that gives $S^{-\alpha}_{\alpha,\theta}\overset{d}=S^{-\alpha}_{\alpha,\theta+n}\times \beta_{(\frac{\theta}{\alpha}+K_{n},\frac{n}{\alpha}-K_{n})}.$ Comparing this with
(\ref{jamesid}) and using the fact that $S_{\alpha,\theta+n}$ is \emph{simplifiable}, see~ \cite{Chaumont} for this concept, concludes the result. [(iv)] follows from [(ii)].
Item[(v)] can be obtained indirectly by taking limits as $\alpha\rightarrow 0$
and otherwise using ideas about mixture representations derived
from posterior distributions of a Dirichlet process as discussed in \cite{James2005b}.
\end{proof}
\section{Stick-breaking: revisiting Perman, Pitman and Yor}

\subsection{The PPY~\cite{PPY92} Markov chain in the case of $\mathrm{PK}_{\alpha}(h\cdot f_{\alpha})$}
From~\cite{PPY92}, we know the following. For $(P_{i})\sim \mathrm{PD}(\alpha,0),$ set $\hat{T}_{\alpha,0}=T_{\alpha}(1)\overset{d}=S_{\alpha}.$ furthermore let $(J_{i})$ denote the size-biased re-arrangement of the ranked jumps $(\Delta_{i})$, such that $\hat{T}_{\alpha,0}=\sum_{k=1}^{\infty}\Delta_{k}=\sum_{k=1}^{\infty}J_{k}$ and set $\hat{T}_{\alpha,k\alpha}=\hat{T}_{\alpha,0}-\sum_{l=1}^{k}J_{k}.$ Where under $\mathrm{PD}(\alpha,0),$ $\hat{T}_{\alpha,k\alpha}\overset{d}=S_{\alpha,k\alpha}.$ Then the stick-breaking weights $W_{k}$ satisfy the general relation
\begin{equation}
\hat{T}_{\alpha,(k-1)\alpha}=\hat{T}_{\alpha,k\alpha}\times W^{-1}_{k}
\label{stickfundamental}
\end{equation}
where $W^{-1}_{k}\overset{d}=\beta_{\alpha k,1-\alpha}$ and is independent of $\hat{T}_{\alpha,k\alpha}.$

The next result is merely a rephrasing of Perman, Pitman, Yor \cite{PPY92}.
\begin{prop}[Perman, Pitman, Yor \cite{PPY92}]\label{PPYtheorem}
Let $(\hat{T}_{\alpha,0}, \hat{T}_{\alpha,\alpha},\ldots,\hat{T}_{\alpha,n\alpha}),$ for each integer $n,$ denote a vector of random variables, defined as
 above under a $\mathrm{PD}(\alpha,0)$ model, such that $\hat{T}_{\alpha,0}\overset{d}=S_{\alpha}$ and satisfies the relationship for each integer $k,$
\begin{equation}
\hat{T}_{\alpha,(k-1)\alpha}=\hat{T}_{\alpha,k\alpha}\times W^{-1}_{k}
\label{Uequation}
\end{equation}
where $W_{k}\overset{d}=\beta_{k\alpha,1-\alpha}$ independent of $\hat{T}_{\alpha,k\alpha}$ and marginally
$\hat{T}_{\alpha,k\alpha}\overset{d}=S_{\alpha,k\alpha}.$ Then, it is easy to check, that the conditional distribution of $\hat{T}_{\alpha,k\alpha}$ given  $\hat{T}_{\alpha,(k-1)\alpha}=t$ is the same for all $k$ and equates to the density,
\begin{equation}
P(\hat{T}_{\alpha,\alpha}\in ds|\hat{T}_{\alpha,0}=t)/ds=\frac{\alpha}{\Gamma(1-\alpha)}\frac{(t-s)^{-\alpha}f_{\alpha}(s)}{tf_{\alpha}(t)},
\label{transitionU}
\end{equation}
$s<t.$
By the change of variable $w=s/t,$
\begin{equation}
P(W_{1}\in dw|\hat{T}_{\alpha,0}=t)/dw=\frac{\alpha}{\Gamma(1-\alpha)}\frac{{(1-w)}^{-\alpha}t^{-\alpha}f_{\alpha}(wt)}{f_{\alpha}(t)}.
\label{transitionW}
\end{equation}
Furthermore, for each $n,$ $(W_{1},\ldots,W_{n})$ are independent variables, independent of $\hat{T}_{\alpha,n\alpha}.$
Hence, the sequence is a time-homogeneous Markov Chain governed by a $PD(\alpha,0)$ law.
\end{prop}

\begin{cor}\label{corollary2}
As consequences of Proposition~\ref{PPYtheorem} the distribution of the quantities above with respect to a $PK(\rho_{\alpha};h\cdot f_{\alpha})$ are given by (\ref{transitionW}) and specifying $\hat{T}_{\alpha,0}$ to have density $h(t)f_{\alpha}(t).$ In particular, the joint law of $(W_{1},\ldots,W_{n},\hat{T}_{\alpha,n\alpha})$ is given by,
\begin{equation}
\left[\prod_{k=1}^{n}f_{B_{k}}(w_{k})\right]h(s/\prod_{l=1}^{n}w_{l})f_{\alpha,n\alpha}(s)ds
\label{JointW}
\end{equation}
where $f_{B_{k}}$ denotes the density of a $\beta_{k\alpha,1-\alpha},$ variable.
\end{cor}
\section{Stick-breaking results for $\mathrm{PG}(\alpha,\zeta)$ and $\mathrm{EPG}(\alpha,\zeta)$}\
Here we use the notation
$(\tilde{W}_{1},\ldots,\tilde{W}_{n},\tilde{S}_{\alpha,n\alpha})$ to denote the variables $(W_{1},\ldots,W_{n},\hat{T}_{\alpha,n\alpha})$ under a
$\mathrm{EPG}(\alpha,\zeta)$ distribution and maintain  $(W_{1},\ldots,W_{n},\hat{T}_{\alpha,n\alpha})$ for the  $\mathrm{PG}(\alpha,\zeta)$ case.

Throughout, set $\zeta_{0}=\zeta$ and for $k=1,2,\dots$ set
\begin{equation}
\zeta_{k}=\mathrm{e}_{k}+\zeta_{k-1}=\sum_{l=1}^{k}\mathrm{e}_{l}+\zeta.
\label{zetarep}
\end{equation}
Similar to the role of Lemma~\ref{lemmaphi}, the next result provides an important component to the stick-breaking representation and analysis. The proof follows quite obviously from elementary conditioning arguments and is omitted.
\begin{lem}\label{Rkdist}Let $\zeta_{j}$ be random variables described in~(\ref{zetarep}) then for any $\zeta>0,$ the conditional joint density of the vector
$$(R_{1},\dots,R_{n})\overset{d}=(({\zeta}/{\zeta_{1}})^{\frac{1}{\alpha}},\ldots,({\zeta_{n-1}}/{\zeta_{n}})^{\frac{1}{\alpha}})$$ given $\zeta,$ is given by
$$
\alpha^{n}\zeta^{n}{\mbox e}^{\zeta}{\mbox e}^{-\zeta/(\prod_{l=1}^{n}y_{l})^{\alpha}}\prod_{l=1}^{n}y^{-(n-l+1)\alpha-1}_{l}
$$
\begin{enumerate}
\item[(i)]When $\zeta\overset{d}=\gamma_{\theta/\alpha},$ corresponding to the $\mathrm{PD}(\alpha,\theta)$ case for $\theta\ge 0,$ it follows that the
$R_{k}$ are independent with
$$R_{k}\overset{d}=\beta_{\theta+(k-1)\alpha,1}
$$
 for $k=1,2,\ldots$.
\item[(ii)]The case $\zeta=0,$ is equivalent to setting $\theta=0.$ Hence $R_{1}=0,$ and otherwise
$$R_{k}\overset{d}=\beta_{(k-1)\alpha,1}.
$$
Note, again, that conditioning in the case $\zeta=0$ not only corresponds to the $\mathrm{PD}(\alpha,0)$ model but any $\mathrm{PG}(\alpha,\zeta)$ model where $P(\zeta=0)>0.$
\end{enumerate}
\end{lem}

\begin{lem}\label{intlem}For positive quantities $(v,t,\zeta,x)$ there is the integral identity,
$$
\zeta\int_{0}^{x}{(x-y)}^{\alpha-1}{\mbox e}^{-\frac{\zeta^{1/\alpha}t}{vy}}y^{-\alpha-1}dy=\Gamma(\alpha)v^{\alpha}t^{-\alpha}x^{\alpha-1}{\mbox e}^{-\frac{\zeta^{1/\alpha}t}{vx}}.
$$
\end{lem}
\begin{proof}Apply the change of variables $r=1/y$ and then $s=rx-1$ to obtain a gamma integral.
\end{proof}
We now describe the specialization of Corollary~\ref{corollary2} to the $\mathrm{PG}(\alpha,\zeta)$ setting.
\begin{prop}\label{Wdensity}
Consider the setting in Proposition~\ref{PPYtheorem} then in the $\mathrm{PG}(\alpha,\zeta)$ case the  joint density of $(W_{1},\ldots,W_{n},\hat{T}_{\alpha,n\alpha})$ conditioned on $\zeta,$ is expressed as follows
\begin{equation}
\left[\prod_{k=1}^{n}f_{B_{k}}(w_{k})\right]{\mbox e}^{\zeta}{\mbox e}^{-(\frac{s\zeta^{1/\alpha}}{\prod_{l=1}^{n}w_{l}})}f_{\alpha,n\alpha}(s)
\label{jointVG3}
\end{equation}
and the joint density of $(W_{1},\ldots,W_{n})$ can be expressed as
\begin{equation}
\left[\prod_{k=1}^{n}f_{B_{k}}(w_{k})\right]{\mbox e}^{\zeta}\mathbb{S}_{\alpha,n\alpha}\left({\frac{\zeta^{1/\alpha}}{\prod_{l=1}^{n}w_{l}}}\right)
\label{jointVG4}
\end{equation}
where $f_{B_{k}}$ denotes the density of a $\beta_{k\alpha,1-\alpha},$ variable. See (\ref{Sid}) for various forms of $\mathbb{S}_{\alpha,n\alpha}(y).$
\begin{enumerate}
\item[(i)]When $\zeta=0,$ it follows from~(\ref{jointVG3}), (\ref{jointVG4}), that for $k=1,2,\ldots$
$$
W_{k}\overset{d}=\beta_{k\alpha,1-\alpha}
$$
and are independent.\qed
\end{enumerate}
\end{prop}

\begin{prop}\label{W2density}
Consider the setting in Proposition~\ref{PPYtheorem} then in the $\mathrm{EPG}(\alpha,\zeta)$ case the  joint density of $(\tilde{W}_{1},\ldots,\tilde{W}_{n},\tilde{S}_{\alpha,n\alpha})$ conditioned on $\zeta,$ is expressed as follows

$$
\zeta^{1/\alpha-1}(t^{*}/\alpha)\left[\prod_{k=1}^{n}f_{B_{k}}(v_{k})\right]{\mbox e}^{-(
t^{*}\zeta^{1/\alpha}-\zeta)}f_{\alpha,n\alpha}(t)
$$
where $t^{*}=t/(\prod_{l=1}^{n}v_{l}),$ and $B_{k}=\beta_{k\alpha,1-\alpha}.$
\begin{enumerate}
\item[(i)]For each $n,$ let  $(\tilde{W}_{2},\ldots,\tilde{W}_{n+1},\tilde{S}_{\alpha,
(n+1)\alpha})$ be determined by $\mathrm{EPG}(\alpha,\zeta)$ and let  $(W_{1},\ldots,W_{n},\hat{T}_{\alpha,n\alpha})$ be determined by $\mathrm{PG}(\alpha,\zeta).$ Then there is the joint equivalence (component-wise) in distribution
$$
\left(\tilde{W}_{2},\ldots,\tilde{W}_{n+1},\tilde{S}_{\alpha,
(n+1)\alpha}\right)\overset{d}=\left(W_{1},\ldots,W_{n},\hat{T}_{\alpha,n\alpha}\right)
$$
\item[(ii)]As a consequence of [(i)] the conditional density of $\tilde{W}_{1}$ given $\left(\tilde{W}_{2},\ldots,\tilde{W}_{n},\tilde{S}_{\alpha,
n\alpha}\right)=(w_{2},\ldots,w_{n},s)$ and $\zeta,$ is
\begin{equation}
\frac{1}{\Gamma(1-\alpha)}w^{\alpha-2}{{(1-w)}^{-\alpha}r^{1-\alpha}
{\mbox e}^{-r\frac{1-w}{w}}}
\label{wmixdensity2}
\end{equation}
where $r=s\zeta^{1/\alpha}/\prod_{k=2}^{n}w_{k}.$ It follows that~(\ref{wmixdensity2}) also corresponds to the conditional density of $\tilde{W}_{1}$ conditioned on $\zeta^{1/\alpha}\tilde{S}_{\alpha,n\alpha}/\prod_{k=2}^{n}\tilde{W}_{k}=r.$ Note furthermore that by definition $\zeta^{1/\alpha}\tilde{S}_{\alpha,n\alpha}/\prod_{k=2}^{n}\tilde{W}_{k}=\tau_{\alpha}(\zeta).$ This can be verified by comparing with~(\ref{wmixdensity}).

\end{enumerate}
\end{prop}
\begin{proof}
Statement [(i)] is proved as follows. Note the joint density of

 $(\tilde{W}_{1},\ldots,\tilde{W}_{n+1},\tilde{S}_{\alpha,(n+1)\alpha}),$ can be expressed as,
$$
\zeta^{1/\alpha-1}(t^{*}/\alpha)\left[\prod_{k=1}^{n+1}f_{B_{k}}(v_{k})\right]{\mbox e}^{-(
t^{*}\zeta^{1/\alpha}-\zeta)}f_{\alpha,(n+1)\alpha}(t)
$$
where $t^{*}=t/(v_{1}\prod_{l=2}^{n+1}v_{l}),$ and $B_{k}=\beta_{k\alpha,1-\alpha}.$ Use the fact that,
$$
1-\tilde{W}_{1}\overset{d}=\frac{\gamma_{1-\alpha}}{\gamma_{1-\alpha}+\tau_{\alpha}(\zeta)}.
$$
Hence it follows that conditioned on $\tau_{\alpha}(\zeta)=s$, the density of $\tilde{W}_{1}$ has the form
\begin{equation}
\frac{1}{\Gamma(1-\alpha)}w^{\alpha-2}{{(1-w)}^{-\alpha}s^{1-\alpha}
{\mbox e}^{-s\frac{1-w}{w}}}
\label{wmixdensity}
\end{equation}
Now integrating over $v_{1},$ using the identity gained from the density in~(\ref{wmixdensity}), and comparing with~(\ref{jointVG3}), leads to the joint density of $({W}_{1},\ldots,{W}_{n},\hat{T}_{\alpha,n\alpha}).$
\end{proof}
\subsection{$\mathrm{PG}(\alpha,\zeta)$ simple sticks}
We now formally prove the stick-breaking representation for the  $\mathrm{PG}(\alpha,\zeta)$ case, which remarkably does not directly involve the
appearance of $\tau_{\alpha}.$
\begin{thm}\label{TheoremPGW} If $(P_{i})\sim\mathrm{PG}(\alpha,\zeta),$ then the sequence $(\tilde{P}_{k}),$ obtained by size-biased sampling from $(P_{i}),$
can be represented as $\tilde{P}_{k}=(1-{W}_{k})\prod_{l=1}^{k-1}W_{l},$ where for each $k,$
$\tilde{P}^{\dagger}_{\alpha,\zeta_{k}}=1-W_{k}$ is the first biased pick from an $\mathrm{EPG}(\alpha,\zeta_{k})=
\mathrm{PG}(\alpha,\zeta_{k-1})$  mass partition
$(P_{l,k-1}).$  The  $(W_{k})$ are generally dependent random variables represented as
\begin{eqnarray}
W_{k}&=& 1-\beta^{(k)}_{({1-\alpha},\alpha)}[1-({\zeta_{k-1}}/{\zeta_{k}})^{\frac{1}{\alpha}}] \\
  & =& 1-\beta^{(k)}_{({1-\alpha},\alpha)}[1-R_{k}]\nonumber
  \label{weights}
\end{eqnarray}
for $(\beta^{(k)}_{(1-\alpha,\alpha)})$ iid $\mathrm{Beta}(1-\alpha,\alpha)$ variables independent of the $R_{k}={(\zeta_{k-1}/\zeta_{k})}^{1/\alpha}$ for  $(\zeta_{k})$  defined  in (\ref{zetarep}).
\end{thm}
\begin{proof}
Setting $R_{k}=({\zeta_{k-1}}/{\zeta_{k}})^{\frac{1}{\alpha}};$ it follows that the conditional joint density of $(W_{1},\ldots,W_{n}),$ specified in the statements above, given $(R_{1},\ldots,R_{n},\zeta)$ is
$$
\prod_{i=1}^{n}C(1-w_{i})^{-\alpha}(w_{i}-y_{i})^{\alpha-1}\indic_{\{0<y_{i}<w_{i}<1\}}
$$
where $1/C=\Gamma(1-\alpha)\Gamma(\alpha).$
Now notice that
$$
{\mbox e}^{-\zeta/(\prod_{l=1}^{n}y_{l})^{\alpha}}=\int_{0}^{\infty}{\mbox e}^{-\frac{\zeta^{1/\alpha}t}{(\prod_{l=1}^{n}y_{l})}}f_{\alpha}(t)dt.
$$
Augmenting this expression, the result is obtained by showing that
$$
\int_{0}^{w_{1}}\cdots\int_{0}^{w_{n}}
\zeta^{n}
{\mbox e}^{-\frac{\zeta^{1/\alpha}t}{(\prod_{l=1}^{n}y_{l})}}\prod_{l=1}^{n}(w_{l}-y_{l})^{\alpha-1}y^{-(n-l+1)\alpha-1}_{l}dy_{n}\ldots dy_{1}
$$
is equal to
$$
{[\Gamma(\alpha)]}^{n}t^{-n\alpha}{\mbox e}^{-\frac{\zeta^{1/\alpha}t}{(\prod_{l=1}^{n}w_{l})}}\prod_{l=1}^{n}w^{l\alpha-1}_{l}.
$$
But this follows from repeated applications of Lemma~\ref{intlem}, starting with $y_{n}$ and initially setting $v=\prod_{i=1}^{n-1}y_{i}.$ Then integrate with respect to $y_{n-1}$, setting $v=w_{n}\prod_{i=1}^{n-2}y_{i}$ and so on. Showing agreement with the densities in 
(\ref{jointVG3}) and~(\ref{jointVG4}).
\end{proof}

Notice that the proof reveals a joint density of $(W_{1},\ldots,W_{n},R_{1},\ldots,R_{n},\hat{T}_{\alpha,n\alpha})$ given $\zeta$ proportional to 
$$
{\mbox e}^{-\frac{\zeta^{1/\alpha}t}{(\prod_{l=1}^{n}r_{l})}}\prod_{l=1}^{n}
{(1-w_{i})}^{-\alpha}
(w_{l}-r_{l})^{\alpha-1}r^{-(n-l+1)\alpha-1}_{l}f_{\alpha}(t)
$$
for $r_{i}<w_{i},$ $i=1,\dots,n.$ Furthermore there is the relation, $\zeta_{k}=\zeta/{(\prod_{l=1}^{k}R_{l})}^{\alpha}.$ We describe a simple Corollary.
\begin{cor}Given $(R_{k},\zeta_{k}),$ $W_{k}$ and $\hat{T}_{\alpha,k\alpha}$ are conditionally independent with joint density
$$
\frac{\sin(\pi \alpha)}{\pi}(1-w_{k})^{-\alpha}(w_{k}-r_{k})^{\alpha-1}{\mbox e}^{-(s\zeta_{k}^{1/\alpha}-\zeta_{k})}f_{\alpha}(s)
$$
for $r_{k}<w_{k}.$
\end{cor}
We now show how to recover the stick-breaking result in the $\mathrm{PD}(\alpha,\theta)$ case for $\theta\ge 0,$ in terms of $(R_{k}).$
\begin{cor}When $\zeta=\gamma_{\theta/\alpha}$ for $\theta\ge0,$ $\mathrm{PG}(\alpha,\gamma_{\theta/\alpha})=\mathrm{PD}(\alpha,\theta),$ it follows that
for each $k,$ $\zeta_{k-1}$ has a $\mathrm{Gamma}(\theta/\alpha+k-1)$ distribution, and the $(R_{k})$ and $(W_{k})$ are collections of independent beta distributed variables where one can set,
$R_{k}:=\beta_{\theta+(k-1)\alpha,1}$
and
$$
1-W_{k}:=\beta^{(k-1)}_{1-\alpha,\alpha}[1-\beta_{\theta+(k-1)\alpha,1}]=\beta_{1-\alpha,\theta+k\alpha}
$$
When $\alpha=0,$ $W_{k}=R_{k}:=\beta_{\theta,1}.$ Hence recovering the stick-breaking representations for $P_{\alpha,\theta}.$
\end{cor}
\subsection{$\mathrm{EPG}(\alpha,\zeta)$-a (stick/simple bridge) representation free of $\tau_{\alpha}$}
Effectively Theorem~\ref{TheoremPGW} combined with Proposition~\ref{W2density} establishes a stick-breaking representation for the 
$\mathrm{EPG}(\alpha,\zeta)$ as follows; first use Theorem~\ref{TheoremPGW} to obtain the explicit representation for $Q_{\alpha,\zeta},$ obtaining 
$(\tilde{W}_{2},\tilde{W}_{3},\ldots),$ then conditioning on $T:=\tau_{\alpha}(\zeta)/\zeta^{1/\alpha}$ to obtain $\tilde{W}_{1}.$ Note in the case of $\mathrm{PD}(\alpha,\theta),$ $T$ is independent of $\tilde{W}_{1}.$ However in general for $\tilde{W}_{1},$ without gamma randomization,  one cannot avoid dealing directly with $\tau_{\alpha},$ and from a practical point of view it is not obvious how  to best negotiate the dependence between this variable and the simple representations given in 
Theorem~\ref{TheoremPGW}. Here again using Theorem~\ref{TheoremPGW}, and our simple bridge representation,  we offer a tractable representation of an $\mathrm{EPG}(\alpha,\zeta)$ bridge that does not involve the direct appearance of $\tau_{\alpha}.$
\begin{prop}
Let $(\tilde{P}_{k})$ denote the stick-breaking sequence obtained by size-biased sampling from $(P_{i})\sim\mathrm{PG}(\alpha,\varepsilon_{\alpha}+\zeta),$
Then applying  Theorem~\ref{TheoremPGW}, the sequence can be represented as $\tilde{P}_{k}=(1-{W}_{k})\prod_{l=1}^{k-1}W_{l},$ where for each $k,$ 
$$
1-W_{k}=\beta^{(k)}_{1-\alpha,\alpha}[1-R_{\alpha,k}]
$$
where $R_{\alpha,k}={[(\zeta_{k-1}+\varepsilon_{\alpha})/(\zeta_{k}+\varepsilon_{\alpha})]}^{1/\alpha}.$ Then it follows by construction that an
$\mathrm{EPG}(\alpha,\zeta)$ bridge has the representation
$$
F_{\alpha,\zeta}(y)=\sum_{k=1}^{\infty}\tilde{P}_{k}\indic_{\{U'_{k}\leq \lambda_{\alpha,\zeta}(y)\}}
$$
where $\lambda_{\alpha,\zeta}$ is the simple bridge with $q_{\alpha,\zeta}:=q_{\alpha,1}=
\zeta/(\zeta+\varepsilon_{\alpha}).$ When $\zeta=\gamma_{(\theta+\alpha)/\alpha},$ $q_{\alpha,\zeta}=
\beta_{(\frac{\theta+\alpha}{\alpha},\frac{1-\alpha}{\alpha})}$ independent of $(R_{\alpha,k}),$ which are independent 
$\beta_{1+\theta+(k-1)\alpha,1}$ variables.
\end{prop}

\section{An interpretation of $(R_{k})$ via the jumps of an independent stable subordinator}\label{CanonicalRep}
We have shown that a $\mathrm{PG}(\alpha,\zeta)$ bridge, $Q_{\alpha,\zeta}$  can be constructed from an iid set of
$\mathrm{Beta}(1-\alpha,\alpha)$ variables
$(\beta^{(k-1)}_{1-\alpha,\alpha})$ and independent of this $(R_{k}),$
$$
R_{k}=
({\zeta_{k-1}}/{\zeta_{k}})^{\frac{1}{\alpha}}
$$
which has practical implications both in terms of ease of use and modelling. Similarly we use these variables to circumvent the appearance of $\tau_{\alpha}$ in the formal representation of the stick-breaking representation for the $\mathrm{EPG}(\alpha,\zeta)$ case. Leading again to highly tractable representations of the bridge, $F_{\alpha,\zeta}.$ The $(R_{k}),$
 also have interpretations that may be deduced from variables, having the same notation, described in Pitman and Yor~\cite{PY97}.  In particular, for $\theta\ge 0,$ the case where the $R_{k}$ are independent $\mathrm{Beta}(\theta+(k-1)\alpha,1)$ variables,
corresponding to the $\mathrm{PG}(\alpha,\gamma_{\theta/\alpha})=\mathrm{PD}(\alpha,\theta)$ case.
This provides a concrete relation to the ranked jumps of a stable subordinator $(\Delta_{i})$ and the corresponding $(P_{i})$ having a $\mathrm{PD}(\alpha,0)$ distribution which we now describe.
\begin{prop} For each integer $n,$ let $(R_{1},\ldots, R_{n})$ denote the random vector as defined in Lemma~\ref{Rkdist}.
Let $(\Delta_{i})$ denote the sequence of ranked jumps of a stable subordinator defined by the L\'evy density $\rho_{\alpha}(s)=\alpha s^{-\alpha-1},$ for $s>0,$ and let $(P_{i}:=\Delta_{i}/T)$ denote the corresponding ranked sequence of probabilities following a $\mathrm{PD}(\alpha,0)$ law.
Then conditional on $\zeta,$ the vector $(R_{i}):=(R_{1},R_{2},\ldots)$ satisfies the following distributional relationship.
$$
\mathcal{L}((R_{i}))|\zeta)=\mathcal{L}(\left(\frac{\Delta_{i+1}}{\Delta_{i}}\right)|\Delta^{-\alpha}_{1}:=\zeta)
=\mathcal{L}(\left(\frac{P_{i+1}}{P_{i}}\right)|\Delta^{-\alpha}_{1}:=\zeta)
$$
This correspondence follows from the equivalence of the following conditional joint distribution,
$$
\mathcal{L}(\Delta^{-\alpha}_{2},\Delta^{-\alpha}_{3},\ldots|\Delta^{-\alpha}_{1}:=\zeta)=\mathcal{L}(\zeta_{1},\zeta_{2},\ldots|\zeta)
$$
with for $k=1,2,\ldots$
$$
\Delta^{-\alpha}_{k}\overset{d}=\zeta_{k-1}=\sum_{l=1}^{k-1}\mathrm{e}_{l}+\zeta.
$$
\end{prop}
\begin{proof}
As in \cite{PY97} let $(X_{k}),$ satisfying $0<X_{1}<X_{2}<\cdots,$ denote the points of a homogeneous Poisson process on $(0,\infty).$ Where, we can set
$X_{1}=\mathrm{e}_{0}\overset{d}=\gamma_{1},$ and for each $k,$ $X_{k}=\sum_{l=1}^{k}\mathrm{e}_{l-1}.$
The result can be read directly from \cite[Proposition 10, (iii)]{PY97}. From there set $X_{n}=\Delta^{-\alpha}_{n},$ use their equation (8), and condition on
$\Delta^{-\alpha}_{1}=\zeta.$
\end{proof}

\subsection{Closing comments}
Notice there is the relation,
$$
\Delta_{k+1}/\Delta_{1}=\prod_{l=1}^{k}R_{l}.
$$
The fact that $Q_{\alpha.\zeta}$ is a function of $((\beta^{(k)}_{1-\alpha,\alpha})),(R_{k}))$ allows one to construct a random process
$\Sigma_{\alpha,\zeta}:=(\Sigma_{\alpha,\zeta}(t)=\sum_{k=1}^{\infty}\prod_{l=1}^{k}R_{l}\indic_{\{\tau_{k}\leq t\}}, t\ge0)$  on the same space. Where
$(\tau_{j})$ are the points of a homogeneous Poisson process on $(0,\infty)$ and conditional on $\zeta,$ equivalently $\Delta_{1},$  $(\Delta_{l+1}/\Delta_{1})$ are the jumps of a subordinator with L\'evy density
$\alpha\zeta u^{-\alpha-1}\indic_{\{u\leq 1\}},$ as described in \cite[Lemma 24 (iii)]{PY97}.  Making the transformation $s=-\ln(1-u)$ leads to a \emph{multiplicatiive subordinator} ${\mbox e}^{-Z_{\alpha,\theta}(t)},$ where $1-{\mbox e}^{-Z_{\alpha,\theta}(t)},$ is a random distribution function.  In general such types of processes arise in
Bayesian Non-parametric statistics \cite{Doksum,JamesNTR} where they are called Neutral to the Right~(NTR) processes, they also play the central role in the theory of regenerative compositions~\cite{GnedinPitman2}, arise in the theory of fragmentation/coagulation~\cite{Pit99,BerFrag}, and most recently are referred to as \emph{fragmenters} in \cite{PitmanWinkel2}.  Restricting $t$ to $[0,1],$ one can see that such processes arise in a Machine Learning context related to the construction of \emph{Indian Buffet} processes, see~\cite{Broderick1,Broderick2, GriffithsZ, TehG}.
Hence the pair $(Q_{\alpha,\zeta}, \Sigma_{\alpha,\zeta})$ allows one to construct more intricate models for applications, which we explore in a subsequent work with additional co-authors.  This also helps us provide an answer to follwing question that is relevant to applications. That is, can we identify applications where one should (not merely prefer to) use an $\mathrm{EPG}(\alpha,\zeta)$ model instead of the $\mathrm{PD}(\alpha,\theta)$ model?  We will demonstrate elsewhere that there are applications where one should choose $\zeta$ to have heavy tails,
which excludes the $\mathrm{PD}(\alpha,\theta)$ distribution where $\zeta$ is gamma distributed. This is in the same spirit as the important observation in Goldwater, Griffiths and Johnson~\cite{Goldwater} which pointed out that the Dirichlet process was completely inappropriate for modelling some aspects of a language model which are known to have power-law behavior. Moreover they proposed to use the Pitman-Yor process, which has the property that $K_{n}$ is of order $n^{\alpha}.$  This insight, along with \cite{IJ2001,IJ2003}, surely led to a substantial increase in the usage of these processes in a wider range of more applications.

\end{document}